\documentclass[10pt,oneside,letter,reqno]{amsart}

\usepackage[T1]{fontenc}
\usepackage[utf8]{inputenc}
\usepackage[british]{babel}
\usepackage{amsmath, amsthm, amssymb, mathtools, mathrsfs}
\usepackage{slashed}
\usepackage{stmaryrd}
\usepackage{tikz-cd}
\usepackage[colorlinks=true,allcolors=gray,hyperindex,breaklinks]{hyperref}
\usepackage{enumitem}

\calclayout

\numberwithin{equation}{section}
\theoremstyle{plain}

\newtheorem{thm}{Theorem}[section]
\newtheorem{lem}[thm]{Lemma}
\newtheorem{prop}[thm]{Proposition}
\newtheorem{cor}[thm]{Corollary}
\newtheorem{step}{Step}

\theoremstyle{definition}
\newtheorem{exmp}[thm]{Example}
\newtheorem{defn}[thm]{Definition}

\newtheorem{rem}[thm]{Remark}

\newcommand{\Z}{\mathbb{Z}}
\newcommand{\R}{\mathbb{R}}
\newcommand{\C}{\mathbb{C}} 
\renewcommand{\H}{\mathbb{H}}
\renewcommand{\P}{\mathbb{P}}
\newcommand{\CP}{\mathbb{CP}}

\newcommand{\del}{\overline{\partial}}
\renewcommand{\Re}{\operatorname{Re}}
\renewcommand{\Im}{\operatorname{Im}}
\newcommand{\rank}{\operatorname{rank}}
\newcommand{\D}{\slashed{D}}
\newcommand{\Hom}{\operatorname{Hom}}

\newcommand{\SO}{\mathrm{SO}}
\newcommand{\SU}{\mathrm{SU}}
\newcommand{\U}{\mathrm{U}}
\newcommand{\Sp}{\mathrm{Sp}}
\newcommand{\SL}{\mathrm{SL}}

\newcommand{\loc}{\mathrm{loc}}
\newcommand{\hol}{\mathrm{hol}}
\newcommand{\sign}{\mathrm{sign}}
\newcommand{\SW}{\mathrm{SW}}
\newcommand{\vol}{\mathrm{vol}}
\newcommand{\Sym}{\mathrm{Sym}}

\newcommand{\cA}{\mathcal{A}}
\newcommand{\cE}{\mathcal{E}}
\newcommand{\cG}{\mathcal{G}}
\newcommand{\cC}{\mathcal{C}}
\newcommand{\cM}{\mathcal{M}}
\newcommand{\cN}{\mathcal{N}}
\newcommand{\cU}{\mathcal{U}}
\newcommand{\cZ}{\mathcal{Z}}
\newcommand{\cO}{\mathcal{O}}
\newcommand{\cP}{\mathcal{P}}

\newcommand{\su}{\mathfrak{su}}
\newcommand{\Met}{\mathfrak{Met}}
\newcommand{\Ob}{\mathfrak{Ob}}

\newcommand{\bp}{\mathbf{p}}

\newcommand{\sP}{\mathscr{P}}

\newcommand{\defined}[1]{\emph{#1}}

\newcommand\restr[2]{{% we make the whole thing an ordinary symbol
  \left.\kern-\nulldelimiterspace % automatically resize the bar with \right
  #1 % the function
  \vphantom{\big|} % pretend it's a little taller at normal size
  \right|_{#2} % this is the delimiter
  }}

\begin{document}

\title[Seiberg--Witten monopoles with multiple spinors on $S^1\times\Sigma$]{Seiberg--Witten monopoles with multiple spinors on a surface times a circle}
\author{Aleksander Doan}
%\date{}
%\address{Department of Mathematics, Stony Brook University, John S. Toll Drive, Stony Brook, NY 11794, USA}
%\email{aleksander.doan@stonybrook.edu}

\begin{abstract}
The Seiberg--Witten equation with multiple spinors generalises the classical Seiberg--Witten equation in dimension three. In contrast to the classical case, the moduli space of solutions $\cM$ can be non-compact due to the appearance of so-called Fueter sections. In the absence of Fueter sections we define a signed count of points in $\cM$ and show its invariance under small perturbations. We then study the equation on the product of a Riemann surface and a circle, describing $\cM$ in terms of holomorphic data over the surface. We define analytic and algebro-geometric compactifications of $\cM$, and construct a homeomorphism between them.  For a generic choice of circle-invariant parameters of the equation, Fueter sections do not appear and $\cM$ is a compact K\"ahler manifold. After a perturbation it splits into isolated points which can be counted with signs, yielding a number independent of the initial choice of the parameters. We compute this number for surfaces of low genus.
\end{abstract}

\maketitle 

\section{Introduction}

This article addresses the questions of transversality and compactness for the moduli spaces of \emph{Seiberg--Witten multi-monopoles}: solutions of the \emph{Seiberg--Witten equation with multiple spinors} introduced in \cite{bryan-wentworth, haydys-walpuski}. We begin by considering the equation on an arbitrary closed Riemannian spin three-manifold $Y$ but the majority of the paper concerns three-manifolds of the form $Y = S^1 \times \Sigma$ for a closed Riemann surface $\Sigma$.

\subsection{Instantons in higher dimensions}
The motivation for studying generalised Seiberg--Witten equations comes from higher-dimensional gauge theory. We only briefly outline this relationship as it will not be essential for understanding the results presented here. Donaldson, Thomas, and Segal initiated a programme of defining an enumerative invariant of Riemannian $7$--manifolds with holonomy contained in the exceptional Lie group $G_2$ \cite{donaldson-thomas, donaldson-segal}. The putative invariant counts, in the sense of Fredholm differential topology, \emph{$G_2$--instantons}: a class of Yang--Mills connections solving a $7$--dimensional analogue of the anti-self-duality equation. A sequence of $G_2$--instantons can concentrate along a three-dimensional \emph{associative submanifold} $Y$, developing a bubble singularity in a way familiar from the four-dimensional theory \cite{tian}. This phenomenon causes the moduli space of $G_2$--instantons to be non-compact. In particular, for a family $(g_t)_{t\in[0,1]}$ of $G_2$--metrics, the cobordism between the moduli spaces of $G_2$--instantons with respect to $g_0$ and $g_1$, given by the one-parameter moduli space, may fail to be compact. If this is the case, the count of $G_2$--instantons with respect to $g_t$ jumps as $t$ varies, and is not a deformation invariant. Haydys and Walpuski \cite{walpuski,haydys-walpuski,haydys4} proposed to compensate for such jumps by adding the number of Seiberg--Witten multi-monopoles on $Y$, counted with signs. This number itself is not a topological invariant of $Y$; it can jump when the Riemannian metric on $Y$ and other parameters of the equation are deformed. The point is that such jumps should occur exactly when a $G_2$--instanton is created or destroyed along $Y$. It is conjectured that a combination of the two numbers: the count of $G_2$--instantons and that of multi-monopoles summed over all associative submanifolds is invariant under deformations.

The case $Y = S^1 \times \Sigma$ discussed in this article is particularly relevant to $G_2$--manifolds of the form $S^1 \times X$ for a Calabi--Yau three-fold $X$ containing $\Sigma$ as a holomorphic curve. $G_2$--instantons over $S^1 \times X$ correspond to Hermitian--Yang--Mills connections on $X$ and so one expects that there is a relationship between multi-monopoles on $S^1 \times \Sigma$ and Hermitian--Yang--Mills connections on $X$ whose energy is highly concentrated around $\Sigma$. From this perspective, our work can be seen as the first step towards a gauge-theoretic interpretation of local Donaldson--Thomas invariants in algebraic geometry \cite{okounkov, diaconescu}.

\subsection{Multi-monopoles on three-manifolds} 
Let $Y$ be a closed three-manifold and let $E \to Y$ and $L \to Y$ be vector bundles with the structure group  $\SU(n)$ and $\U(1)$ respectively. Fix an $\SU(n)$--connection $B$ on $E$. A spin structure on $Y$ and a choice of a Riemannian metric endow $Y$ with the spinor bundle $S$: a rank two Hermitian bundle with trivial determinant line bundle and a compatible connection. The \emph{Seiberg--Witten equation with $n$ spinors} is the following differential equation for a pair $(A,\Psi)$ of a connection on $L$ and a section of $\mathrm{Hom}(E, S \otimes L)$:
\begin{equation} \label{eqn:nsw0}
\left\{
\begin{array}{l}
\slashed{D}_{AB} \Psi = 0, \\
F_A = \Psi \Psi^* - \frac{1}{2} | \Psi |^2.
\end{array}
\right.
\end{equation} 
Here, $\D_{AB}$ is the Dirac operator on $\Hom(E, S \otimes L)$ twisted by the connections $A$ and $B$, $F_A$ is the curvature two-form, and in the second equation we identify imaginary-valued two-forms with traceless skew-Hermitian endomorphisms of $S$ using the Clifford multiplication.

For $n=1$ this is the standard Seiberg--Witten equation in dimension three. As in the classical setting, one introduces the moduli space of gauge-equivalence classes of solutions which depends on the choice of the parameters of the equation: the Riemannian metric on $Y$, the connection $B$, and a closed two-form used to perturb the equation to guarantee transversality. 
\begin{defn}
Denote by $\sP$ the space of all parameters of the equation as described above and for $\bp \in\sP$ let $\cM(\bp)$ be the corresponding moduli space of solutions to the Seiberg--Witten equation with $n$ spinors; see section \ref{sec:nsw}, in particular Definition \ref{Def_Parameters}, for more details.
\end{defn}
The dependence on $\bp$ does not play an important role for $n=1$.
In that case, for a generic\footnote{By \emph{generic} we mean: from a residual subset of the space of objects in question. A subset of a topological space is residual if it contains a countable intersection of open and dense subsets. Baire’s theorem asserts that a residual subset of a complete metric space is dense. Unless said otherwise, we use the $C^{\infty}$--topology. } choice of $\bp$, the moduli space of irreducible solutions is an oriented, compact, zero-dimensional manifold, that is: a finite collection of points equipped with signs. If $b_1(Y) > 1$, there are no reducible solutions and the moduli spaces for two different choices of $\bp$ are connected through an oriented, compact, one-dimensional cobordism. As a consequence, the signed count of points in $\cM(\bp)$, for any generic $\bp$, is a topological invariant of $Y$ \cite{meng-taubes, lim, chen}. 

A new feature of the equation for $n \geq 2$ is that $\cM(\bp)$ may be non-compact for some choices of $\bp$. Building on work of Taubes \cite{taubes3, taubes2}, Haydys and Walpuski \cite{haydys-walpuski} showed that a sequence of points in $\cM(\bp)$ which does not have any  convergent subsequence can be rescaled so that it converges in an appropriate sense to a \emph{Fueter section}, a section of a fibre bundle over $Y$ satisfying a non-linear analogue of the Dirac equation. We review Haydys and Walpuski's compactness theorem in subsection \ref{subsec:compactness}; Fueter sections are introduced in Definition \ref{defn:fueter}. An important point is that the differential equation obeyed by a Fueter section, just like the Seiberg--Witten equation itself, depends on the parameter $\bp$.
As a result, Fueter sections may or may not exist depending on the choice of $\bp\in\sP$. 

\begin{defn}
Denote by  $\cU \subset \sP$ the set of all parameters $\bp \in \sP$ for which the moduli space $\cM(\bp)$ is compact. Haydys and Walpuski's theorem and Corollary \ref{cor:compactness} imply that $\cU$ contains an open neighbourhood of the set of all $\bp\in\sP$ for which no Fueter sections exist.
\end{defn}

Our first result, proved in section \ref{sec:nsw}, shows that the existence of Fueter sections is the only obstruction to defining a signed count of multi-monopoles.

\begin{thm} \label{thm:moduli}
Let $Y$ be a closed oriented spin three-manifold with $b_1(Y) > 1$. Fix vector bundles $E\to Y$ and $L\to Y$ with the structure group $\SU(n)$ and $\U(1)$ respectively.
\begin{enumerate}
\item There exists a locally constant function $\SW \colon \cU \to \Z$ such that if $\cM(\bp)$ is Zariski smooth with the obstruction bundle $\Ob \to \cM(\bp)$, then $\SW(\bp)$ equals the integral of the Euler class of $\Ob$ over $\cM(\bp)$; see Definitions \ref{defn:zariski}, \ref{defn:obstruction}, and \ref{defn:tau2}.
\item For a generic $\bp \in \cU$ the moduli space $\cM(\bp)$ is a compact, zero-dimensional manifold equipped with a natural orientation described in subsection \ref{subsec:orientations}. In particular, in this case $\SW(\bp)$ equals the signed count of points in $\cM(\bp)$.
\end{enumerate}
\end{thm}

At present, little is known about the set $\cU$ and the existence of Fueter sections. The main difficulty in understanding Fueter sections lies in the fact that they are defined only on the complement of a singular set---which is known to be closed and of Hausdorff dimension at most one \cite{taubes2}---and thus not amenable to standard elliptic theory. Conjecturally, $\cU$ is open and dense in $\sP$ and its complement has codimension one \cite{donaldson-segal, haydys-walpuski}, a prediction strongly supported by recent work of Takahashi \cite{takahashi}. If this is the case, $\cU$ has multiple connected components on which the function $\SW\to\Z$ is constant. These components are separated by codimension one walls on which Fueter sections appear and the value of $\SW$ jumps as we cross one of them. It is exactly this jumping phenomenon that indicates a relationship between the enumerative theories for multi-monopoles and $G_2$--instantons. 

%\begin{rem}
%In the second item of Theorem \ref{thm:moduli} the Riemannian metric can be fixed; it is enough to perturb $B$ and $\eta$ to achieve transversality.
%\end{rem}

In view of Theorem \ref{thm:moduli}, the central problem in the study of multi-monopoles on three-manifolds---one of importance to the applications in higher-dimensional gauge theory---is that of the existence and properties of Fueter sections. We make some progress on this problem by describing the moduli spaces of multi-monopoles and Fueter sections for a particular class of parameters $\bp$.  In particular, we exhibit the first examples of $Y$ and $\bp$ such that
\begin{itemize}
  \item $\bp\in\cU$ and $\cM(\bp)$ is non-empty, compact, and consists of irreducible solutions,
  \item $\bp\notin\cU$ and $\cM(\bp)$ contains sequences of solutions converging to a Fueter section,
  \item there exist Fueter sections whose singular sets are non-empty \footnote{In fact, more is true: these Fueter sections do not arise from everywhere defined harmonic spinors and so are examples of singular harmonic $\Z_2$ spinors as in \cite[Definition 1.1]{doan3}.},
  \item there exist Fueter sections that do not appear in the compactification of $\cM(\bp)$.
\end{itemize} 
  
\begin{rem}
Since the first version of this paper appeared, Theorem \ref{thm:moduli} has been used to study the wall-crossing phenomenon for multi-monopoles and to prove the existence of Fueter sections, for some choice of $\bp$, on all closed spin three-manifolds with $b_1 > 1$ \cite{doan3}.
\end{rem}

\subsection{Gauge theory on Riemann surfaces}
The examples of $Y$ and $\bp$ mentioned in the four bullet points above are constructed by means of dimensional reduction. We consider three-manifolds of the form $Y=S^1\times\Sigma$ equipped with a spin structure induced from a spin structure on $\Sigma$. The bundles $E$ and $L$ are assumed to be pulled-back from bundles on $\Sigma$, and in the space of all parameters of the equation $\sP$ we distinguish the subspace $\sP_{\Sigma}$ consisting of the parameters pulled-back from $\Sigma$; see Definition \ref{Def_ParameterSpaceSigma}.

For $n=1$ the classical Seiberg--Witten equation on $S^1 \times \Sigma$ reduces to the vortex equation on $\Sigma$ and the moduli space of solutions is homeomorphic to the symmetric product of $\Sigma$ \cite{morgan-et.al,mrowka-et.al,munoz}. Similarly, in section \ref{sec:reduction} we prove that for $\bp\in\sP_{\Sigma}$ all irreducible solutions of the Seiberg--Witten equation with multiple spinors are pulled back from solutions of a generalised vortex equation on $\Sigma$. In fact, we prove this for a much broader class of three-dimensional Seiberg--Witten equations associated with representations; see Theorem \ref{thm:invariance}.

In section \ref{sec_holomorphic} we show that solutions of the dimensionally reduced equation correspond to triples of the form $(\mathcal{L}, \alpha, \beta)$ where $\mathcal{L} \to \Sigma$ is a holomorphic line bundle and $\alpha$, $\beta$ are holomorphic sections of certain holomorphic vector bundles over $\Sigma$ twisted by $\mathcal{L}$. This is a version of the Hitchin--Kobayashi correspondence, following earlier work of Bryan and Wentworth \cite{bryan-wentworth}.  We construct the moduli space $\cM_{\hol}(\bp)$ of isomorphism classes of such triples. In section \ref{sec_compactifications} we introduce also its natural complex-geometric compactification $\overline{\cM}_{\hol}(\bp)$ which is a compact complex analytic space containing $\cM_{\hol}$ as a Zariski open dense subset. In parallel, guided by an enhanced version of the Haydys and Walpuski's compactness theorem \cite{doan}, we define also a gauge-theoretic compactification $\overline{\cM}(\bp)$ of the Seiberg--Witten moduli space. In what follows we assume for simplicity that $\rank E = 2$.

\begin{thm} \label{thm:homeo0}
For every $\bp\in\sP_{\Sigma}$ there is a homeomorphism $\overline{ \cM}(\bp) \to \overline{ \cM}_{\hol}(\bp)$ restricting to an isomorphism of real analytic spaces  $\cM(\bp) \to\cM_{\hol}(\bp)$.
\end{thm}

%
%\begin{cor}
%There exist Fueter sections that are not the limits of sequences of Seiberg--Witten multi-monopoles in $\cM$.
%\end{cor}
%
%This stems from the fact $\overline{\cM} \setminus \cM$ comprises only of those Fueter sections satisfying conditions $(1)$, $(2)$, $(3)$ listed above; see Example \ref{exmp:fueterlimiting} for details.

\begin{cor}\label{cor:intro}
The gauge-theoretic compactification $\overline{\cM}(\bp)$ is a compact real analytic space containing $\cM(\bp)$ as a Zariski open dense subset.
\end{cor}

\begin{rem}
It is desirable to define $\overline{\cM}(\bp)$ and prove Corollary \ref{cor:intro} without the assumptions $Y = S^1 \times \Sigma$ and $\bp\in\sP_{\Sigma}$. Even though $\cM(\bp)$ is expected to be compact for a generic choice of $\bp\in\sP$, one needs a good notion of a compactification in order to study generic one-parameter families of moduli spaces and the wall-crossing phenomenon \cite[Section 5]{doan3}.
We hope that the construction of $\overline{\cM}(\bp)$ and the analysis constituting the proof of Theorem \ref{thm:homeo0} offer some guidance in the study of compactness for arbitrary three-manifolds.
\end{rem}

The holomorphic description of Fueter sections allows us to prove that a generic point of $\sP_{\Sigma}$ belongs to the set $\cU$ introduced earlier---this is the main result of sections \ref{sec_fuetersections} and \ref{sec:vortices}.

%  Suppose that the genus $g(\Sigma)$ of $\Sigma$ is at least one and that $L \to Y$ is pulled-back from a degree $d$ line bundle over $\Sigma$---these assumptions are not restrictive as otherwise $\cM$ is generically empty. 
%The next result is a consequence of Theorems \ref{thm:nofueter} and \ref{thm:sigmacount}.

\begin{thm}\label{thm:generic0}
Let $\Sigma$ be a closed spin surface of genus $g(\Sigma)\geq 1$. Suppose that $Y= S^1\times\Sigma$ is equipped with a spin structure induced from the spin structure on $\Sigma$,  and that the bundles $E$ and $L$ are pulled-back from bundles over $\Sigma$. Let $d = \langle c_1(L), [\Sigma] \rangle$ be the degree of $L$.

For a generic choice of $\bp\in\sP_{\Sigma}$ there exist no Fueter sections with respect to $\bp$ and the moduli space $\cM(\bp)$ is a compact K\"ahler manifold of dimension 
\[ \dim_{\C} \cM(\bp) = g(\Sigma)-1 \pm 2d, \]
where the sign depends on $\bp$.
In this case, the signed count of Seiberg--Witten multi-monopoles is, up to a sign, the Euler characteristic of the moduli space:
\[ \SW(\bp) = (-1)^{g(\Sigma)-1} \chi(\cM(\bp)). \]
Moreover, $\SW(\bp)$ does not depend on the choice of a generic $\bp\in\sP_{\Sigma}$.
\end{thm}

\begin{rem}
We should stress that here $\bp$ is generic in $\sP_{\Sigma}$ (parameters pulled-back from $\Sigma$) but not in $\sP$ (all parameters on $Y$). In particular, $\cM(\bp)$ may have positive dimension even though the expected dimension is zero. In other words, the moduli space is Zariski smooth and obstructed as in Theorem \ref{thm:moduli}. This is familiar from the study of gauge-theoretic invariants of complex surfaces \cite[Chapter 10]{donaldson-kronheimer}, \cite[Chapter IV]{friedman-morgan}, \cite{friedman-morgan2}.
\end{rem}

The main idea behind the proof of Theorem \ref{thm:generic0} is to interpret the existence of Fueter sections as a Fredholm problem. One difficulty in doing so stems from the possible non-smoothness of the singular sets  of Fueter sections. However, an argument using a Weitzenb\"{o}ck formula shows that in fact such a  set is a finite collection of circles of the form $S^1 \times \{ \mathrm{point} \}$. We then show that these singularities can be removed---in an appropriate sense---so that every Fueter section gives rise to a globally defined holomorphic object. As a result, the deformation theory of Fueter section is described by a complex-linear Fredholm operator of non-positive index. Using this, we prove that Fueter sections appear in real codimension at least two and can be avoided in a generic one-dimensional family of parameters in $\sP_{\Sigma}$.

%The isomorphism $\cM \cong \cM_{\hol}$ gives an interesting interpretation of the above discussion. Recall that $\cM_{\hol}$ can be defined purely in terms of the complex geometry of $\Sigma$. As a complex manifold, it depends on the complex structure of $\Sigma$, the holomorphic structure on $E \to \Sigma$ induced from the connection $B$, and the degree $d = \deg L$. We will see in Section \ref{sec:vortices} that the biholomorphism type of $\cM_{\hol}$ does change when we vary this data. 
%
%\begin{cor}
%The diffeomorphism type of  $\cM_{\hol}$ is independent on the choice of the complex structure on $\Sigma$ and the holomorphic structure on $E \to \Sigma$, as long as the latter is induced by a generic connection on $E \to \Sigma$. Furthermore, the Euler characteristic  $\chi\left( \cM_{\hol} \right)$ does not change when we replace $d$ by $-d$. 
%\end{cor}
%
%See Proposition \ref{prop:invariance} and Corollary \ref{cor:sign} for details. In classical Seiberg--Witten theory, the duality between $d$ and $-d$ is expressed by the identity
%\[ \chi(\Sym^{d+g-1} \Sigma) = \chi(\Sym^{-d+g-1} \Sigma).\]

Finally, in section \ref{sec:examples} we use complex-geometric methods and classical results on stable vector bundles on Riemann surfaces \cite{atiyah, narasimhan} to describe $\cM(\bp)$ and compute $\SW(\bp)$ in some cases. The generic properties of the moduli spaces are summarised in Theorem \ref{thm:moduliexamples}. We study also some non-generic cases which provide us with examples of Fueter sections and non-compact moduli spaces; see Examples \ref{exmp:fueterlimiting}, \ref{exmp_nongeneric1}, Proposition \ref{prop_genuszero}, and Remark \ref{rem_nongeneric}

\begin{exmp}
  For a genus two surface $\Sigma$ the space of parameters $\sP_{\Sigma}$ contains a copy of $\CP^3$, thought of as the moduli space of semi-stable $\SL(2,\C)$--bundles on $\Sigma$. A point $\bp$ in $\CP^3$ corresponds to a plane $H_{\bp}$ in the dual projective space $(\CP^3)^*$. Let $T^4 / \Z_2$ be the singular Kummer surface in $(\CP^3)^*$. For a generic choice of $\bp\in\CP^3$, the plane $H_{\bp}$ intersects $T^4/\Z_2$ transversely and misses all of the $16$ singular points of $T^4/\Z_2$, so that $(T^4/\Z_2) \cap H_{\bp}$ is a smooth genus three surface. The moduli space $\cM(\bp)$ is then the preimage of this intersection under the double covering $T^4 \to T^4/\Z_2$. It is a genus five Riemann surface and so $\SW(\bp) = 8$. When $\bp$ is chosen so that $H(\bp)$ passes through one of the singular points of $T^4/\Z_2$, the moduli space is necessarily singular and non-compact. 
\end{exmp}

\subsection*{Acknowledgements} 
 The work presented here is part of my PhD thesis at Stony Brook University. I am most grateful to my advisor Simon Donaldson for his guidance, support, and encouragement. Discussions with Thomas Walpuski and Andriy Haydys have contributed greatly to this paper.
I thank also Aliakbar Daemi, Cristian Minoccheri, Vicente Mu\~{n}oz, Benjamin Sibley, David Stapleton, Ryosuke Takahashi, Alex Waldron, Malik Younsi, and two anonymous referees for sharing their insights and helping me improve this article.

My research is supported by the \href{https://sites.duke.edu/scshgap/}{\emph{Simons Collaboration on Special Holonomy in Geometry, Analysis, and Physics}}.

\newpage
\subsection*{Notation and conventions}
\mbox{} \\
 
\begin{tabular}{l c l }
$\Lambda^p Y = \Lambda^p T^*Y$ & & exterior product of the cotangent bundle \\
$\Gamma(Y,V)$ or $\Gamma(V)$ & & sections of a vector bundle $V \to Y$ \\
$\Omega^p(Y,V)$ or $\Omega^p(V)$ & & differential forms with values in $V$ \\
$\cA(Y,V)$ or $\cA(V)$ & & connections on $V$ compatible with the structure group \\
$\mathrm{Cl}(\alpha)$ or $\alpha \ \cdot$ & & Clifford multiplication by a form $\alpha \in \Lambda^*Y$ \\
$\langle \cdot, \cdot \rangle$ & & Euclidean inner product \\
$(\cdot, \cdot )$ & & Hermitian inner product\\
$g(\Sigma)$ & & genus of a surface $\Sigma$\\
$J^d$ & & Jacobian of degree $d$ holomorphic line bundles\\
\end{tabular}
\mbox{} \\

When it is not likely to cause confusion, we will use the same notation for a principal $\U(n)$ or $\SU(n)$ bundle and the associated rank $n$ Hermitian vector bundle. We use the following sign convention for the Clifford multiplication: under the identification of the spinor space of $\R^3$ with the quaternions $\mathbb{H}$, with the complex structure given by right-multiplication by $i$, the action of $e^1, e^2, e^3$ is given by left-multiplication by $i$, $j$, $k$ respectively.

\section{Seiberg--Witten monopoles with multiple spinors}
\label{sec:nsw}
In this section we introduce the Seiberg--Witten equation with multiple spinors, discuss transversality and orientations, and review the compactness theorem of Haydys and Walpuski.
The main results are Propositions \ref{prop:regular} and \ref{prop:regular2} which lead to the proof of Theorem \ref{thm:moduli}.

\subsection{The Seiberg--Witten equation with multiple spinors} 
Let $(Y,g)$ be a compact, connected, oriented Riemannian three-manifold equipped with a spin structure. Let $S \to Y$ be the spinor bundle, $E \to Y$ an $\SU(n)$--bundle and $L \to Y$ a $\U(1)$--bundle. Fix a connection $B \in \cA(E)$ and a closed two-form $\eta \in \Omega^2(Y, i\R)$.

\begin{defn}
The \defined{Seiberg--Witten equation with multiple spinors} with parameters $(g,B,\eta)$\footnote{We do not need to assume $d \eta = 0$ in order to write down equation \eqref{eqn:nsw2} but Lemma \ref{lem:muderivative} below and the Bianchi identity imply that there are no solutions unless $d\eta = 0$.} is the following differential equation for a pair $(A,\Psi) \in \cA(L) \times \Gamma(\Hom(E, S \otimes L))$
\begin{equation} \label{eqn:nsw}
\left\{
\begin{array}{ll}
\D_{AB} \Psi = 0, \\
F_A = \mu(\Psi) + \eta.
\end{array}
\right.
\end{equation}
Here, $F_A \in \Omega^2(i\R)$ is the curvature and  $\mu$ is the quadratic map
\[ \mu(\Psi) := \left( \Psi \Psi^* - \frac{1}{2} | \Psi |^2  \ \mathrm{id} \right), \]
with $\Psi^* \in \Gamma(\Hom(S\otimes L, E))$ being the adjoint of $\Psi$.
Using the natural isomorphism $L \otimes L^* \cong \C$, we consider the composition $\Psi\Psi^*$ as a section of $\mathrm{End}(S)$. 
It actually takes values in the subspace $i \mathfrak{su}(S)$ of trace-free self-adjoint endomorphisms which in the equation we implicitly identify with $i \Lambda^2 Y$ using the Clifford multiplication. 

We will refer to solutions of the equation as \emph{Seiberg--Witten monopoles with multiple spinors} or, following  \cite{bryan-wentworth}, \emph{Seiberg--Witten multi-monopoles}\footnote{The name \emph{multi-monopoles} is often used in reference to solutions of the Bogomolny equation.}.
\end{defn}

The Seiberg--Witten equation with multiple spinors was introduced in \cite{bryan-wentworth} and \cite{haydys-walpuski}. If we trivialise $E$ and represent $\Psi$ by an $n$--tuple $(\Psi_1, \ldots, \Psi_n)$ with $\Psi_i \in \Gamma(S \otimes L)$, then
\[ \mu(\Psi) =  \sum_{i=1}^n \left( \Psi_i \Psi_i^* - \frac{1}{2} | \Psi_i^2  | \  \mathrm{id} \right). \]
In particular, for $n=1$ one recovers the classical three-dimensional Seiberg--Witten equation. 

\begin{defn}
Let $\cG = C^{\infty}(Y, S^1)$ be the group of unitary automorphisms of $L$. An element $u \in \cG$ acts on a pair $(A,\Psi)$ by
\[ u(A, \Psi) := (A - u^{-1}du, u \Psi). \]
$(A,\Psi)$ is called \emph{irreducible} if its $\cG$--stabiliser is trivial, i.e. $\Psi \neq 0$, and \emph{reducible} otherwise. 
\end{defn}

To make equation \eqref{eqn:nsw} elliptic modulo the action of $\cG$ we introduce an additional unknown $f \in C^{\infty}(Y,i\R)$. The modified equation for a triple $(A,\Psi,f)$ reads
\begin{equation} \label{eqn:nsw2}
\left\{
\begin{array}{l}
\D_{AB} \Psi + f \Psi = 0, \\
F_A - \ast d f - \mu(\Psi) - \eta = 0.
\end{array}\right.
\end{equation}
In the next proposition we show that this modification does not produce new solutions. 
For $\Psi, \Phi \in \Hom(E, S \otimes L)$ let
\[ \mu(\Psi, \Phi) := \frac{1}{2} \left\{ \Psi \Phi^* + \Phi \Psi^* - \Re \langle \Psi, \Phi \rangle \ \mathrm{id} \right\}. \]

\begin{lem} \label{lem:muderivative}
Considering $\mu(\Psi, \Phi)$ as an imaginary-valued two-form, we have
\[ \ast d \left( \mu(\Psi, \Phi) \right) = i \Im ( \D_{AB} \Psi, \Phi ) - i \Im ( \Psi,  \D_{AB} \Phi). \]
\end{lem}

This is a straightforward computation; see \cite[Proposition A.4]{doan2}. 

%\begin{proof}
%For every $\alpha \in \Lambda^1 Y \otimes i\R$ we have
%\begin{equation}
%\label{eqn_mu}
% \Re \langle \alpha \cdot \Psi, \Phi \rangle = - \langle \alpha, \ast \mu(\Psi, \Phi) \rangle.
%\end{equation}
%Let $(e^1, e^2, e^3)$ be a local orthonormal frame of $\Lambda^1 Y$ and $\nabla_1, \nabla_2, \nabla_3$ the corresponding partial covariant derivatives of $\nabla_{AB}$. Applying $\nabla_k$ to both sides of \eqref{eqn_mu} with $\alpha = i e^k$ results in
%\[ \begin{split}
%  \Re \left\{ \langle i \nabla_k e^k \cdot \Psi , \Phi \rangle + \langle i e^k \cdot \nabla_k \Psi , \Phi \rangle + \langle i e^k \cdot \Psi, \nabla_k \Phi \rangle \right\} & \\ 
% = - \langle i \nabla_k e^k , \ast \mu(\Psi, \Phi)  \rangle - \langle i e^k, \nabla_k \ast \mu(\Psi, \Phi) \rangle.
% \end{split} \]
%Summing over $k = 1, 2,3$ and cancelling the terms occurring on both sides yields
%\[ 
%\begin{split}
%\Re \left\{ \langle i \D_{AB} \Psi, \Phi \rangle + \langle \Psi, i \D_{AB} \Phi \rangle  \right\}
%&=  - \sum_{k=1}^3 \langle i e^k , \nabla_k \ast \mu(\Psi, \Phi) \rangle \\
%&=  -i \ d^* ( \ast \mu(\Psi, \Phi)) =  i  \ast d \left( \mu(\Psi, \Phi) \right).  \qedhere 
%\end{split}
%\]
%\end{proof}

\begin{prop}\label{prop:irreducible} If $(A, \Psi, f)$ is a solution of \eqref{eqn:nsw2}, then $f$ is constant and $(A,\Psi)$ is a solution of the original equation \eqref{eqn:nsw}. Moreover, if $\Psi \neq 0$, then $f = 0$.
\end{prop}
\begin{proof}
Applying $\ast d$ to both sides of the equation $F_A - \ast df - \mu(\Psi) - \eta = 0$ results in 
\[ \begin{split}
 0 &= i \Im ( \D_{AB} \Psi, \Psi ) - i \Im ( \Psi, \D_{AB} \Psi ) - \Delta f \\
&= -i \Im ( f \Psi, \Psi ) + i \Im ( \Psi, f \Psi )- \Delta f = -\left( 2| \Psi|^2 + \Delta \right)f, 
\end{split} \]
where we have used that $\D_{AB} \Psi = - f \Psi$ and $f$ is purely imaginary. If $\Psi = 0$, then $\Delta f = 0$ and $f$ is constant. If $\Psi \neq 0$, then $\Delta +2 | \Psi |^2$ is invertible and $f=0$.
\end{proof}

\subsection{Deformation theory} \label{subsec:moduli}
We want to study the space of solutions of  \eqref{eqn:nsw}. Although it is naturally equipped with the $C^{\infty}$--topology, it is convenient to use the Sobolev topology instead.

\begin{defn}\label{defn_sobolevspaces}
For $k\geq 2$ we consider the following Sobolev spaces:\\

\begin{tabular}{lcl}
$\cA_k(V)$ & & $W^{k,2}$ connections on a bundle $V \to Y$, \\
$\Gamma_k(V)$ & & $W^{k,2}$ sections of $V$, \\
$\Omega^p_k(V)$ & & $W^{k,2}$ differential forms with values in $V$, \\
$\cC_k = \cA_k(L) \times \Gamma_k(\Hom(E,S\otimes L))$ & & $W^{k,2}$ configurations, \\
$\cC^*_k = \cA_k(L) \times (\Gamma_k(\Hom(E,S\otimes L)) \setminus \{ 0 \})$ & & $W^{k,2}$ irreducible configurations, \\
$\cC_k(U)$ & & $W^{k,2}_{\loc}$ configurations on an open subset $U \subset Y$, \\
$\cG_{k+1}$ & & $W^{k+1,2}$ gauge transformations of $L$.
\end{tabular}
\end{defn}

$\cC_k$ and $\cC_k^*$ are Hilbert manifolds modelled on the Hilbert space $\Omega_k^1(i\R) \oplus \Gamma_k(\Hom(E, S \otimes L))$ and $\cG_{k+1}$ is a Hilbert Lie group modelled on $\Omega^0_{k+1}(i\R)$. 

\begin{prop}\label{prop:slice}
The action of $\cG$ on $\cA(L) \times \Gamma(Y, \Hom(E, S \otimes L))$ extends to a smooth, proper action of $\cG_{k+1}$ on $\cC_k$ with a metrisable, second countable quotient $\mathcal{B}_k := \cC_k / \cG_{k+1}$.  Moreover, at every point of $\cC_k$ there exists a local slice for the action and, as a result, $\mathcal{B}_k^* := \cC_k^* / \cG_{k+1}$ has the structure of a Hilbert manifold.
\end{prop}
For the definition of local slices and the proof see \cite[Sections 4.3-4.5]{morgan}. We only mention that a local slice passing through $(A, \Psi) \in \cC_k$ can be chosen to have the form
\begin{equation} \label{eqn:slice}
 (A,\Psi) + K_{A,\Psi}, 
 \end{equation}
 where $K_{A,\Psi} \subset T_{(A,\Psi)} \cC_k$ is a neighbourhood of zero in $\ker G^*$, where $G$ is the linearised action of the gauge group at $(A,\Psi)$ and $G^*$ is the formal adjoint of $G$ restricted to $\cC_k$. 
 
\begin{defn}
The \emph{moduli space of solutions} $\cM(g,B, \eta)$ is the subspace of $\mathcal{B}_k$ consisting of gauge equivalence classes of pairs $(A, \Psi)$ satisfying  \eqref{eqn:nsw} with parameters $(g,B,\eta)$. We define also the \emph{moduli space of irreducible solutions} $\cM^*(g,B,\eta)  = \cM(g, B,\eta) \cap \mathcal{B}_k^*$. 
\end{defn}

\begin{rem}
We will often write $\sigma = (B,\eta)$ and denote the corresponding moduli spaces by $\cM(g,\sigma)$ and $\cM^*(g,\sigma)$. When $(g,\sigma)$ is clear from the context we write simply $\cM$ and $\cM^*$. 
\end{rem}

$\cM$ and $\cM^*$ are endowed with the subspace topology induced from $\mathcal{B}_k$; hence they are metrisable and second countable. The particular choice of $k \geq 2$ is immaterial: if the perturbation $\sigma$ is of class $W^{l,2}$, then the moduli spaces for different choices of $k \leq l+1$ are naturally homeomorphic. If moreover $l = \infty$, then they are all homeomorphic to the space of smooth solutions modulo smooth gauge transformations.

\begin{prop} \label{prop:regularity}
Suppose that $\sigma = (B,\eta)$ is of class $W^{l,2}$ for some $l \geq 1$. 
\begin{enumerate}
\item For every $(A,\Psi) \in \cC_2$ satisfying \eqref{eqn:nsw} there exists  $u \in \cG_3$ such that $u (A, \Psi) \in \cC_{l+1}$. Two such gauge transformations differ by an element of $\cG_{l+2}$. 
\item If $(A_i, \Psi_i)$ is a convergent sequence of solutions in $\cC_2$, then there is a sequence $u_i \in \cG_3$ such that $u_i(A_i, \Psi_i)$ are in $\cC_{l+1}$ and converge in $\cC_{l+1}$. 
\end{enumerate}
\end{prop}

This is a direct consequence of Hodge theory and elliptic regularity for $\D$ and $d + d^*$. Without loss of generality we fix $k=2$ and drop the subscript $k$, writing $\cC$, $\mathcal{B}$, $\cG$, and so on. The local structure of $\cM$ at a solution $(A,\Psi)$ is encoded in the \emph{deformation complex} 
\[
\begin{tikzcd} 
\Gamma_3(i\R) \ar{r}{G_{A, \Psi}} & \Omega^1_2(i\R) \oplus \Gamma_2(E^* \otimes S \otimes L) \oplus \Omega^0_2( i\R) \ar{r}{S_{A, \Psi}} & \Gamma_1(E^* \otimes S \otimes L) \oplus \Omega^2_1( i \R). 
\end{tikzcd}
\]
Here, $G_{A, \Psi}$ is the linearisation of the gauge group action
\[ G_{A,\Psi}(h) := (-dh, h\Psi, 0), \]
and $S_{A, \Psi}$ is the linearisation of the modified equation \eqref{eqn:nsw2} at  $(A, \Psi, 0)$ 
\begin{equation}\label{eqn:linearised} S_{A,\Psi}(a, \phi, v) := (\D_{AB} \phi + a\cdot \Psi + v\Psi, da - 2\mu(\Psi, \phi) - \ast d v). 
\end{equation}
The complex is well-defined by the Sobolev multiplication theorem. It is elliptic and has index zero. One way to see this is to introduce the formal adjoint
\[ G_{A,\Psi}^*(a, \phi, v) := -d^*a + i \Im ( \Psi, \phi ) \]
and consider the \emph{extended Hessian} $L_{A,\Psi} = S_{A,\Psi}  \oplus G_{A,\Psi}^*$. After rearranging the direct sums and identifying the spaces of two-forms and functions via the Hodge star, $L_{A,\Psi}$ is given by
\[ \begin{tikzcd}
 \Omega^1_2(i\R) \oplus \Omega^0_2(i \R) \oplus \Gamma_2(E^* \otimes S \otimes L) \ar{r}{L_{A,\Psi}}  & \Omega^1_1(i\R) \oplus \Omega^0_1(i \R) \oplus \Gamma_1(E^* \otimes S \otimes L), 
\end{tikzcd}
\]  
\[
L_{A,\Psi} \left( \begin{array}{c} a \\ v \\ \phi \end{array} \right) =
\left(
\begin{array}{c}
\ast da - dv - 2\ast \mu(\phi, \Psi) \\
- d^*a + i \Im ( \Psi, \phi ) \\
 a \cdot \Psi + v \Psi + \D_{AB}\phi
\end{array} \right).
\]
$L_{A,\Psi}$ is elliptic because up to terms of order zero it agrees with the direct sum of $\D_{AB}$ and the total operator of the de Rham complex. Representing $L_{A,\Psi}$ by the matrix
\[ L_{A,\Psi} = \left(
\begin{array}{ccc}
\ast d & - d & - 2\ast \mu(\Psi, \cdot) \\
- d^*a & 0 &  i \Im ( \Psi, \cdot ) \\
 \mathrm{Cl}( \cdot) \Psi & (\cdot) \Psi & \D_{AB}
\end{array} \right).
\]
we also see that it is self-adjoint, and so the deformation complex is elliptic and has index zero. 

Let $H^0_{A,\Psi}$, $H^1_{A,\Psi}$, and $H^2_{A,\Psi}$ be the homology groups of the deformation complex at $(A,\Psi)$. They are finite dimensional vector spaces. If the solution is irreducible, then $H^0_{A,\Psi} = 0$ and, since $L_{A,\Psi}$ is self-adjoint, $H^1_{A,\Psi} = \ker L_{A,\Psi}$ and $H^2_{A,\Psi} = \mathrm{coker} L_{A,\Psi}$ are naturally isomorphic. We will later need an explicit description of these groups. 

\begin{lem} \label{lem:h1} $H^1_{A,\Psi}$ consists of pairs $(a,\phi)$ solving
\[ 
\left\{
\begin{array}{l}
- d^*a + i \Im ( \Psi, \phi ) = 0, \\
\D_{AB} \phi + a \cdot \Psi = 0, \\
da - 2 \mu(\Psi, \phi) = 0. 
\end{array}
\right. \]
\end{lem}
The proof is similar to that of Proposition \ref{prop:irreducible}.

%\begin{proof}
%Since $H^1_{A,\Psi}$ is the kernel of $S_{A,\Psi} \oplus G_{A,\Psi}^*$, the content of the lemma is that every triple $(a, \phi, v)$ in the kernel of  $S_{A,\Psi}$ must satisfy $v\Psi = 0$ and $dv = 0$.  The proof is similar to that of Proposition \ref{prop:irreducible}. Suppose that $S_{A,\Psi}(a,\phi,v) = 0$. Applying $\ast d$ to
%\[ da - 2 \mu(\Psi, \phi) - \ast dv = 0 \]
%results in
%\[  i \Im \langle \Psi, \D_{AB} \phi \rangle - \Delta v=0, \]
%where we have used Lemma \ref{lem:muderivative} and condition $\D_{AB} \Psi = 0$ to compute $\ast d \mu(\Psi, \phi)$. On the other hand, we have
%\[ \D_{AB} \phi = - a \cdot \Psi -  v \Psi. \]
%Plugging it into the previous equation, we obtain
%\[ - i \Im \langle \Psi, a\cdot \Psi \rangle - i \Im \langle \Psi, v \Psi \rangle - \Delta v = 0. \] 
%Now under the Clifford multiplication $a$ corresponds to a matrix in $i \mathfrak{su}(2)$, which in particular is self-adjoint. Therefore, $\langle \Psi, a \cdot \Psi \rangle$ is a real number. On the other hand, $v$ is imaginary and so we get
%\[ - ( \Delta + | \Psi |^2) v = 0, \]
%which implies that either $v = 0$ or $\Psi = 0$ and $v$ is constant.
%\end{proof}

\begin{defn} 
We call an irreducible solution $(A,\Psi)$ \emph{regular} if $H^1_{A,\Psi} = H^2_{A,\Psi} = 0$. 
\end{defn}

\begin{rem} \label{rem:regular}
By elliptic regularity for $L_{A,\Psi}$, the property of being a regular solution does not depend on the chosen Sobolev setup (compare with \cite[Proposition 3.1.10]{mcduff-salamon}). 
\end{rem}

The elliptic complex gives rise to a Kuranishi model for the moduli space.

\begin{prop}\label{prop:kuranishi}
For an irreducible solution $(A, \Psi)$ there is a smooth map 
\[\kappa \colon H^1_{A, \Psi} \to H^2_{A, \Psi}\]
 such that $\kappa(0) = 0$, $d\kappa(0) = 0$, and a neighbourhood of $(A,\Psi)$ in $\cM^*$ is homeomorphic to a neighbourhood of $0$ in $\kappa^{-1}(0)$. 
\end{prop}

\begin{proof}
The proof is standard but we need it for future reference. 
Let $U$ be a local slice passing through $(A,\Psi)$ as in \eqref{eqn:slice}. Let $V = \cC \times \Omega^0_2(i\R)$, $W = \Gamma_1(E^* \times S \times L) \times \Omega^2_1(i \R)$, and $\mathcal{F} \colon V \to W$ be given by the left-hand side of \eqref{eqn:nsw2}. The restriction of $\mathcal{F}$ to $U \times \Omega^0_2(i\R)$, which we denote by $f$, is a smooth Fredholm map. By the implicit function theorem, there is a neighbourhood of $(A,\Psi)$ in $U \times \Omega^0_2(i\R)$ and local charts in which $f$ takes the form
\[ f \colon E \times \ker d f_{A,\Psi,0} \to E \times \mathrm{coker} \ d f_{A,\Psi,0}, \]
\[ f(e, x) = (e, K(e,x)) \]
for a Banach space $E$ and a smooth map $K \colon  \ker d f_{A,\Psi,0}  \to  \mathrm{coker} \ d f_{A,\Psi,0}$. The point $(A,\Psi)$ corresponds to $(0,0)$ in $E \times \ker d f_{A,\Psi,0}$, and we have $K(0,0) = (0,0)$ and $dK(0,0) = 0$. Let $\kappa(x) = K(0,x)$. Then the zero set $f^{-1}(0)$ is locally homeomorphic to $\kappa^{-1}(0)$. This proves the proposition because $H^1_{A,\Psi} \cong \ker d f_{A,\Psi,0} $ and $H^2_{A,\Psi} \cong \mathrm{coker} \ df_{A,\Psi,0}$.
\end{proof}

\begin{rem} \label{rem:analytic}
The existence of local Kuranishi models allows us to equip $\cM^*$ with the structure of a real analytic space; since $\mathcal{F}$ is real analytic, we can choose the Kuranishi map $\kappa$ to be analytic as well. Thus, $\cM^*$ is locally homeomorphic to the real analytic set $V = \kappa^{-1}(0)$. As we vary $\kappa$ and $V$, the rings of analytic functions $\cO / I(V)$ glue to a globally defined structure sheaf making $\cM^*$ into a real analytic space; see \cite[Sections 4.1.3-4.1.4]{friedman-morgan}. 
 \end{rem}
 
  In section \ref{sec:reduction} we will face a situation in which the moduli space $\cM^*$ is not regular, but it still has the structure of a smooth manifold. 
  
\begin{defn} \label{defn:zariski}
The moduli space $\cM$ is said to be \emph{Zariski smooth} if $\cM = \cM^*$ and for every $[A,\Psi] \in \cM$ there is a local Kuranishi model with the map $\kappa \colon H^1_{A,\Psi} \to H^2_{A,\Psi}$ being zero. 
\end{defn}

For a Zariski smooth moduli space the homology group $H^1_{A,\Psi}$ is the tangent space to $\cM$ at $[A,\Psi]$. The groups $H^2_{A,\Psi}$ form a vector bundle over the moduli space. To see this, consider the trivial bundles over $\cC^*$
\begin{align*}
 \mathcal{V}_0 & := \cC^* \times (\Gamma_2(E^* \otimes S \otimes L) \oplus \Omega^2_2(i\R)) \\
 \mathcal{W}_0 & := \cC^* \times (\Omega^1_1(i \R) \oplus \Gamma_1(E^* \otimes S \otimes L) \oplus \Omega^0_1(i \R)) 
\end{align*}
and the section $s_0 \in \Gamma(\Hom(\mathcal{V}_0, \mathcal{W}_0))$  which at a point $(A,\Psi)$ is given by $S^*_{A,\Psi}$. The triple $(\mathcal{V}_0, \mathcal{W}_0, s_0)$ is $\cG$-equivariant and descends to a triple $(\mathcal{V}, \mathcal{W}, s)$ of vector bundles $\mathcal{V}$ and $\mathcal{W}$ over $\mathcal{B}^*$ and a section $s \in \Gamma(\Hom ( \mathcal{V}, \mathcal{W}))$. By construction, $\ker s([A,\Psi]) \cong H^2_{A,\Psi}$. 
 
\begin{defn} \label{defn:obstruction}
The \emph{obstruction bundle} is the subspace of $\mathcal{V}$ given by
\[ \Ob := \left\{ ([A,\Psi], v) \ | \ [A,\Psi] \in \cM^*, \ v \in \ker s([A,\Psi])\right\} \]
together with the map $\Ob \to \cM^*$ induced from the bundle projection $\mathcal{V} \to \mathcal{B}^*$. 
\end{defn}

\begin{lem} \label{lem:zariski}
If $\cM$ is Zariski smooth, then it is a submanifold of $\mathcal{B}^*$ and $\Ob \to \cM$ is a vector bundle isomorphic to the tangent bundle $T \cM$. 
\end{lem}

More precisely, $\cM$ is a disjoint union of smooth manifolds of possibly different dimensions. If $C$ is a connected component containing a point $[A,\Psi]$, then $C$ is a smooth submanifold of $\mathcal{B}^*$ of dimension $\dim H^1_{A,\Psi}$ and the restriction of $\Ob \to \cM$  to $C$ is a smooth vector bundle isomorphic to $TC$ as unoriented real vector bundles. The proof is standard and we omit it.

\subsection{Perturbations and transversality}\label{subsec:transversality}
The moduli space $\cM(g, \sigma)$ depends on the choice of a Riemannian metric $g$ and a parameter $\sigma = (B,\eta)$. In this subsection we fix $g$ and study the properties of $\cM(g,\sigma)$ for a generic choice of $\sigma$.

\begin{defn}
\label{Def_Parameters}
We introduce the following Fr\'echet manifolds of $C^\infty$ parameters \\

\begin{tabular}{lcl}
$\Met$ & & the space of Riemannian metrics on $Y$, \\
$\cZ$  & & the space of closed imaginary-valued two-forms, \\
$\cP := \cA(E) \times \cZ$ & & the \emph{space of perturbations},
\end{tabular}\\

and their Sobolev completions $\cZ_k$ and $\cP_k$ with respect to the $W^{k,2}$-topology.
The \emph{space of parameters of the equation}, appearing in the introduction and Theorem \ref{thm:moduli}, is 
\begin{equation*}
\sP := \Met\times\cP.
\end{equation*} However, typically we will fix a Riemannian metric and perturb only $\sigma\in\cP$. 
\end{defn}

Recall that  a subset of a topological space is called \emph{residual} if it contains a countable intersection of open and dense subsets. Baire’s theorem asserts that a residual subset of a complete metric space is dense.

\begin{prop} \label{prop:transversality}
For every $g \in \Met$, the set
\[
  \{  \sigma \in \cP \ | \ \textnormal{all solutions in } \cM^*(g, \sigma)\textnormal{ are regular} \}
\]
 is residual in $\cP$.
\end{prop}

\begin{proof}
We first prove the corresponding statement for $\cP_k$ in place of $\cP$, with solutions in  the configuration space $\cC_{k+1}^*$. Consider the smooth map
\begin{align*}
\mathcal{S} \colon \cP_k \times \cC_{k+1}^* \times \Omega^0_{k+1}(i\R) &\longrightarrow \Gamma_k(E^* \otimes S \otimes L) \times \Omega^2_k(i \R) \\
 \mathcal{S}(B, \eta, A, \Psi, f) &= \left( \D_{A B} \Psi + f \Psi, F_A - \ast d f- \mu(\Psi) - \eta \right). 
 \end{align*}
We will show that $\mathcal{S}$ is a submersion at all irreducible solutions.  
Let $(A,\Psi) \in \cC_{k+1}^*$ be such that $\mathcal{S}(B,\eta, A, \Psi, 0) = 0$. The derivative at $(B, \eta, A, \Psi, 0)$ is 
\[ d \mathcal{S}(b, \xi, a, \phi, v) = S_{A,\Psi} (a, \phi, v) + (b \cdot \Psi, -\xi). \]
Recall that $S_{A,\Psi}$ is the linearisation operator \eqref{eqn:linearised} introduced in the previous subsection. Since  $S_{A,\Psi} \oplus G^*_{A,\Psi}$ is Fredholm, so is $d \mathcal{S} \oplus G^*_{A,\Psi}$, and the image of $d \mathcal{S}$ is closed and of finite codimension. Suppose that $(\Xi, \omega)$ is $L^2$--orthogonal to the image of $d \mathcal{S}$. We will prove $\omega = 0$. 
First, observe that $(\Xi,\omega)$ satisfies 
\begin{align*}
- i\Im ( \Psi, \Xi )+ \ast d \omega &= 0, \\
\D_{AB} \Xi - 2 \omega \cdot \Psi &= 0, \\
d^* \omega + \mu(\Psi,\Xi) &= 0.
\end{align*}
The first equation is a consequence of varying $a$ alone, the second of varying $\phi$ alone, and the third of varying $v$ alone.
On the other hand, $\omega$ is $L^2$--orthogonal to $\cZ_k$, as a result of varying $\xi \in \cZ_k$ alone,  so in particular $d^* \omega = 0$. Applying $\ast d^*$ to $\mu(\Psi,\Xi) = 0$ and using \cite[Proposition A.4]{doan2}, we obtain
\begin{align*}
0 &= \ast d^*\mu(\Psi,\Xi) = \mu(\Psi, \D_{AB}\Xi) - \ast \frac{i}{2}  \Im(\nabla_{AB} \Psi, \Xi) - \ast \frac{i}{2} i \Im (\Psi, \nabla_{AB} \Xi) \\
&= \mu(\Psi, 2\omega \cdot \Psi) + d^* d \omega.
\end{align*}
Taking the $L^2$-inner product with $\omega$ gives us
\[
   2 \| \omega \cdot \Psi \|_{L^2}^2 + \| d \omega \|_{L^2}^2 = 0,
\]
so $d \omega = 0$. We conclude that $\omega = 0$, $\D_{AB} \Xi = 0$, and $\Im (\Psi, \Xi )= 0$.  Next, we prove $\Xi = 0$. Suppose by contradiction that $\Xi$ is not identically zero. By the unique continuation theorem for harmonic spinors \cite[Corollary 3]{bar}, the set $\{ | \Xi | > 0  \}$ is open and dense in $Y$. The same is true for $\Psi$ and so there exists $x \in Y$ such that $| \Psi(x) | > 0$ and $| \Xi(x) | > 0$. On the other hand, for all $a \in \Omega^1_{k+1}(i \R)$ and $b \in \Omega^1_k( \mathfrak{su}(E) )$ 
\[ 0 = \langle d \mathcal{S}( b, 0, a, 0 , 0 ) , \Xi \rangle_{L^2} = \Re ((a+b) \cdot \Psi , \Xi )_{L^2}. \]
However, we can find $a(x) \in \Lambda^1_x \otimes i \R$ and $b(x) \in \Lambda^1_x \otimes \mathfrak{su}(E_x)$ such that
\[ \Re ( (a(x) + b(x)) \cdot \Psi(x), \Xi(x))> 0. \]
This is an elementary fact of linear algebra; it is obvious when $\rank E =1$ and for $\rank E \geq 2$ see Lemma \ref{lem:transversality} below. We extend $a(x)$, $b(x)$  to differential forms $a$ and $b$ satisfying
\[ \Re ( (a+b) \cdot \Psi , \Xi )_{L^2} > 0. \]
This is a contradiction; we conclude that $\Xi = 0$.

We have shown that $\mathcal{S}$ is a submersion at all irreducible solutions. By the implicit function theorem, $\mathcal{S}^{-1}(0)$ is a submanifold of $\cP_k \times \cC_{k+1}^* \times \Omega^0_{k+1}(i\R)$. The gauge group $\cG_{k+2}$ acts freely on $\mathcal{S}^{-1}(0)$. Denote the quotient space by $\mathcal{X}$. The projection on $\cP_k$ induces a smooth map $\pi \colon \mathcal{X} \to \cP_k$ whose fibre over $\sigma \in \cP_k$ is $\cM^*(g,\sigma)$. The derivative $d\pi$ is Fredholm and at a point $(A, \Psi) \in \cM^*(g,\sigma)$ there is a natural identification $\mathrm{coker} \ d\pi_{A,\Psi} = H^{2}_{A,\Psi}$. By the Sard-Smale theorem, the set of regular values of $\pi$ is residual in $\cP_k$. For every such regular value $\sigma$ all solutions in $\cM^*(g,\sigma)$ are regular. This proves the statement for $\cP_k$. 

The final part of the proof is to replace $\cP_k$ by $\cP$. We follow an argument due to Taubes used in the context of pseudo-holomorphic curves \cite[Theorem 3.1.5]{mcduff-salamon}. For $N > 0$ define $\cU_N \subset \cP$ as the set of all $\sigma$ such that all solutions $(A,\Psi)$ in $\cM^*(g,\sigma)$ satisfying 
\begin{equation} \label{eqn:transversality1}
 \frac{1}{N} \leq \| \Psi \|_{L^{\infty}} \leq N 
 \end{equation}
are regular. We need to show that the intersection $\bigcap_{N \geq 1} \cU_N$ is residual. First, we show that $\cU_N$ is open. Let $\sigma_i$ be a sequence in $\cP \setminus \cU_N$ converging to $\sigma$ in $C^{\infty}$. By the definition of $\cU_N$, there is a sequence of solutions $(A_i, \Psi_i)$ solving equations \eqref{eqn:nsw2} with parameters $\sigma_i$, satisfying \eqref{eqn:transversality1} and $H^2_{A_i, \Psi_i} \neq 0$. By the first part of Theorem \ref{thm:compactness} from the next subsection, after passing to a subsequence $(A_i,\Psi_i)$ converges smoothly modulo gauge. The limit $(A,\Psi)$ satisfies condition \eqref{eqn:transversality1} and the equations with parameter $\sigma$. Since $L_{A_i, \Psi_i}$ converges to $L_{A,\Psi}$,  the latter operator is not surjective and $\sigma \notin \cU_N$. We conclude that $\cP \setminus \cU_N$ is closed.
The last step is to show that $\cU_N$ is dense in $\cP$. We use the statement for $\cP_k$ proved earlier. Let $\cU_{k,N}$ be the subspace of $\cP_k$ defined analogously to $\cU_N$. Since $\cU_{k,N}$ contains the set of regular values of $\pi$ discussed earlier, it is dense in $\cP_k$. It is also open in $\cP_k$ by the same argument as before. By Proposition \ref{prop:regularity} and Remark \ref{rem:regular}, $\cU_N = \cU_{k,N} \cap  \cP$ for all $k \geq 2$. Every element $\sigma$ of $\cP$ can be approximated in the $W^{k,2}$--topology by elements of $\cU_{k,N}$. In fact, it can be approximated by elements of the intersection $\cU_{k,N} \cap \cP$ because $\cU_{k,N}$ is open in $\cP_k$ and $\cP$ is dense in $\cP_k$. Therefore, we have shown that $\sigma$ can be approximated by elements in $\cU_N$ in any Sobolev norm. By the usual diagonal argument and the Sobolev embedding theorem, we can find a sequence in $\cU_N$ converging to $\sigma$ in $C^{\infty}$. This shows that $\cU_N$ is dense in $\cP$.
\end{proof}

\begin{lem}
  \label{lem:transversality}
  Let $n \geq 2$ and $V_n = \C^n$ be the standard representation of $\SU(n)$.
  For every pair of non-zero vectors $v, w \in V_2 \otimes V_n$ there exists $b \in \su(2) \otimes \su(n)$ such that 
  \begin{equation*}
    ( bv, w )> 0.
  \end{equation*} 
\end{lem} 

\begin{proof}
  The proof is similar to that of \cite[Theorem 1.5]{anghel}. 
  Let $e_1, \ldots, e_n$ be an orthonormal basis of $V_n$. 
  Write $v$ and $w$ as
  \begin{equation*}
    v = \sum_{i=1}^n v_i \otimes e_i \quad \textnormal{and} \quad
    w = \sum_{i=1}^n w_i \otimes e_i,
  \end{equation*}
  for $v_i, w_i \in V_2$. 
  Likewise, denoting by $\sigma_1, \sigma_2, \sigma_3$ the standard basis of $\su(2)$, we can write $b$ as
  \begin{equation*}
    b = \sum_{k=1}^3 \sigma_k \otimes b_k
  \end{equation*}
  for some $b_k \in \su(n)$, so that
  \begin{align*}
    ( bv, w ) &= \sum_{k=1}^3 \sum_{i,j} ( \sigma_k v_i, w_j ) ( e_i, b_k e_j).
  \end{align*}
  Suppose by contradiction that $( bv, w ) = 0$ for all $b \in \su(2) \otimes \su(n)$. 
  In particular, setting $b_k$ to be elementary off-diagonal anti-Hermitian matrices, we see that for $k=1,2,3$ and $i \neq j$
  \begin{align*}
    ( \sigma_k v_i, w_j ) -  ( \sigma_k v_j, w_i ) &= 0, \\
    i ( \sigma_k v_i, w_j ) + i ( \sigma_k v_j, w_i ) &= 0.
  \end{align*} 
  Hence,
  \begin{equation*}
    ( \sigma_k v_i, w_j ) = 0
  \end{equation*}
  for $k=1,2,3$ and $i\neq j$. 
  Suppose without loss of generality that $v_1 \neq 0$. 
  Then $\sigma_1 v_1, \sigma_2 v_1, \sigma_3 v_1$ are linearly independent over $\R$ and thus span $V_2$ over $\C$. 
  It follows that $w_j = 0$ for $j = 2, 3, \ldots n$. 
  On the other hand, setting $b_k = \mathrm{diag}(1,-1,0,\ldots, 0) \in \su(n)$, we obtain that for $k=1,2,3$ 
  \begin{equation*}
    ( \sigma_k v_1,  w_1 ) = 0,
  \end{equation*}
  which shows that $w_1 = 0$ and so $w = 0$, yielding a contradiction.
\end{proof}

\subsection{Reducible solutions} 
The moduli space $\cM$ might contain reducible solutions at which it develops singularities. In this paper we focus on the favourable case when reducibles can be avoided. As in the classical setting \cite[Lemma 14]{lim}, this is guaranteed by $b_1(Y) > 1$. 

\begin{prop} \label{prop:reducibles} For a fixed $g \in \Met$ the subset of parameters $\sigma \in \cP$ for which $\cM(g, \sigma)$ contains a reducible solution is contained in a closed affine subspace of $\cP$ of codimension $b_1(Y)$.  
\end{prop}

\begin{proof}
If $(A, 0)$ is a reducible solution of equations \eqref{eqn:nsw} with a parameter $\sigma = (B, \eta)$, then $F_A = \eta$ and passing to the de Rham cohomology we obtain
\[ [ \eta ] = - 2\pi i c_1(L) \in H^2(Y, i\R). \]
Consider the affine subspace of $\cZ$ given by
\[ V = \left\{ \eta \in \cZ \ | \ [ \eta] = - 2\pi i c_1(L) \right\}. \]
In other words, $V$ is the preimage of $-2\pi i c_1(L)$ under the projection $\pi \colon \cZ \to H^2(Y, \R)$ associating to each closed form its de Rham class. The map $\pi$ is linear and surjective. It is also continuous because by the Hodge decomposition theorem it is continuous with respect to the $W^{k,2}$--topology for all $k$. Therefore, $V$ is a closed affine subspace of codimension $\dim H^2(Y, \R) = \dim H^1(Y, \R) = b_1$. 
\end{proof}

\subsection{Compactness and Fueter sections}\label{subsec:compactness}
The moduli space of multi-monopoles needs not be compact. 
%Recall that in the classical setting compactness is a consequence of the Weitzenb\"ock formula and algebraic identity 
%\[ \langle \mu(\Psi) \Psi , \Psi \rangle = | \mu(\Psi) |^2. \]
%Both formulae hold in the generalised setting, independently on $n = \mathrm{rk} E$. However, when $n=1$, the map $\mu$ satisfies the additional relation $| \mu(\Psi) | = | \Psi |^2$. This allows us to control the $L^4$ norm of $\Psi$ in terms of the expression $|\mu(\Psi)|^2$ appearing in the Weitzenb\"ock formula. Together with elliptic regularity, this leads to $C^{\infty}$ bounds for solutions. For $n>1$ the map $\mu$ is no longer proper and there is no \emph{a priori} bound for the spinor part of a solution of \eqref{eqn:nsw}. In particular, the condition $\mu(\Psi) = 0$ does not imply the vanishing of $\Psi$. 
Haydys and Walpuski \cite{haydys-walpuski} studied sequences of solutions whose $L^2$--norms diverge to infinity. They proved that in the limit we obtain a harmonic spinor taking values in the zero set $\mu^{-1}(0)$. In general, it is defined only on the complement of a closed, nowhere dense subset of $Y$. 

\begin{defn} \label{defn:fueter}
Let $Z \subsetneq Y$ be a closed, proper subset and $(A, \Psi) \in \cC_k(Y \setminus Z)$ (see Definition \ref{defn_sobolevspaces}). A triple $(A, \Psi, Z)$ is called a \emph{Fueter section}\footnote{For the explanation of this name see \cite[Remark 3.4]{haydys2}, \cite[Appendix A]{haydys-walpuski}, \cite[Section 1.3]{doan2}.} \emph{with singular set} $Z$ if it satisfies 
\begin{enumerate}
  \item $\D_{AB} \Psi = 0$ and $\mu(\Psi) = 0$,
  \item $\int_{Y \setminus Z} |\Psi|^2 =1$ and $\int_{Y\setminus Z} | \nabla_A \Psi |^2 < \infty$,
  \item  $| \Psi |$ extends to a H\"older continuous function on $Y$ such that $Z = | \Psi |^{-1}(0)$.
\end{enumerate}
\end{defn}

\begin{rem}
The definition of a Fueter section depends, through the Dirac equation, on the choice of the metric $g$ and the background connection $B$. We will call $(A,\Psi, Z)$ a \emph{Fueter section with respect to} $(g,B)$ when we want to stress this dependence. 
\end{rem}

There is elliptic regularity theory for Fueter sections \cite[Appendix A]{haydys-walpuski}, which implies that if $B$ is of class $W^{l,2}$, then after a gauge transformation $(A,\Psi)$ is in $\cC_{l+1}(Y \setminus Z)$. Taubes \cite{taubes2} proved that the singular set of a Fueter section has Hausdorff dimension at most one. 

\begin{thm}[Haydys--Walpuski \cite{haydys-walpuski} \footnote{In the original statement  \cite[Theorem 1.5]{haydys-walpuski} the sequence $(g_i, \sigma_i)$ is constant. As has been communicated to me by Thomas Walpuski, the proof can be easily adapted to the more general setting of Theorem \ref{thm:compactness}.
}] \label{thm:compactness}
Let $(g_i, \sigma_i) \in \Met \times \cP$ be a sequence converging to $(g, \sigma)$ and $[A_i, \Psi_i]$ a sequence in $\mathcal{B}$ such that $[A_i, \Psi_i] \in \cM(g_i, \sigma_i)$. 
\begin{enumerate}
\item If $\limsup_{i \to \infty} \| \Psi_i \|_{L^2} < \infty$, then after passing to a subsequence and applying gauge transformations  the sequence $(A_i, \Psi_i)$ converges smoothly to a solution $(A,\Psi)$ representing a point in $\cM(g, \sigma)$. 
\item If $\limsup_{i \to \infty} \| \Psi_i \|_{L^2} = \infty$, then there exists a Fueter section $(A, \Psi, Z)$ with respect to $(g,B)$ such that after passing to a subsequence and applying gauge transformations $A_i \to A$ weakly in $W^{1,2}_{\loc}$ and $\Psi_i / \| \Psi_i \|_{L^2} \to \Psi$ weakly in $W^{2,2}_{\loc}$ over $Y \setminus Z$. 
\end{enumerate}
\end{thm}

\begin{cor}  \label{cor:compactness}
Let $g \in \Met$ and $ \sigma = (B, \eta) \in \cP$.  If there are no Fueter sections with respect to $(g,B)$, then there is an open neighbourhood of $(g,\sigma)$ in $\Met \times \cP$ such that for all $(g',\sigma')$ from this neighbourhood the moduli space $\cM(g', \sigma')$ is compact.
\end{cor}

\begin{proof}
Let $K = \sup\{ \| \Psi \|_{L^2} \ | \ [A,\Psi] \in \cM(g, \sigma)\}$. We have $K < \infty$ by Theorem \ref{thm:compactness} and the assumption that there are no $(g, B)$-Fueter sections. We claim that there is an open neighbourhood of $(g, \sigma)$ such that for all $(g',\sigma')$ from this neighbourhood 
\begin{equation} \label{eqn:compactness}
\sup_{[A,\Psi] \in \cM(g',\sigma')} \| \Psi \|_{L^2} < K + 1,
\end{equation}
and the proposition follows by Theorem \ref{thm:compactness}.  Assume by contradiction that there is a sequence $(g_i, \sigma_i)$ converging to $(g,\sigma)$ and violating \eqref{eqn:compactness}. There is a corresponding sequence $(A_i, \Psi_i)$ representing points in $\cM(g_i, \sigma_i)$ and satisfying $\| \Psi_i \|_{L^2} \geq K+1$. The first alternative of Theorem \ref{thm:compactness} cannot hold because otherwise we would extract a subsequence converging modulo gauge to a solution $(A,\Psi)$ with $[A,\Psi] \in \cM(g,\sigma)$ and $\| \Psi \|_{L^2} \geq K +1 > K$. On the other hand, the second alternative would imply the existence of a $(g, B)$-Fueter section. 
\end{proof}

\subsection{Orientations} 
\label{subsec:orientations}
If $\cM^*$ is regular, one prescribes an orientation to every point in $\cM^*$. We briefly outline this construction as we will need it later. For $(A,\Psi, f) \in \cC \times \Omega^0_2(i\R)$ let $S_{A,\Psi,f}$ be the linearisation of \eqref{eqn:nsw2} at $(A,\Psi, f)$\footnote{Previously we have only dealt with the linearisation at irreducible solutions for which $f = 0$.} and let $L_{A,\Psi, f} = G_{A,\Psi}^* \oplus S_{A,\Psi,f}$. Let $\mathrm{Det} \to \cC^* \times \Omega^0_2(i \R)$ be the determinant line bundle of the family $L_{A,\Psi,f}$; see \cite[Appendix A.2]{mcduff-salamon}. As explained in \cite[Section 2.4]{lim}, $\mathrm{Det}$ descends to a real line bundle $\Lambda \to \mathcal{B}^* \times \Omega^0_2(i \R)$ which is shown to be trivial by considering the $\cG$--equivariant homotopy $\left\{ L_{A, t\Psi, tf} \right\}_{t \in [0,1]}$ connecting $L_{A,\Psi,f}$ with $L_{A,0,0} = (d + d^*) \oplus \D_{AB}$. The kernel and cokernel of $d + d^*$ are naturally identified with $H^0(Y,\R)\oplus H^1(Y,\R)$, while $\D_{AB}$ is complex-linear and has trivial determinant. It follows that the determinant bundle of the family of operators $L_{A,0,0}$ is trivial, and therefore so is $\Lambda$. Denote the global orientation on $\Lambda$ given by this trivialisation by $\mathrm{or}_{\Lambda}$. 

On the other hand, $L_{A,\Psi, f}$ is formally self-adjoint, so $\mathrm{coker} \ L_{A,\Psi, f}$ and $\ker L_{A,\Psi,f}$ are naturally isomorphic. More precisely, for a self-adjoint Fredholm operator $T$ we define the \emph{Knudsen--Mumford isomorphism}  $\det (\ker T) \otimes \det (\mathrm{coker} \ T)^* \to \R$ by
\[ \omega \otimes \eta \mapsto (-1)^{ \frac{ \dim \ker T  ( \dim \ker T + 1) }{ 2}} \eta ( \omega), \]
where $\mathrm{coker} \  T = \ker T^* = \ker T$. These isomorphisms do not yield a global trivialisation of $\mathrm{Det}$ because the dimension of the kernel jumps. However, they induce a natural trivialisation of $\mathrm{Det}$ over every stratum of $\mathcal{B}^* \times \Omega^0_2(i\R)$ over which $\ker L_{A,\Psi, f}$ has constant rank \cite[Appendix B]{bohn}. We denote the induced orientation on $\Lambda$, restricted to every stratum, by $\mathrm{or}_0$.

\begin{defn}
If $\cM^*$ is regular, we define for every $x \in \cM^*$
\[
\sign(x) := \left\{
\begin{array}{ll}
+1 & \textnormal{if } \mathrm{or}_{\Lambda}(x) = \mathrm{or}_0(x) \\
-1 & \textnormal{if } \mathrm{or}_{\Lambda}(x) \neq \mathrm{or}_0(x)
\end{array}
\right.
\] 
%Equivalently,  if $x = [ A, \Psi]$, then $\sign(x)$ is given by the orientation transport along the path $\{ L_{A, t\Psi, 0} \}_{t \in [0,1]}$. For the proof see, for example, \cite[Appendix C]{bohn}. 
\end{defn}

The definition can be extended to the case when  $\cM$ is Zariski smooth since for every connected component $C$ of $\cM$ the dimension of $\ker L_{A,\Psi,f}$ is constant along $C$. 
%Thus, we have two trivialisations of $\restr{\Lambda}{C}$: the global trivialisation $\mathrm{or}_{\Lambda}$ and the trivialisation $\mathrm{or}_{0}$.

\begin{defn}
If $\cM$ is Zariski smooth and $C$ is a connected component of $\cM$, set
\[
\sign(C) := \left\{
\begin{array}{ll}
+1 & \textnormal{if } \mathrm{or}_{\Lambda}(x) = \mathrm{or}_0(x) \\
-1 & \textnormal{if } \mathrm{or}_{\Lambda}(x) \neq \mathrm{or}_0(x)
\end{array}
\right.
\] 
where $x$ is any point in $C$.
%Alternatively, $\sign(C)$ is equal to the orientation transport along the path $\{ L_{A, t\Psi, 0} \}_{t \in [0,1]}$ for any $[A,\Psi] \in C$.
\end{defn}

Alternatively, we have $\Lambda = \det T\cM \otimes ( \det \Ob )^*$, so the trivialisation $\mathrm{or}_{\Lambda}$ induces a \emph{relative orientation} on $\Ob \to \cM$ which orients  the zero set of every transverse section of $\Ob$.

\subsection{Counting solutions} \label{subsec:counting} 
We come to the main point of this section, that is defining the signed count of Seiberg--Witten multi-monopoles. The next two propositions are deduced in the standard way from Proposition \ref{prop:transversality}, Proposition \ref{prop:reducibles}, and the discussion of orientations.

\begin{prop}\label{prop:regular}
For every $g\in\Met$ there is a residual subset $\cP^{reg}(g) \subset \cP$ such that for every $\sigma \in \cP^{reg}(g)$ the moduli space $\cM^*(g, \sigma)$ is an oriented zero-dimensional manifold. If $b_1(Y) > 0$, we may additionally assume that there are no reducible solutions and $\cM^*(g, \sigma) = \cM(g, \sigma)$.
\end{prop}

%\begin{proof}
%Let $\cP^{reg}(g)$ be the set of those $\sigma \in \cP$ for which all solutions in $\cM^*(g,\sigma)$ are regular. By Proposition \ref{prop:transversality} it is of the second Baire category in $\cP$. From the local Kuranishi model described in Proposition \ref{prop:kuranishi} it follows that for all $\sigma \in \cP^{reg}(g)$ the moduli space is a zero-dimensional manifold. A way of assigning orientations is described in the previous subsection. Let $V \subset \cP$ be the set of those $\sigma$ for which $\cM(g,\sigma)$ contains a reducible solution. If $b_1(Y) >0$, then $V$ is nowhere dense by Proposition \ref{prop:reducibles}. Thus, replacing $\cP^{reg}(g)$ by its intersection with $\cP \setminus V$, we can guarantee that for all $\sigma \in \cP^{reg}(g)$ we have $\cM(g,\sigma) = \cM^*(g,\sigma)$.
%\end{proof}

\begin{defn} Let $\sigma_0$ and $\sigma_1$ be perturbations in $\cP$. Denote by $\cP(\sigma_0, \sigma_1)$ the space of smooth paths $\{ \sigma_t \}_{t \in [0,1]}$ of elements in $\cP$ connecting $\sigma_0$ and $\sigma_1$. The subspace topology induced from the space $C^{\infty}([0,1], \cP)$ makes $\cP(\sigma_0, \sigma_1)$ into a Fr\'echet manifold.  
\end{defn}

\begin{prop}\label{prop:regular2}
For every smooth path $\{ g_t \}_{t \in [0,1]}$ in $\Met$ and $\sigma_0 \in \cP^{reg}(g_0)$, $\sigma_1 \in \cP^{reg}(g_1)$ there is a residual subset of  $\cP(\sigma_0, \sigma_1)$ such that for all smooth paths  $\{ \sigma_t \}_{t\in[0,1]}$ from this subset the one-parameter moduli space
\[ \bigcup_{t \in [0,1]} \cM^*(g_t, \sigma_t) \]
is an oriented one-dimensional cobordism between $\cM^*(g_0, \sigma_0)$ and $\cM^*(g_1, \sigma_1)$. If $b_1(Y) > 1$, we may additionally assume that $\cM^*(g_t, \sigma_t) = \cM(g_t, \sigma_t)$ for all $t \in [0,1]$. 
\end{prop}

%\begin{proof}
%First we remark that the spinor bundles for metrics $g_t$ can be identified in a natural way with the spinor bundle $S$ of the initial metric $g_0$  \cite[section 3]{lim}. Then each moduli space $\cM(g_t, \sigma_t)$ can be seen as a subspace of the same Banach manifold $\mathcal{B}$. The statement is proved in the same way as Proposition \ref{prop:regular}. It relies on the fact that a generic path in $\cP$ is transverse to the projection $\pi \colon \mathcal{X} \to \cP$ from the proof of Proposition \ref{prop:transversality}. By Proposition \ref{prop:reducibles} the subspace $V \subset \cP$ of parameters for which the moduli space contains a reducible solution is of codimension $b_1$, thus for $b_1 > 1$ a generic path in $\cP$ is disjoint from $V$.
%\end{proof}

\begin{defn} \label{defn:tau}
If $\cM(g,\sigma)$ is compact and consists of irreducible and regular solutions, set
\[ \SW(g,\sigma) :=  \sum_{x \in \cM(g,\sigma)} \sign(x). \]
\end{defn}

The definition of $\SW$ can be extended to the case when $\cM$ is compact and Zariski smooth.
\begin{defn} \label{defn:tau2}
If $\cM(g,\sigma)$ is compact and Zariski smooth, set
\[ \SW(g,\sigma) := \sum_{C} \sign(C) \chi(C), \]
where the sum is taken over all connected components of $\cM$ and $\chi$ is the Euler characteristic.
\end{defn}

The Gauss--Bonnet theorem and the Poincar\'{e}--Hopf index theorem give us two equivalent descriptions of $\SW$.  

\begin{prop}
Assume that $\cM(g,\sigma)$ is compact and Zariski smooth.
\begin{enumerate}
\item If $s$ is a transverse section of $\Ob \to \cM(g,\sigma)$, then 
 \[ \SW(g,\sigma) = \sum_{x \in s^{-1}(0)} \sign(x), \]
 where $\sign(x)$ is obtained by comparing the orientation on $s^{-1}(0)$ induced fom the relative orientation on $\Ob \to \cM$ with the natural orientation of a point.
\item Suppose that all components of $\cM$ are orientable and orient them arbitrarily. The relative orientation on $\Ob \to \cM$ makes $\Ob$ into an oriented real vector bundle. If $e(\Ob)$ is the Euler class of $\Ob$, then
\[ \SW(g,\sigma) =  \int_{\cM} e( \Ob). \]
\end{enumerate}
\end{prop}

%\begin{thm} \label{thm:count}
%Let $g$ be a Riemannian metric on $Y$ and $\sigma = (B, \eta) \in \cP$. Suppose that $b_1(Y) > 1$ and there exist no $(g,B)$-Fueter sections. Then there is an open neighbourhood $U$ of $(g, \sigma)$ in $\Met \times \cP$ with the following significance.  
%\begin{enumerate}
%\item For all $(g',\sigma') \in U$ the moduli space $\cM(g',\sigma')$ is compact.
%\item The subset $U_s = \{ (g', \sigma') \in U \ | \ \cM(g',\sigma') \ \textnormal{is Zariski smooth} \}$ is dense in $U$.
%\item The number $\SW(g',\sigma')$ does not depend on the choice of $(g',\sigma') \in U_s$.
%\end{enumerate}
%\end{thm}
\begin{proof}[Proof of Theorem \ref{thm:moduli}]
For $(g,\sigma) \in \cU$ such that $\cM(g,\sigma)$ consists of irreducible and regular solutions $\SW(g,\sigma)$ is well-defined as in Definition \ref{defn:tau}. If $(g,\sigma) \in \cU$ is an arbitrary point, define $\SW(g,\sigma)$ to be equal to $\SW(g',\sigma')$ for any $(g',\sigma')$ in the same connected component of $\cU$ and such that $\cM(g',\sigma')$ consists of irreducible and regular solutions. For two such pairs $(g_0, \sigma_0)$ and $(g_1, \sigma_1)$, by Proposition \ref{prop:regular2} there is a path $(g_t, \sigma_t)$ in $\cU$ such that the union
\[ \mathcal{W} := \bigcup_{t \in [0,1]} \cM(g_t, \sigma_t) \]
is an oriented compact cobordism from $\cM(g_0, \sigma_0)$ to $\cM(g_1, \sigma_1)$. It follows that  $\SW(g_0, \sigma_0) = \SW(g_1, \sigma_1)$ and so $\SW(g,\sigma)$ is well-defined for any $(g,\sigma) \in \cU$. It remains to show that if $\cM(g,\sigma)$ is Zariski smooth, then the integer obtained in this way is equal to that from Definition \ref{defn:tau2}. This is a general fact; see \cite[Lemma 3.3]{friedman-morgan2}. 
%Let $\cE \to \mathcal{X}$ is a Hilbert vector bundle over a Hilbert manifold. Let $s_0$ and $s_1$ be two Fredholm sections of $\cE$ which are sufficiently $C^1$-close. Suppose that the zero sets $Z_0 = s^{-1}_0(0)$ and $Z_1 = s_1^{-1}(0)$ are compact, $s_0$ is transverse to zero, and $Z_1$ is Zariski smooth. Let $\Ob \to Z_1$ be the obstruction bundle. Under these assumptions, there is a transverse section $\alpha \in \Gamma(Z_1, \Ob)$ such that $\alpha^{-1}(0)$ is isotopic to $s_0^{-1}(0)$ as oriented submanifolds of $\mathcal{X}$. To apply the lemma to our situation we choose the sections $s_0$ and $s_1$ to be given by the Seiberg--Witten equations \eqref{eqn:nsw2} for two choices of $(g, \sigma)$ such that the $\cM(g,\sigma)$ is irreducible and regular for one of them and Zariski smooth for the other.  
\end{proof}

%In order to use Theorem \ref{thm:moduli} we need to guarantee that there are no Fueter sections for some choice of $(g,\sigma) \in \Met \times \cP$. This assumption is in general hard to verify. In the next two sections we exhibit examples of $(g,\sigma)$ for which this assumption is satisfied.

\section{A dimensional reduction} \label{sec:reduction}
We focus on the case $Y = S^1 \times \Sigma$. Assuming that the parameters of the equation are invariant in the circle direction, we  show that irreducible Seiberg--Witten multi-monopoles on $Y$ are pulled back from configurations on $\Sigma$ obeying a generalised vortex equation. This is analogous to the correspondence between classical Seiberg--Witten monopoles and vortices, which in turn correspond to effective divisors \cite{morgan-et.al}, \cite{mrowka-et.al}, \cite{munoz}. 

%As a consequence, the moduli space $\cM$ is described in terms of holomorphic data on $\Sigma$. We end the section with a discussion of the compactifications of $\cM$ and Fueter sections. 

\subsection{Seiberg--Witten equations and quaternionic representations} \label{subsec:generalised}
Equation \eqref{eqn:nsw} is an example of the \emph{Seiberg--Witten equation associated with a quaternionic representation} introduced in \cite{taubes} and \cite{pidstrygach}; see also  \cite{haydys}, \cite{haydys2}, \cite[Section 6]{nakajima2}, \cite{doan2}. The language of quaternionic representations is well-suited for proving the circle-invariance of solutions.

Let $M$ be a quaternionic vector space equipped with a Euclidean inner product $\langle \cdot , \cdot \rangle$ compatible with the complex structures $i$, $j$, $k$. Let $\Sp(M)$ be the group of quaternionic automorphisms of $M$ preserving $\langle \cdot , \cdot \rangle$. Suppose that $G$ and $H$ are compact connected Lie groups together with a representation $G \times H \to \Sp(M)$. There is an associated \emph{hyperk\"ahler moment map} $\mu \colon M \to \mathfrak{su}(2) \otimes \mathfrak{g}$, where $\mathfrak{g}$ is the Lie algebra of $G$, determined by the identity
\[ \langle a, \mu(x) \rangle = \langle a x, x \rangle \]
for all $x \in M$ and $a \in \mathfrak{su}(2) \otimes \mathfrak{g}$. Here, we think of $\mathfrak{su}(2)$ as the space of imaginary quaternions.

%If we split $\mu$ into three components
%\[ \mu = i \mu_i + j \mu_j + k \mu_k \in \mathfrak{su}(2) \otimes \mathfrak{g}, \]
%then $\mu_i \colon M \to \mathfrak{g}$ is given by
%\[ \langle \mu_i(x), \xi \rangle = \langle i \xi x, x \rangle \quad \textnormal{for every } \xi \in \mathfrak{g} \]
%and similarly for $\mu_j$ and $\mu_k$. 
 
Let $Y$ be a Riemannian spin  three-manifold as before.  Let $P \to Y$ and $Q \to Y$ be principal bundles with structure groups $G$ and $H$, respectively. Define
\[ \mathbb{M} = (P \times_Y Q \times_Y S) \times_{G \times H \times SU(2)} M, \]
$\mathbb{M}$ has the structure of a left module over the Clifford bundle of $Y$. The moment map descends to a fibre-preserving map $\mu \colon \mathbb{M} \to \Lambda^2 Y\otimes \mathfrak{g}_P$, where $\mathfrak{g}_P$ is the adjoint bundle of $P$.
Fix a connection $B$ on $Q$. For any connection $A$ on $P$, denote by $\nabla_{AB}$ the covariant derivative on $\mathbb{M}$ induced from $A$, $B$, and the Levi--Civita connection. The pair $(\mathbb{M}, \nabla_{AB})$ is a \emph{Dirac bundle} in the sense of \cite[Definition 5.2]{michelsohn-lawson}, and as such is equipped with the Dirac operator $\D_{AB} \colon \Gamma(Y,\mathbb{M}) \to  \Gamma(Y,\mathbb{M})$. The \emph{Seiberg--Witten equation associated with the quaternionic representation} $(G\times H,M)$ is the following differential equation for a pair $(A,\Psi) \in \cA(P) \times \Gamma(\mathbb{M})$: 
\begin{equation} \label{eqn:generalised}
\left\{
\begin{array}{l}
\D_{AB} \Psi = 0, \\
F_A = \mu(\Psi).
\end{array}
\right.
\end{equation}

\begin{defn} \label{defn:irreducible}
A pair $(A,\Psi) \in \cA(P) \times \Gamma(Y, \mathbb{M})$ is called \emph{irreducible} if there exists $x \in Y$ such that the $G$--stabiliser of $\Psi(x)$ in $M$ is trivial. 
\end{defn}

\begin{exmp}
The Seiberg--Witten equation with multiple spinors \eqref{eqn:nsw} corresponds to 
\[ M = \H \otimes \C^n, \qquad H = \SU(n), \qquad G = \U(1) \]
with the standard representations.
\end{exmp}

\subsection{Circle-invariance of solutions} \label{subsec:invariance}
The goal of the next few paragraphs, which will be achieved in Proposition \ref{prop:invariance}, is to describe solutions of \eqref{eqn:generalised} under the following assumptions:
\begin{enumerate}[label=(\Alph*)]
\item $Y = S^1 \times \Sigma$ for a closed Riemann surface $\Sigma$; we endow $Y$ with the product metric.
\item The spin structure on $Y$ is induced from a spin structure $K^{1/2}$ on $\Sigma$ \footnote{Recall that a spin structure on $\Sigma$ is equivalent to a complex line bundle $K^{1/2} \to \Sigma$ together with an isomorphism of $K^{1/2} \otimes K^{1/2}$ with the canonical bundle $K = \Lambda^{1,0} \Sigma$. As a principal bundle, $S \to Y$ is pulled back from the $\SU(2)$--bundle associated to $K^{-1/2}$, the dual of $K^{1/2}$, via the inclusion $\U(1) \hookrightarrow \SU(2)$. As a vector bundle, $S = K^{1/2} \oplus K^{-1/2}$.}.
\item $Q$ and $B$ are pulled back from a bundle and a connection on $\Sigma$.
\item $P$ is pulled back from a bundle over $\Sigma$ \footnote{This will eventually follow from the existence of an irreducible solution of \eqref{eqn:generalised}.}.
\end{enumerate}

To keep the notation simple we use the same symbols $K^{1/2}$, $P$, $Q$, $B$, and so forth for the corresponding objects over $\Sigma$ and its pull-back to $Y$. Observe that $\mathbb{M} \to Y$ is pulled back from 
\[ \mathbb{M} = ( P \times_{\Sigma} Q \times_{\Sigma} K^{-1/2} ) \times_{G \times H \times \U(1) } M. \]
The action of unit quaternions $\SU(2)$ on $M$ rotates the sphere of complex structures with $\U(1) \subset \SU(2)$ being the stabiliser of $i$; thus, $\mathbb{M} \to \Sigma$ is a complex vector bundle. Consider the quaternionic vector bundle $V = (P \times_{\Sigma} Q) \times_{G \times H} M$; then $\mathbb{M} = V \otimes_{\C} K^{-1/2}$, where the complex structure on $V$ is given by $i$. The remaining part of the quaternionic structure is encoded in an anti-linear involution $j \colon V \to V$. Taking the tensor product of $j$ with the anti-linear map $K^{1/2} \to K^{-1/2}$ given by the metric, we obtain an anti-linear isomorphism
\[ \sigma \colon V \otimes K^{1/2} \to V \otimes K^{-1/2}. \]
We define similarly a map in the opposite direction, also denoted by $\sigma$, so that $\sigma^2 = -1$. Equivalently, $\sigma$ can be seen as a map $\sigma \colon \mathbb{M} \otimes K \to \mathbb{M}$,  which is a two-dimensional manifestation of the Clifford multiplication.  

Next, we relate sections and connections on $Y$ to those on $\Sigma$. Write $\cA(\Sigma, P)$ and $\cA(Y, P)$ for the spaces of connections on $P \to \Sigma$ and its pull-back to $Y$. Let $t \in [0,1]$ be the coordinate on $S^1$ in $S^1 \times \Sigma$. Any connection $A_Y \in \cA(Y,P)$ can be uniquely written in the form
\[ A_Y = A(t) + b(t) dt \]
for one-periodic families $A(t)$ of connections in $\cA(\Sigma,P)$ and sections $b(t)$ of $\mathfrak{g}_P$.  $\cA(\Sigma, P)$ embeds in $\cA(Y,P)$ by pulling-back connections. Its image consists of those $A_Y = A(t) + bdt$  for which $b(t) = 0$ and $A(t)$ is independent of $t$. Likewise, any $\Psi \in \Gamma(Y, \mathbb{M})$ can be identified with a one-periodic family $\Psi(t) \in\Gamma(\Sigma, \mathbb{M})$ and $\Gamma(\Sigma, \mathbb{M})$ embeds into $\Gamma(Y, \mathbb{M})$ as sections independent of $t$. The gauge group $\cG(\Sigma, P)$ is naturally a subgroup of $\cG(Y, P)$.
  
\begin{defn}
A \emph{circle-invariant configuration} is an element of the image of 
\[
\cA(\Sigma, P) \times \Gamma(\Sigma, \mathbb{M}) \hookrightarrow \cA(Y,P) \times \Gamma(Y, \mathbb{M}).
\]
We will identify a circle-invariant configuration on $Y$ with the corresponding pair on $\Sigma$.
\end{defn}

\begin{lem}\label{lem:inv-gauge}
If two circle-invariant configurations differ by $g \in \cG(Y,P)$, then $g \in \cG(\Sigma, P)$. In particular, $\cA(\Sigma, P) \times \Gamma(\Sigma, \mathbb{M}) / \cG(\Sigma, P)$ is a submanifold of $\cA(Y,P) \times \Gamma(Y, \mathbb{M}) / \cG(Y,P)$. 
\end{lem}

%\begin{proof}
%Let $A_1$ and $A_2$ be two circle-invariant connections. We identify $g$ with a one-periodic family $g(t)$ of gauge transformations in $\cG(\Sigma, P)$. If $A_2 = g( A_1)$, then
%\[ A_2 - A_1 = g(t)^{-1} d_{A_1} (g(t)) + \frac{\partial g}{\partial t} dt. \]
%Since the one-form $A_2 - A_1$ does not have a $dt$ part, it follows that $\partial g / \partial t = 0$ and $g(t)$ does not depend on $t$. We conclude that $g \in \cG(\Sigma, P)$. 
%\end{proof}

The Dirac operator on $Y$ can be expressed in terms of the Dolbeault operator on $\Sigma$. In the simplest case $Y = \R^3 = \R \times \C$ and $M = \mathbb{H}$, denoting coordinates on $\R \times \C$ by $t$ and $z = x + iy$, we have for a map $\Psi \colon \mathbb{R}^3 \to \mathbb{H}$ 
\begin{equation}\label{eqn:dirac0}
 \D \Psi =  i \frac{\partial \Psi}{\partial t} + j \frac{\partial \Psi}{\partial x} + k \frac{\partial \Psi}{\partial y} =   \frac{\partial \Psi}{\partial t} + j \left( \frac{\partial \Psi}{\partial x} - i \frac{\partial \Psi}{\partial y} \right) = i \frac{\partial \Psi}{\partial t} + 2 j \frac{\partial \Psi}{\partial z}. 
 \end{equation}
 In general, $\partial / \partial z$ is replaced by the Dolbeault operator 
 \[ \partial_{AB} \colon \Gamma(\Sigma, \mathbb{M}) \to \Gamma(\Sigma, \mathbb{M} \otimes K) \]
 or equivalently
  \[ \partial_{AB} \colon \Gamma(\Sigma, V \otimes K^{-1/2}) \to \Gamma(\Sigma, V \otimes K^{1/2}), \]
which is defined as the $(1,0)$--part of the covariant derivative on $\mathbb{M}$. The proof of the next lemma is a simple calculation in conformal coordinates, almost the same as \eqref{eqn:dirac0} \footnote{The difference between the constants $2$ in \eqref{eqn:dirac0} and  $\sqrt{2}$ in Lemma \ref{lem:dirac} comes from the fact that $|dz| = \sqrt{2}$ with respect to the Euclidean metric on $\C$.}.

\begin{lem}\label{lem:dirac}
Let $A_Y = A(t) + b(t)dt$ be a connection in $\cA(Y, P)$ and $\Psi = \Psi(t)$ a section in $\Gamma(Y, \mathbb{M})$. The Dirac operator $\D_{A_Y, B}$ acting on $\Gamma(Y, \mathbb{M})$ is given by
\[ \D_{A_Y B} \Psi = i \left( \frac{\partial \Psi}{\partial t}  + b(t) \Psi \right)+ \sqrt{2} \sigma \left( \partial_{A(t)B} \Psi  \right). \]
\end{lem}

In the dimensionally-reduced setting, we use the splitting of the hyperk\"ahler moment map $\mu \colon M \to \mathfrak{g} \otimes \R^3$ into the real and complex parts: $\mu_{\R} \colon M \to \mathfrak{g}$ and $\mu_{\C} \colon M \to \mathfrak{g} \otimes \C$. If $\mu = (\mu_i,\mu_j,\mu_k)$ are the three components of $\mu$, then $\mu_{\R} = \mu_i$ and $\mu_{\C} = \mu_j + i \mu_k$. The following identity will be useful later:
\begin{equation} \label{eqn:moment}
\langle \mu_{\C}(x) j x, x \rangle = | \mu_{\C}(x) |^2 
\end{equation}
Under the reduction of the structure group of $Y$ from $\SO(3)$ to $\U(1)$, the splitting $\R^3 = \R \oplus \C$ gives us $\mathfrak{su}(S) = \underline{\R} \oplus K^{-1}$. Accordingly, $\mu \colon \mathbb{M} \to \mathfrak{su}(S) \otimes \mathfrak{g}_{P}$  splits into the direct sum of
 \[ \mu_{\R} \colon \mathbb{M} \to \mathfrak{g}_P \qquad \textnormal{and} \qquad  \mu_{\C} \colon \mathbb{M} \to K^{-1} \otimes \mathfrak{g}_P. \]
$\mu_{\C}$ is holomorphic when restricted to fibres. Similarly, we have the conjugate maps
\[ \mu_{\R} \colon \overline{\mathbb{M}} \to \mathfrak{g}_P \qquad \textnormal{and} \qquad \mu_{\C} \colon \overline{\mathbb{M}} \to K \otimes \mathfrak{g}_P, \]
which satisfy 
\[ \mu_{\R} \circ \sigma = - \mu_{\R} \qquad \textnormal{and} \qquad \mu_{\C} \circ \sigma = \overline{\mu_{\C}}. \]

%which are the global versions of $\mu_{\R} \circ j = - \mu_{\R}$ and $\mu_{\C}\circ j = \overline{\mu_{\C}}$.

\begin{prop}\label{prop:invariance}
Let $A_Y = A(t) + b(t)dt$ a connection in $\cA(Y,P)$ and $\Psi = \Psi(t)$ a section in $\Gamma(Y, \mathbb{M})$.  The generalised Seiberg--Witten equation \eqref{eqn:generalised} for $(A_Y, \Psi)$ is equivalent to 
\begin{equation}\label{eqn:generalised2}
\left\{
\begin{array}{l}
i \left(\frac{\partial \Psi}{\partial t} + b \Psi \right) + \sqrt{2} \sigma \left( \partial_{AB} \Psi \right) = 0, \\
 \left( \frac{\partial A}{\partial t} + d_{A} b \right)^{0,1} = - \frac{i}{2} \mu_{\C}(\Psi), \\
\ast F_A  = \mu_{\R}(\Psi).
\end{array}
\right.
\end{equation}
In particular, for a circle-invariant configuration $(A,\Psi)$ the equation simplifies to
\begin{equation} \label{eqn:reduction}
\left\{
\begin{array}{l}
\partial_{AB} \Psi = 0, \\
\mu_{\C} (\Psi) = 0, \\
\ast F_A = \mu_{\R} (\Psi).
\end{array}
\right.
\end{equation}
\end{prop}

\begin{proof}
By Lemma \ref{lem:dirac}, the first equation in \eqref{eqn:generalised2} is equivalent to $\D_{A_Y B} \Psi = 0$. The remaining two equations are obtained from the identifications
\begin{equation}\label{eqn:proofreduction}
\mathfrak{su}(S) \cong \Lambda^0 \Sigma \oplus K^{-1} \quad \textnormal{and} \quad
 \mu \cong  \mu_{\R} \oplus \mu_{\C}
 \end{equation}
discussed earlier. Under the decomposition
\begin{equation}\label{eqn:proofreduction2}
 \Lambda^2  Y =  \left( \Lambda^2  \Sigma \right) \oplus \left( \Lambda^1 S^1 \otimes \Lambda^1 \Sigma \right)
\end{equation}
the curvature $F_{A_Y}$ decomposes into
\[ F_{A_Y} = F_{A} + dt \wedge \left( \frac{\partial A}{\partial t} + d_A b \right). \]
We need to identify the splittings \eqref{eqn:proofreduction} and \eqref{eqn:proofreduction2}  under the isomorphism  $\Lambda^2 Y \cong \mathfrak{su}(S)$. For simplicity, consider the flat case $Y = \R \times \C$, with coordinates $t$ and $z = x + i y$---the general case differs from it by a conformal factor. The isomorphism $\Lambda^2 \R^3 \cong \mathfrak{su}(2)$ is given by
\[ dx \wedge dy \mapsto i , \quad dy \wedge dt \mapsto j, \quad dt \wedge dx \mapsto k. \]
On the other hand, $\mathfrak{su}(2)$ is identified with $ \R \oplus \C$ via the map
\[ ai  + bj + ck \mapsto (a, b+ic). \]
Let $\alpha + dt \wedge \beta $ be a two-form on $\R^3$, where 
\[ \alpha = a dx \wedge dy, \quad \beta = b_1 dx + b_2 dy. \]
Under the identifications $\Lambda^2 \R^3 = \mathfrak{su}(2) = \R \oplus \C$,
\[ \alpha + dt \wedge \beta  \mapsto a i - b_2 j + b_1 k \mapsto (a, -b_2 + i b_1). \]
Observe that $a = \ast \alpha$ and $(- b_2 + ib_1) d\bar{z} = 2i \beta^{0,1}$, where $\beta^{0,1}$ is $(0,1)$-part of $\beta$. It follows that under the splittings \eqref{eqn:proofreduction} and \eqref{eqn:proofreduction2} the isomorphism
\[  \left( \Lambda^2  \Sigma \right) \oplus \left( \Lambda^1 S^1 \otimes \Lambda^1 \Sigma \right) \cong \Lambda^0 \Sigma \oplus K^{-1} \]
is the direct sum of the Hodge star $\Lambda^2 \Sigma \to \Lambda^0 \Sigma$ and the map $\Lambda^1 \Sigma \to \Lambda^{0,1}  \Sigma$ taking a one-form $\beta$ to $ 2i \beta^{0,1}$. Thus, $F_{A_Y} = \mu(\Psi)$ is equivalent to the last two equations in \eqref{eqn:generalised2}.
\end{proof}

\begin{rem}\label{rem:invariance} 
Since it is more common to consider holomorphic rather than aholomorphic sections, we can complete the picture by considering the conjugate bundle
\[ \overline{\mathbb{M}} = (Q \times P \times K^{1/2}) \times_{G \times H \times U(1)} M = V \otimes K^{1/2} = \mathbb{M} \otimes K.\]
We have the Dolbeault operators
\[ \partial_{AB} \colon \Gamma(\Sigma, \mathbb{M}) = \Gamma(\Sigma, V \otimes K^{-1/2}) \longrightarrow \Gamma(\Sigma, V \otimes K^{1/2} ) = \Gamma(\Sigma, \overline{\mathbb{M}}), \]
\[ \del_{AB} \colon \Gamma(\Sigma, \overline{\mathbb{M}}) = \Gamma(\Sigma, V \otimes K^{1/2}) \longrightarrow \Gamma(\Sigma, V \otimes K^{-1/2} ) = \Gamma(\Sigma, \mathbb{M}), \]
and the maps $\sigma \colon \overline{\mathbb{M}} \to \mathbb{M}$ and $\sigma \colon \mathbb{M} \to \overline{ \mathbb{M}}$ that intertwine them:
\[ \sigma \partial_{AB} = \del_{AB} \sigma . \]
Thus, $\sigma$ maps aholomorphic sections of $\mathbb{M}$ to holomorphic sections of $\overline{\mathbb{M}}$ and vice versa. It follows from the K\"ahler identities that
\[
\del_{AB} = - \partial_{AB}^*
\]
where $\partial_{AB}^*$ is the formal adjoint of $\partial_{AB}$.

Using $\sigma$, we can rewrite \eqref{eqn:reduction} as a system of equations for $\overline{\Psi} := \sigma(\Psi) \in \Gamma(\Sigma,\overline{\mathbb{M}})$:
\begin{equation}\label{eqn:reduction2}
\left\{
\begin{array}{l}
\del_{AB} \overline{\Psi} = 0, \\
\mu_{\C}(\overline{\Psi}) = 0, \\
\ast F_A + \mu_{\R}(\overline{\Psi}) = 0.
\end{array}
\right.
\end{equation}
This is an example of a \emph{symplectic vortex equation} discussed in \cite{cieliebak-et.al}. The target singular symplectic space\footnote{For example, for the classical Seiberg--Witten equation $\mu^{-1}_{\C}(0) = \{ (x,y) \in \C^2 \ | \ xy = 0 \}$.} is the zero locus $\mu_{\C}^{-1}(0) \subset M$.
\end{rem}

We now proceed to the main result of this section.

\begin{thm}\label{thm:invariance}
Suppose that conditions $(A)$, $(B)$, $(C)$ listed at the beginning of this subsection are satisfied. If $(A_Y, \Psi)$ is an irreducible solution of the generalised Seiberg--Witten equation \eqref{eqn:generalised}, then  $P \to Y$ is pulled-back from a bundle over $\Sigma$ (that is: condition $(D)$ is satisfied) and $(A_Y, \Psi)$ is gauge-equivalent to a circle-invariant configuration obeying equation \eqref{eqn:reduction}.
\end{thm}

\begin{rem}
Theorem \ref{thm:invariance} asserts that there is a natural one-to-one correspondence between gauge equivalence classes of solutions of the Seiberg--Witten equation on $Y = S^1 \times \Sigma$ with the hyperk\"ahler target $M$ and gauge equivalence classes of solutions of the symplectic vortex equation on $\Sigma$ with the K\"ahler target $\mu_{\C}^{-1}(0)$.
\end{rem}

\begin{proof}
Assume for simplicity that the flavour symmetry $H$ is trivial---the general proof is the same after adjusting the notation. Identify $S^1$ with $[0,1]$ with the endpoints glued together. Pull-back the data on $Y=S^1 \times \Sigma$ to one-periodic data on $[0,1] \times \Sigma$. Since $[0,1] \times \Sigma$ is homotopy equivalent to $\Sigma$, there is a principal $G$-bundle $P_{\Sigma} \to \Sigma$ and a gauge transformation $g \in \cG(\Sigma, P_{\Sigma})$  such that $P$ is the quotient of $[0,1] \times P_{\Sigma}$ by the relation $(0,p) \sim (1, g(p))$. The isomorphism class of $P$ depends only on the homotopy class of $g$. Similarly,  $\mathbb{M} = (P \times S) \times_{G \times SU(2)} M$ over $Y$ is obtained from pulling-back  $\mathbb{M}_{\Sigma} =  (P_{\Sigma} \times K^{-1/2})_{G \times U(1)} M$ to $[0,1] \times \Sigma$ and identifying the fibres over $0$ and $1$ using $g$. As before there is an anti-linear map $\sigma \colon \mathbb{M}_{\Sigma} \otimes K \to \mathbb{M}_{\Sigma}$.  
 
For $A_Y \in \cA(Y,P)$ and $\Psi \in \Gamma(Y, \mathbb{M})$ we have
 \[ A_Y = A(t) + b(t) dt, \quad \Psi = \Psi(t), \]
 where $A(t), b(t)$, and $\Psi(t)$ are families of connections and sections on $\Sigma$, as discussed earlier. The only difference now is that the families are periodic with respect to the action of $g$:
 \[ A(1) = g \left( A(0) \right), \quad b(1) = g \left(b(0) \right), \quad \Psi(1) = g\left( \Psi(0) \right). \]
Define a gauge transformation $h$ over $[0,1] \times \Sigma$ by
\begin{equation} \label{eqn:temporal}
 h(t) = \exp\left( \int_0^t b(s) ds \right). 
 \end{equation}
$h$ does not necessarily descend to an automorphism of $P \to Y$; this happens if and only if $h(1) = \mathrm{Ad}_g(h_0) = \textnormal{id}$.  In any case, $h$ is well-defined over $[0,1]\times \Sigma$ and the new connection
\[  C \coloneqq h(A_Y) = A_Y - h^{-1} d_{A_Y} h = h(t)\left( A(t) \right) \]
does not have a $dt$ part---it is in a \emph{temporal gauge}. Thus, it is identified with a path of connections $\{ C(t) \}_{t \in [0,1]}$ on $P_{\Sigma}$ satisfying $C(1) = h(1)g(C(0))$. Likewise, we identify the section $h(\Psi)$ with a path $\{ \Phi(t) \}_{t\in[0,1]}$ of sections of $\mathbb{M}_{\Sigma}\to\Sigma$ satisfying $\Phi(1) = h(1)g(\Phi(0))$. 

By Proposition \ref{prop:invariance}, the Seiberg--Witten equation for $(C,\Phi)$ is equivalent to
\begin{equation} \label{eqn:invarianceproof}
\left\{
\begin{array}{l}
i \frac{\partial \Phi}{\partial t} + \sqrt{2} \sigma \left( \partial_C \Phi \right) = 0, \\
\left( \frac{\partial C}{\partial t} \right)^{0,1} = - \frac{i}{2}\mu_{\C}( \Phi ), \\
\ast F_C = \mu_{\R}(\Phi). 
\end{array}
\right. 
\end{equation}
Differentiating the first equation with respect to $t$ and using \eqref{eqn:invarianceproof}, we obtain
\[
\begin{split}
 0 & = i \frac{\partial^2 \Phi}{\partial t^2} + \sqrt{2} \sigma \left\{ \left( \frac{\partial C}{\partial t} \right)^{1,0} \Phi + \partial_C \left( \frac{\partial \Phi}{\partial t} \right)   \right\} \\
 & = i \frac{\partial^2 \Phi}{\partial t^2} + \frac{\sqrt{2}}{2} \sigma i\overline{\mu_{\C}} (\Phi) \Phi + 2  \sigma  i\partial_C \sigma \partial_C \Phi \\
 & = i \frac{\partial^2 \Phi}{\partial t^2} -  \frac{\sqrt{2}}{2} i \mu_{\C}(\Phi) \sigma \Phi - 2 i  \sigma  \partial_C \sigma \partial_C \Phi.
 \end{split}
\]
We have used the anti-linearity of $\sigma$ and the fact that  $\partial C / \partial t$ is a real $\mathfrak{g}$--valued one-form, so its $(1,0)$ part is conjugate to the $(0,1)$ part. Multiplying the obtained identity by $i$ and taking the pointwise inner product with $\Phi$ yields
\begin{equation} \label{eqn:invarianceproof2}
 0 = \left\langle - \frac{\partial^2 \Phi}{\partial t^2}, \Phi \right\rangle +  \frac{\sqrt{2}}{2}  \left\langle  \mu_{\C}(\Phi) \sigma \Phi, \Phi \right\rangle + 2 \left\langle \sigma \partial_C \sigma \partial_C \Phi , \Phi \right\rangle. 
 \end{equation}
 By formula \eqref{eqn:moment} the second term simplifies to $\sqrt{2}/2 | \mu_{\C}(\Phi)|^2$.  Remark \ref{rem:invariance} implies that
\[ \sigma \partial_C \sigma \partial_C = \del_C \left( \sigma^2 \right) \partial_C = - \del_C \partial_C = \partial_C^* \partial_C. \] 
We conclude that 
\[ \left\langle \sigma \partial_C \sigma \partial_C \Phi , \Phi \right\rangle_{L^2(\Sigma)} =  \left\langle \partial_C^* \partial_C \Phi, \Phi \right\rangle_{L^2(\Sigma)} = \left\| \partial_C \Phi \right\|_{L^2(\Sigma)}^2. \]
For a fixed value of $t$ integration of \eqref{eqn:invarianceproof2} over $\Sigma$ yields
\[ 0 = \int_{\Sigma} \left\langle - \frac{\partial^2 \Phi}{\partial t^2} , \Phi \right\rangle \mathrm{vol}_{\Sigma} +   \frac{\sqrt{2}}{2} \| \mu_{\C}( \Phi) \|_{L^2(\Sigma)}^2 + 2 \| \partial_C \Phi \|_{L^2(\Sigma)}^2. \]
Integrate the last equality by parts with respect to $t \in [0,1]$. The boundary terms vanish because  $\Phi$ is periodic up to the action of $h(1)g$ which preserves the inner product. We obtain
\[ 0 = \left\| \frac{\partial \Phi}{\partial t} \right\|^2_{L^2} + \frac{\sqrt{2}}{2} \left\| \mu_{\C}(\Phi) \right\|^2_{L^2} + 2 \left\| \partial_C \Phi \right\|^2_{L^2}, \]
which shows that 
\[ \frac{\partial \Phi}{\partial t} = 0, \quad  \frac{\partial C}{\partial t} = 0. \]
Thus, the families $C(t) = C$ and $\Phi(t) = \Phi$ are constant and
\[ C = C(1) = k \left( C(0) \right) = k(C), \quad \Phi = \Phi(1) = k\left( \Phi(0) \right) = k(\Phi) \]
for the gauge transformation $k = h(1)g$ over $\Sigma$. The first equality implies $d_C k = 0$, so $k$ is covariantly constant. On the other hand, by irreducibility, there exists a point $x \in \Sigma$ such that the $G$--stabiliser of $\Phi(x)$ is trivial. Hence, $k(x) = \textnormal{id}$, so $k = \textnormal{id}$ everywhere and $g = h(1)^{-1}$. The path $h(t)^{-1}$ is a homotopy of gauge transformations connecting $g$ with $h(0)^{-1} = \textnormal{id}$ and so $P \to Y$ is pulled-back from $P_{\Sigma} \to \Sigma$. In particular, we could have chosen $g=\mathrm{id}$, then $h(1) = \mathrm{id}$ and $h$ descends to a gauge transformation of $P$ mapping $(A_Y, \Psi)$ to the circle-invariant solution $(C,\Phi)$. By Proposition \ref{prop:invariance}, $(C,\Phi)$ satisfies equation \eqref{eqn:reduction}.  
\end{proof}

\begin{rem}
Much of this discussion can be extended to the setting when $M$ is a hyperk\"ahler manifold with an isometric $\SU(2)$--action rotating the sphere of complex structures. The Dirac operator $\D_{AB}$ and equation \eqref{eqn:generalised} have natural generalisations \cite{haydys, haydys2}. For $Y = S^1 \times \Sigma$ one introduces the non-linear Dolbeault operator $\partial_{AB}$ as in \cite{cieliebak-et.al} so that Lemma \ref{lem:dirac} and Proposition \ref{prop:invariance} hold. However, our proof of Theorem \ref{thm:invariance} makes use of the vector space structure on $M$ and does not immediately generalise to the non-linear setting. We expect the result to be true but in the proof one should use the Weitzenb\"ock formulae for non-linear Dirac operators \cite{taubes, pidstrygach, callies}.
\end{rem}

\subsection{An abelian vortex equation}
We apply Theorem \ref{thm:invariance} to the Seiberg--Witten equation with multiple spinors. In this case, $M = \H \otimes \C^n$, $H = \SU(n)$, and $G = \U(1)$. The action of $\SU(n)$ on $M$ is the transpose of the standard representation. We identify $\H$ with $\C \oplus \overline{\C}$ via
\[ a + b i + cj + dk = (a + b i) + (c - d i) j \mapsto ( a + bi, c + di) , \]
so that the action of $\U(1)$ from the left preserves the complex structure given by the left multiplication by $i$. Then $M$ is identified with $\C^n \oplus \overline{\C}^n$ and the moment maps are
\[ \mu_{\R}(x,y) = i ( |x|^2 - |y|^2), \qquad \mu_{\C}(x,y) = y^* x. \] 
Suppose that $E \to Y$ and $L \to Y$ are pulled back from bundles over $\Sigma$, and so is the background connection $B$. We form the associated bundle
\[ \overline{\mathbb{M}} = (E \times L \times K^{1/2}) \times_{SU(n) \times U(1)} M = (E \otimes L^* \otimes K^{1/2}) \oplus (E^* \otimes L \otimes K^{1/2}). \]
Every circle-invariant section $\Psi$ of $\overline{\mathbb{M}}$ can be written in the form $\Psi = (\beta, \alpha)$ for 
\[ \alpha \in \Gamma(\Sigma, E^* \otimes L \otimes K^{1/2})  \quad \textnormal{and} \quad \beta \in \Gamma(\Sigma, E \otimes L^* \otimes K^{1/2}) . \]
The real moment map is $\mu_{\R}(\beta, \alpha) = i( |\beta|^2 - |\alpha|^2)$ and the complex moment map is
\[ \Gamma(\Sigma, E^* \otimes L \otimes K^{1/2} ) \times \Gamma(\Sigma, E \otimes L^* \otimes K^{1/2}) \to \Gamma(\Sigma, K), \]
\[ (\alpha, \beta) \mapsto \alpha \beta. \]
Suppose that the closed one-form $\eta$ used to perturb \eqref{eqn:nsw} is pulled back from $\Sigma$. By Remark \ref{rem:invariance}, for circle-invariant configurations equation \eqref{eqn:nsw} reduces to
\begin{equation} \label{eqn:nsw-invariant}
\left\{
\begin{array}{l}
\del_{AB} \alpha = 0, \\
\del_{AB} \beta = 0, \\
\alpha \beta = 0, \\
i \ast  F_A + | \alpha |^2 - |\beta|^2 - i \ast \eta = 0.
\end{array}
\right.
\end{equation}

Now is a good point to introduce the space $\sP_{\Sigma}$  appearing in Theorem \ref{thm:generic0}.

\begin{defn}
  \label{Def_ParameterSpaceSigma}
  Let $Y = S^1 \times \Sigma$ be equipped with a spin structure pulled-back from $\Sigma$. Let $E \to Y$ be an $\SU(n)$--bundle pulled-back from $\Sigma$.
  
  We introduce the following Fr\'echet manifolds equipped with the $C^{\infty}$--topology \\

\begin{tabular}{lcl}
$\Met_{\Sigma}$ & & the space of Riemannian metrics on $\Sigma$, \\
$\cZ_{\Sigma}$  & & the space of closed imaginary-valued two-forms on $\Sigma$ \\
$\cP_{\Sigma} := \cA(\Sigma, E) \times \cZ_{\Sigma}$ & & the \emph{space of perturbations pulled-back from} $\Sigma$,
\end{tabular}\\

The \emph{space of parameters of the equation pulled-back from $\Sigma$} is 
\begin{equation*}
\sP_{\Sigma} := \Met_{\Sigma} \times\cP_{\Sigma}.
\end{equation*} 
We consider it as a subspace of the full parameter space $\sP$ introduced in \ref{Def_Parameters}.
\end{defn}

\begin{prop} \label{prop:nsw-invariant}
Assume the situation described in Definition \ref{Def_ParameterSpaceSigma}.
If $(A, \Psi)$ is an irreducible solution of the Seiberg--Witten equation with multiple spinors \eqref{eqn:nsw} with respect to a Riemannian metric $g \in \Met_{\Sigma}$ and a parameter $\sigma = (B,\eta) \in \cP_{\Sigma}$, then 
\begin{enumerate}
\item $L$ is pulled back from a bundle over $\Sigma$, and 
\item $(A,\Psi)$ is gauge-equivalent to a circle-invariant configuration satisfying equation \eqref{eqn:nsw-invariant}.
\end{enumerate}
\end{prop}

\section{A holomorphic description of the moduli space}
\label{sec_holomorphic}

The main result of this section, Theorem \ref{thm:holomorphic} below, identifies the moduli space of Seiberg--Witten monopoles on $Y = S^1 \times \Sigma$ with a certain moduli space of holomorphic data on $\Sigma$. 

\subsection{Bryan--Wentworth moduli spaces}
In the situation of Proposition \ref{prop:nsw-invariant}, the moduli space $\cM^*=\cM^*(g,\sigma)$ has the following description. Let
\[ \cC_{\Sigma} = \cA(\Sigma, L) \times \Gamma(\Sigma, E^* \otimes L \otimes K^{1/2}) \times \Gamma(\Sigma, E \otimes L^* \otimes K^{1/2}). \]
Consider the subspace of $\cC_{\Sigma}$ consisting triples $(A,\alpha,\beta)$ satisfying equations \eqref{eqn:nsw-invariant} and condition $(\alpha,\beta) \neq (0,0)$. The gauge group $\cG(\Sigma) = C^{\infty}(\Sigma, S^1)$ acts freely on this subspace. By Proposition \ref{prop:nsw-invariant} and Lemma \ref{lem:inv-gauge}, the quotient is homeomorphic to $\cM^*$. 

The next result extends the work of Bryan and Wentworth \cite{bryan-wentworth} who described Seiberg--Witten multi-monopoles on K\"ahler surfaces under the assumption that the background bundle $E$ is trivial and $B$ is the product connection.  Before stating the theorem, we introduce
\[ d = \deg(L) := \langle c_1(L), [\Sigma] \rangle \qquad \textnormal{and} \qquad \tau := \int_{\Sigma} \frac{i \eta}{2\pi}. \]
If $d - \tau < 0$, then the last equation of \eqref{eqn:nsw-invariant} forces $\alpha$ to be non-zero for
\[ 0 = \int_{\Sigma} \left\{ i F_A - i \eta + (| \alpha |^2 - | \beta |^2) \mathrm{vol}_{\Sigma} \right\} = 2\pi( d - \tau ) + \| \alpha \|_{L^2}^2 - \| \beta \|_{L^2}^2. \] 
Likewise, if $d - \tau > 0$, then $\beta$ must be non-zero. In both cases there are no reducible solutions and $\cM = \cM^*$. When $d - \tau = 0$, either both $\alpha$ and $\beta$ are non-zero or both of them vanish yielding a reducible solution. 

Recall that in dimension two every unitary connection equips the underlying vector bundle with a holomorphic structure. In particular, $K^{1/2}$ and $K$ are holomorphic line bundles.

\begin{defn}
 Denote by $\cE_B \to \Sigma$ the holomorphic $\SL(n,\C)$--bundle obtained from $E^*$ with the dual connection $B^*$. We will write simply $\cE$ in a context in which $B$ is fixed. 
\end{defn}

\begin{thm} \label{thm:holomorphic}
If $d - \tau< 0$, then $\cM^*$ is isomorphic as real analytic spaces to the moduli space $\cM_{\hol}$ of triples $(\mathcal{L}, \alpha, \beta)$ consisting of
\begin{itemize}
\item a degree $d$ holomorphic line bundle $\mathcal{L} \to \Sigma$, 
\item holomorphic sections
\[ 
\alpha \in H^0(\Sigma, \cE_B \otimes \mathcal{L} \otimes K^{1/2} ) \quad \textnormal{and} \quad \beta \in H^0(\Sigma, \cE_B^* \otimes \mathcal{L}^* \otimes K^{1/2})
\]
satisfying $\alpha \neq 0$ and $\alpha \beta = 0 \in H^0(\Sigma, K)$.
\end{itemize}
Two such triples $(\mathcal{L},\alpha,\beta)$ and $(\mathcal{L}', \alpha', \beta')$ correspond to the same point in the $\cM_{\hol}$ if there is a holomorphic isomorphism $\mathcal{L} \to \mathcal{L}'$ mapping $\alpha$ to $\alpha'$ and $\beta$ to $\beta'$. 

The statement still holds when $d - \tau \geq 0$ with the difference that for $d - \tau  > 0$ it is $\beta$ instead of $\alpha$ that is required to be non-zero and for $d - \tau = 0$ both $\alpha$ and $\beta$ are required to be non-zero.
\end{thm}

The next few paragraphs are occupied with a construction of $\cM_{\hol}$. A related construction was considered in \cite[Section 1]{friedman-morgan2}. The first three equations in  \eqref{eqn:nsw-invariant},
\begin{equation} \label{eqn:holomorphic}
\left\{
\begin{array}{l}
\del_{AB} \alpha = 0, \\
\del_{AB} \beta = 0, \\
\alpha \beta = 0,
\end{array}
\right.
\end{equation}
are invariant under the action of the \emph{complexified gauge group} $\cG^c(\Sigma) := C^{\infty}(\Sigma, \C^*)$ of complex automorphisms of $L$. The action of $g \colon \Sigma \to \C^*$ on $(A,\alpha,\beta) \in \cC_{\Sigma}$ is given by
\[ g(A, \alpha, \beta) = \left(A + \overline{g}^{-1} \partial \overline{g} - g^{-1} \del g, g \alpha, g^{-1} \beta \right).    \] 
In terms of the associated Dolbeault operators we have
\[
\begin{array}{ll}
 \del_{g(A)B} = g  \del_{B A} g^{-1} & \textnormal{on } \Gamma(\Sigma, E^* \otimes L \otimes K^{1/2}), \\
 \del_{g(A)B} = g^{-1} \del_{BA} g & \textnormal{on }  \Gamma(\Sigma, E \otimes L^* \otimes K^{1/2}).
 \end{array}
  \]
  
\begin{defn} \label{defn:holomorphic}
Consider the subspace of $\cC_{\Sigma}$ consisting of triples $(A,\alpha,\beta)$ satisfying equations \eqref{eqn:holomorphic} and subject to the condition
\[ \left\{ \begin{array}{ll}
\alpha \neq 0 & \textnormal{if } d - \tau < 0, \\
\beta \neq 0 & \textnormal{if } d - \tau > 0, \\
\alpha \neq 0 \textnormal{ and } \beta \neq 0 & \textnormal{if } d - \tau = 0.
\end{array}  \right. \]
We define $\cM_{\hol}$ to be the quotient of this subspace by the action of $\cG^c(\Sigma)$. The points of $\cM_{\hol}$ parametrise the isomorphism classes of triples $( \mathcal{L}, \alpha, \beta )$ considered in Theorem \ref{thm:holomorphic}.
\end{defn}

\begin{rem}
The moduli space $\cM_{\hol}$ depends on the conformal class of the metric $g$ on $\Sigma$, the holomorphic bundle $\cE_B$, the degree $d$ of $L$, and the sign of $d - \tau$.
\end{rem}

Employing the methods of subsection \ref{subsec:moduli}, one shows that $\cM_{\hol}$ is metrisable, second countable, and has a natural complex analytic structure given by local Kuranishi models as in Remark \ref{rem:analytic}. The discussion is almost the same as that for the Seiberg--Witten equation, so we only outline the details. To set up the Fredholm theory, consider the modified equation
\begin{equation} \label{eqn:holomorphic2}
\left\{
\begin{array}{l}
\del_{AB} \alpha + i f \bar{\beta} = 0, \\
\del_{AB} \beta - i f \bar{\alpha} = 0, \\
\alpha \beta + \partial f = 0.
\end{array}
\right.
\end{equation}
A solution of \eqref{eqn:holomorphic2} is a quadruple $(A, \alpha, \beta, f)$ where $A$, $\alpha$, and $\beta$ are as before and $f \in C^{\infty}(\Sigma, \C)$. The equation is elliptic modulo the action of $\cG^c(\Sigma)$. An analogue of Proposition \ref{prop:irreducible} is

\begin{prop}
If $(A,\alpha,\beta,f)$ is a solution of \eqref{eqn:holomorphic2} with $(\alpha,\beta) \neq 0$, then $f = 0$.
\end{prop}

Using the linearisation of \eqref{eqn:holomorphic2} together with the complex Coulomb gauge fixing we represent $\cM_{\hol}$ as the zero set of a Fredholm section. The local structure of the moduli space is encoded in the elliptic complex at a solution $(A, \alpha, \beta, 0)$:
\[ 
\begin{tikzcd}
\Omega^0(\C) \ar{r}{G^c_{A,\alpha,\beta}} & \Omega^{0,1} \oplus \Gamma(E^* \otimes L \otimes K^{1/2}) \oplus \Gamma(E \otimes L^* \otimes K^{1/2}) \oplus \Omega^0(\C) \ar{r} & \hphantom{0} \\
 \hphantom{0}  \ar{r}{F_{A,\alpha,\beta}} & \Gamma(E^* \otimes L \otimes K^{-1/2}) \oplus \Gamma(E \otimes L^* \otimes K^{-1/2}) \oplus \Omega^{1,0}  
\end{tikzcd} 
\]
where $G^c_{A,\alpha,\beta}$ is the linearised action of the complexified gauge group
\[ G^c_{A,\alpha,\beta}(h) = (- \del h, h\alpha, -h\beta, 0), \]
and $F_{A,\alpha,\beta}$ is the linearisation of equations \eqref{eqn:holomorphic2}
\[ F_{A,\alpha,\beta}(a^{0,1}, u, v, t) = 
\left(
\begin{array}{l}
\del_{AB} u + a^{0,1} \alpha + i t \bar{\beta} \\
\del_{AB} v - a^{0,1} \beta - i t \bar{\alpha}  \\
u \beta + \alpha v + \partial t 
\end{array}
\right). \]
Even though the map given by the left-hand side of \eqref{eqn:holomorphic2} is not holomorphic, its derivative $F_{A,\alpha,\beta}$ at a solution $(A,\alpha,\beta,0)$ is complex linear and so the cohomology groups $H^0_{A,\alpha,\beta}$, $H^1_{A,\alpha,\beta}$, $H^2_{A,\alpha,\beta}$ are complex vector spaces. If the solution is irreducible, then $H^0_{A,\alpha,\beta} = 0$. We are left with complex vector spaces $H^1_{A,\alpha,\beta}$ and $H^2_{A,\alpha,\beta}$ of the same dimension. They have the following description---the proof is yet another variation of that of Proposition \ref{prop:irreducible}.

\begin{lem} \label{lem:cohomology}
Let $(A,\alpha, \beta, f)$ be a solution of \eqref{eqn:holomorphic} with $(\alpha, \beta) \neq 0$ and $f = 0$. Then the deformation space $H^1_{A,\alpha,\beta}$ is the quotient of the space of solutions 
\[  (a^{0,1}, u, v) \in \Omega^{0,1}(\C) \oplus \Gamma(E^* \otimes L \otimes K^{1/2} ) \oplus \Gamma(E \otimes L^* \otimes K^{1/2}), \]
\[ \left\{
\begin{array}{l}
\del_{AB} u + a^{0,1} \alpha = 0, \\
\del_{AB} v - a^{0,1} \beta = 0, \\
u \beta + \alpha v = 0. 
\end{array}
\right.\] 
by the subspace generated by $( - \del h, h\alpha, - h\beta)$ for $h \in \Omega^0(\C)$. The obstruction space $H^2_{A, \alpha, \beta}$ is canonically isomorphic to the dual space  $(H^1_{A, \alpha, \beta})^*$ as complex vector spaces. 
\end{lem}

The analytic structure on a neighbourhood of $[A,\alpha,\beta]$ in $\cM_{\hol}$ is induced from a Kuranishi map $\kappa \colon H^1_{A,\alpha,\beta} \to H^2_{A,\alpha,\beta}$. Since the derivative of $F_{A,\alpha,\beta}$ is complex linear at a solution, $\kappa$ can be taken to be complex analytic which shows that $\cM_{\hol}$ is a complex analytic space.

\subsection{A homeomorphism between the moduli spaces}

Since equation \eqref{eqn:holomorphic} is part of \eqref{eqn:nsw-invariant} and $\cG(\Sigma)$ is a subgroup of $\cG^c(\Sigma)$, every point of $\cM^*$ gives rise to a point in $\cM_{\hol}$.

\begin{prop}
\label{prop_homeomorphism}
The natural map $ \cM^* \to \cM_{\hol}$ is a homeomorphism.
\end{prop}

The proof relies on a generalisation of a classical theorem of Kazdan and Warner.

\begin{lem}[Bryan--Wentworth; Lemma 3.4 in \cite{bryan-wentworth}] \label{lem:bryan-wentworth}
Let $X$ be a compact Riemannian manifold and let $P$, $Q$, and $w$ be smooth functions on $X$ with $P$ and $Q$ non-negative, and
\[ \int_X P - Q > 0, \qquad \int_X w > 0. \]
Then the equation 
\[ \Delta u + P e^u - Q e^{-u} = w \]
has a unique solution $u \in C^{\infty}(X)$.
\end{lem}

\begin{rem} \label{rem:bryan-wentworth}
We refer to \cite{bryan-wentworth} for the proof of Lemma \ref{lem:bryan-wentworth}, but let us remark that it can be easily extended to the case when $\int w = 0$ and both $P$ and $Q$ are not identically zero (without the assumption on the sign of $\int P - Q$). First, observe that if we replace $P$ and $Q$ by $P' = e^C P$ and $Q' = e^{-C} Q$ respectively, then solving the corresponding equation
\[ \Delta u + P' e^u - Q' e^{-u} = w \]
is equivalent to solving the original equation  with $P$ and $Q$. Indeed,  if $u$ solves the former, then the function $u + C$ is a solution the latter. Since both $P$ and $Q$ are not identically zero, their integrals are positive and by choosing the constant
\[ C = \frac{1}{2} \log \left( \frac{ \int_X Q }{ \int_X P} \right), \]
we can guarantee that $\int P' - Q' = 0$. Thus, we may as well assume that this holds for the original functions $P$ and $Q$. After this adjustment, we simply repeat the proof from \cite{bryan-wentworth}. The only difference is the construction of sub- and super-solutions, that is functions $u_-$ and $u_+$ satisfying $u_- \leq u_+$ and
\[ \Delta u_- + Pe^{u_-} - Q e^{-u_-} - w \leq 0, \]
\[ \Delta u_+ P e^{u_+} - Q e^{-u_+} - w \geq 0. \]
Let $v_1$ and $v_2$ be solutions of $\Delta v_1 = w$ and $\Delta v_2 = - P + Q$. Choose a constant $M$ such that $M \geq \sup | v_1 + v_2| $  and set 
\[ u_+ = v_1 + v_2 + M \qquad \textnormal{and} \qquad u_- = v_1 + v_2 - M. \]
Then clearly $u_- \leq u_+$ and
\[\Delta u_- + Pe^{u_-} - Q e^{-u_-} - w = P \left( e^{v_1 + v_2 - M} -1 \right) + Q \left(1- e^{-v_1 - v_2 + M}   \right) \leq 0,  \]
\[\Delta u_+ + Pe^{u_+} - Q e^{-u_+} - w = P\left( e^{v_1 + v_2 + M} - 1\right) + Q \left( 1 - e^{-v_1 - v_2 -M} \right) \geq 0. \] 
The remaining part of the proof from \cite{bryan-wentworth} goes through unchanged.
\end{rem}

\begin{proof}[Proof of Proposition \ref{prop_homeomorphism}]
It is clear that the map $\cM^* \to \cM_{\hol}$ is continuous, so it remains to construct a continuous inverse $\cM_{\hol} \to \cM^*$.
Let $(A, \alpha, \beta)$ be a solution of \eqref{eqn:holomorphic}. As in \cite{bryan-wentworth}, we seek $h \in \cG^c(\Sigma)$ such that $h (A, \alpha, \beta) = (A', \alpha', \beta')$ satisfies also the third equation of \eqref{eqn:nsw-invariant}. We can assume $h = e^f$ for $f \colon \Sigma \to \R$. We have 
\[ 
(A', \alpha', \beta') = (A - \del f + \partial f, e^f \alpha, e^{-f} \beta)
\]
so the curvature of $A'$ is 
\[ F_{A'} = F_{A} - 2 \partial \del f = F_{A} - i \ast \Delta f, \]
where $\Delta$ is the positive definite Hodge Laplacian. Thus, \eqref{eqn:nsw-invariant} for $(A', \alpha', \beta')$ is equivalent to
\[ 
\begin{split}
0 &= i \ast F_{A'} + |\alpha'|^2 - |\beta'|^2 - i\eta  \\
&=  \Delta f + e^{2f} | \alpha |^2 - e^{-2f} | \beta |^2 + i (\ast F_{A} - \eta).
\end{split}
\]
Assume $d - \tau< 0$ and set $P = | \alpha |^2$, $Q = | \beta|^2$ and $w = -i \ast (F_{A} - \eta)$. We need to solve
\begin{equation} \label{eqn:bryan-wentworth2}  \Delta f + P e^{2f} - Q e^{-2f}  = w. 
\end{equation}
If $d - \tau < 0$, then $\alpha$ is assumed to be non-zero. After applying a gauge transformation of the form $h = e^C$ for $C$ constant, we may assume that
\[ \int_{\Sigma} (P-Q)  = \int_{\Sigma} | \alpha |^2 - | \beta|^2 > 0. \]
Moreover, we have
\[
 \int_{\Sigma} w =  - \int_{\Sigma} \left( i F_{A} - \eta \right) = -2\pi (d - \tau) > 0.
\]
The hypotheses of Lemma \eqref{lem:bryan-wentworth} are satisfied and there is a unique solution $f \in C^{\infty}(\Sigma)$ to \eqref{eqn:bryan-wentworth2}. This shows that there exists $h \in \cG^c(\Sigma)$, unique up to an element of $\cG(\Sigma)$, mapping $(A,\alpha, \beta)$ to a solution of \eqref{eqn:nsw-invariant}. If $d - \tau> 0$ the proof is the same with $P = | \beta |^2$, $B = |\alpha |^2$, $w = i \ast( F_A - \eta)$ and $f$ replaced in the equation by $-f$. The case $d = 0$ follows from Remark \ref{rem:bryan-wentworth}.

This gives us an inverse to $\cM^* \to \cM_{\hol}$; it remains to show that it is continuous.  Let $[A_i, \alpha_i, \beta_i]$ be a convergent sequence of points in $\cM_{\hol}$. Let $(A_i', \alpha_i', \beta_i')$ be the corresponding solutions of \eqref{eqn:nsw-invariant}. There is a sequence $h_i = u_i e^{f_i}$ such that $h_i (A_i', \alpha_i', \beta_i')$ converges in $\cC_{\Sigma}$. The functions $f_i$ satisfy \eqref{eqn:bryan-wentworth2} with coefficients $P_i$, $Q_i$, $w_i$ converging in $C^{\infty}(\Sigma)$. It follows from the proof of \cite[Proposition 3.1]{doan} that for every $k$ there is a $C^k$ bound for $f_i$, independent of $i$. By the Arzel\`{a}--Ascoli theorem, after passing to a subsequence, $f_i$ converges in $C^{\infty}(\Sigma)$. It follows that $[A_i',\alpha_i',\beta_i']$ converges in $\cM^*$, which proves the continuity of $\cM_{\hol}\to\cM^*$.
\end{proof}

\subsection{Proof of Theorem \ref{thm:holomorphic}}
It remains to compare the deformation theories of the two moduli spaces to show that the homeomorphism $\cM^* \to \cM_{\hol}$ is an isomorphism of real analytic spaces.

\begin{step}
$\cM^*$ is isomorphic to the moduli space $\cM_{\Sigma}$ of solutions of \eqref{eqn:nsw-invariant}.
\end{step}
Let $\cM_{\Sigma}$ be the space of $\cG(\Sigma)$--orbits of triples
\[ (A, \alpha, \beta) \in \cA(\Sigma,L) \times \Gamma(\Sigma, E^* \otimes L \otimes K^{1/2}) \times \Gamma(\Sigma, E \otimes L^* \otimes K^{1/2}) \]
 satisfying $(\alpha,\beta) \neq 0$ and
\[
\left\{
\begin{array}{l}
\del_{AB} \alpha = 0, \\
\del_{AB} \beta = 0, \\
\alpha \beta = 0, \\
i \ast  F_A + | \alpha |^2 - |\beta|^2 - i \ast \eta = 0
\end{array}
\right.
\]
Endow $\cM_{\Sigma}$ with a real analytic structure using local Kuranishi models\footnote{As in the constructions of $\cM$ and $\cM_{\hol}$ one introduces an extra term $f \in C^{\infty}(\Sigma, i\R)$ to make the equation elliptic modulo gauge.We can ignore this because analogues of Proposition \ref{prop:irreducible} and Lemma \ref{lem:h1} hold.}. Let $\cC^*_{\Sigma}$ and $\cC^*$ be the  spaces of irreducible configurations $(A,\Psi)$ over $\Sigma$ and $Y$ respectively. It follows from Proposition \ref{prop:nsw-invariant} and Lemma \ref{lem:inv-gauge} that the inclusion
\[ \mathcal{B}^*_{\Sigma} = \cC^*_{\Sigma} / \cG(\Sigma) \hookrightarrow \cC^* / \cG(\Sigma)  = \mathcal{B}^* \]
induces a homeomorphism $\cM_{\Sigma} \to \cM^*$. The Seiberg--Witten moduli space $\cM^*$ is, at least locally, given as the zero set of a Fredholm section $s$ of a bundle over $\mathcal{B}^*$. On the other hand, the restriction of $s$ to $\mathcal{B}^*_{\Sigma}$ gives a Fredholm section defining $\cM_{\Sigma}$. In order to show that the induced real analytic structures agree we need to prove
\begin{equation} \label{eq_deformationspaces}
 H^1_{A,\Psi} = \ker ds_{(A,\Psi)} = \ker d (\restr{s}{\mathcal{B}^*_{\Sigma}})_{(A,\Psi)}
\end{equation}
 for every $[A, \Psi] \in  \cM_{\Sigma} = \cM^*$. The corresponding equality of cokernels follows then from the natural isomorphism between $H^1_{A,\Psi}$ and $H^2_{A,\Psi}$ (and likewise for the equations over $\Sigma$). 

Equality \eqref{eq_deformationspaces} is the linearised version of Theorem \ref{thm:invariance}.  Let $(A,\Psi)$ be a circle-invariant solution.  According to Lemma \ref{lem:h1}, $H^1_{A,\Psi} = \ker S_{A,\Psi} \oplus G_{A,\Psi}^*$, where $S_{A,\Psi}$ is the linearisation of the equation without the extra term $f$ and $G_{A,\Psi}$ the infinitesimal gauge group action. Using Proposition \ref{prop:invariance}, we identify $H^1_{A,\Psi}$ with the space of pairs
\[ (a(t) + b(t) dt, \phi(t)) \in \Gamma(S^1 \times \Sigma, \Lambda^1(i \R) \oplus \Lambda^0( i \R) \oplus (E^* \otimes S \otimes L ) ) \]
satisfying
\[
\left\{
\begin{array}{l}
i \left( \frac{ \partial \phi}{\partial t} + b \Psi \right) + \sqrt{2} \sigma \left( \partial_{AB} \phi + a^{1,0} \Psi \right) = 0, \\
 \frac{ \partial a^{1,0} }{\partial t} + \partial b - i \mu_{\C}(\Psi, \phi) = 0, \\
 \ast da + 2\mu_{\R} ( \Psi, \phi) = 0, \\
 - d^* a - \frac{\partial b}{\partial t} + i \Im\langle \Psi, \phi \rangle = 0.
\end{array}
\right.
\]
Equality \eqref{eq_deformationspaces} will be established by showing that any solution $(a + bdt, \phi)$ satisfies
\[ \frac{\partial a}{\partial t} = 0, \quad \frac{\partial \phi}{\partial t} = 0, \quad b = 0. \]
This is done in the same way as in the proof of Theorem \ref{thm:invariance}. First,  apply $\partial / \partial t$ to the first two equations, then get rid of the terms $\partial \phi / \partial t$, $\partial b / \partial t$, and $\partial a^{1,0} / \partial t$.  This results in 
\[ 
- \frac{\partial^2 \phi}{\partial t^2} + 2 \partial_{AB}^* \partial_{AB} \phi - i \Im \langle \Psi, \phi \rangle \Psi + d^* a \cdot \Psi + 2 (\partial^* a^{1,0} ) \Psi + \sqrt{2} \sigma \mu_{\C}(\Psi, \phi) \Psi = 0, \]
\[ 
- \frac{\partial^2 a^{1,0}}{\partial t^2} +  \partial \partial^* a^{1,0} +  i \partial \Im \langle \Psi, \phi \rangle + \sqrt{2} \mu_{\C} ( \Psi, \sigma \partial \phi) + \sqrt{2} \mu_{\C} ( \Psi, \sigma a^{1,0} \Psi ) = 0. 
  \]
  Take the real $L^2$--product of the first equation with $\phi$ and the second equation with $a^{1,0}$. Integrating by parts as in the proof of Theorem \ref{thm:invariance}, we obtain
\[ \left\| \frac{ \partial \phi}{ \partial t} \right\|^2_{L^2} + \left\| \frac{ \partial a}{ \partial t} \right\|^2_{L^2} + 2 \left\| \partial_{AB} \phi + a^{1,0} \Psi \right\|_{L^2}^2 + \left\| - d^*a  + i \Im \langle \Psi, \phi \rangle \right\|^2_{L^2} + \sqrt{2} \left\| \mu_{\C}(\Psi, \phi) \right\|_{L^2}^2 = 0. \]
We have used identity  \eqref{eqn:moment} to relate $\mu_{\C}$ to the inner product. Thus, we have proved that $b = 0$, $\phi$ and $a$ are pulled-back from $\Sigma$ and satisfy
\begin{equation} \label{eqn:deformation}
\left\{
\begin{array}{l}
 \partial_{AB} \phi + a^{1,0} \Psi  = 0, \\
\mu_{\C}(\Psi, \phi) = 0, \\
 \ast da + 2\mu_{\R} ( \Psi, \phi) = 0, \\
 - d^* a -  i \Im\langle \Psi, \phi \rangle = 0.
\end{array}
\right.
\end{equation}
Recall that with the conventions of subsection \ref{subsec:invariance} we identify $\overline{\Psi}$ with a pair $(\alpha, \beta)$. After a conjugation equation \eqref{eqn:deformation} translates to the following equation for $a$ and $\overline{\phi} = (u,v)$
\begin{equation} \label{eqn:deformation2}
\left\{
\begin{array}{l}
\del_{AB} u + a^{0,1} \alpha = 0, \\
\del_{AB} v - a^{0,1} \beta = 0, \\
\alpha v + u \beta = 0, \\
 \ast i da + 2 \Re \langle \alpha, u \rangle - 2 \Re \langle \beta, v \rangle = 0, \\
 - d^* a -  i \Im\langle \alpha, u \rangle - i \Im \langle \beta, v \rangle = 0.
\end{array}
\right.
\end{equation}
This is the linearisation of \eqref{eqn:nsw-invariant} together with the Coulomb gauge fixing condition. We conclude that \eqref{eq_deformationspaces} holds and $\cM^*$ is isomorphic to $\cM_{\Sigma}$ as real analytic spaces.

\begin{step} $\cM_{\Sigma}$ is isomorphic to $\cM_{\hol}$.
\end{step}
The proof is similar to that of \cite[Theorem 2.6]{friedman-morgan}, so we only outline the argument. As before, the main point is to show an isomorphism of the deformation spaces for  $\cM_{\Sigma}$ and $\cM_{\hol}$. The former is given by \eqref{eqn:deformation2} and the latter consists of solutions of the first three equations together with a choice of a local slice for the action of $\cG^c(\Sigma)$. The Lie algebra of $\cG^c(\Sigma)$ splits is the direct sum of $C^{\infty}(\Sigma, \R)$ and the Lie algebra of $\cG(\Sigma)$. Under this splitting, we can choose a slice of $\cG^c(\Sigma)$--action imposing the standard Coulomb gauge condition for $\cG(\Sigma)$, which is the last equation of \eqref{eqn:deformation2}, together with a choice of a slice for the action of $C^{\infty}(\Sigma,\R)$:
\[ e^f (A, \alpha, \beta) = (A + \partial f - \del f, e^f \alpha, e^{-f} \beta). \]
The linearisation of this action at $(A,\alpha,\beta)$ is 
\begin{equation} \label{eqn:linearisedmap}
 f \mapsto (- \del f + \partial f, f \alpha, -f \beta). 
 \end{equation}
A local slice for the action of $C^{\infty}(\Sigma,\R)$ can be obtained from any subspace of 
\[\Omega^1(i\R) \oplus \Gamma(E^* \otimes L \otimes K^{1/2}) \oplus \Gamma(E \otimes L^* \otimes K^{1/2})\]
 which is complementary to the image of \eqref{eqn:linearisedmap}. Hence, to show that the deformation spaces of $\cM_{\Sigma}$ and $\cM_{\hol}$ are  isomorphic it is enough to prove that the subspace given by 
\[ i \ast da + 2 \Re \langle \alpha, u \rangle - 2 \Re \langle \beta, v \rangle = 0 \]
is complementary to the image of \eqref{eqn:linearisedmap}. In other words, we need to know that for any triple $(a,u,v)$ there is a unique function $f \in C^{\infty}(\Sigma,\R)$ such that
\[ \begin{split} 0 &= i \ast d(a - \del f + \partial f) + 2 \Re\langle \alpha, u + f\alpha \rangle - 2 \Re \langle \beta, v - f \beta \rangle \\ 
&= \left\{ \Delta + 2( | \alpha |^2 + | \beta |^2) \right\} f + i \ast da + 2 \Re \langle \alpha, u \rangle - 2 \Re \langle \beta, v \rangle.
\end{split}
 \] 
This is true because $(\alpha, \beta) \neq 0$ and so the operator $\Delta + 2 ( |\alpha|^2 + | \beta|^2)$ is invertible on $C^{\infty}(\Sigma,\R)$. In the same way as in \cite[Theorem 2.6]{friedman-morgan} we conclude that \eqref{eqn:deformation2} provides a local Fredholm model for both $\cM_{\Sigma}$ and $\cM_{\hol}$ and so the two spaces have isomorphic analytic structures.
\qed

\section{A tale of two compactifications}
\label{sec_compactifications}
The goal of this section is to define natural compactifications of $\cM^*$ and $\cM_{\hol}$ and to extend the isomorphism $\cM^* \cong \cM_{\hol}$ to a homeomorphism between these compactifications, thus completing the proof of Theorem \ref{thm:homeo0}. We assume $d - \tau < 0$ so in particular $\cM = \cM^*$. The discussion can be easily adapted to the cases $d - \tau = 0$ and $d - \tau > 0$.

\subsection{A complex-geometric compactification}
$\cM_{\hol}$ has a natural compactification analogous to the one described in \cite{bryan-wentworth}. Consider the subspace $S \subset \cC_{\Sigma} \times \C$ given by
\[ S := \left\{ (A, \alpha, \beta, t) \ | \  (A,\alpha,\beta) \textnormal{ satisfies equations \eqref{eqn:holomorphic}}, \alpha \neq 0, \textnormal{ and } (\beta, t) \neq (0,0) \right\}. \]
The group $\cG^c(\Sigma) \times \C^*$ acts freely on $S$ by the standard action of the first factor on $\cC_{\Sigma}$ and
\[ \lambda(A, \alpha, \beta, t) = (A, \alpha, \lambda \beta, \lambda t) \quad \textnormal{for } \lambda \in \mathbb{C}^{\times}. \]

\begin{defn}
We define $\overline{ \cM_{\hol}}$ to be the quotient of $S$ by $\cG^c(\Sigma) \times \C^*$. 
\end{defn}

This is analogous to compactifying $\C^N$ by $\mathbb{CP}^N$ which is the quotient of $(\C^{N} \times \C) \setminus \{ (0,0) \}$ by the free action of $\C^*$; in fact, $\overline{\cM}_{\hol}$ is obtained by applying this construction fibrewise.

\begin{defn} \label{defn:vortexmoduli}
Let $\cN$ be the subspace of $\cM_{\hol}$ consisting of triples of the form $(A, \alpha, 0)$. Equivalently, $\cN$ is the space of $\cG^c(\Sigma)$--orbits of pairs $(A,\alpha)$ satisfying $\del_{AB} \alpha = 0$ and $\alpha \neq 0$. 
\end{defn}

We will see momentarily that $\cN$ is compact. The natural projection $(A, \alpha, \beta) \mapsto (A, \alpha)$ induces a surjective map $\pi \colon \cM_{\hol} \to \cN$. Let $[A, \alpha] \in \cN$ and denote by $\mathcal{L}_A$ the holomorphic structure on $L$ induced by $A$. The fibre $\pi^{-1}([A, \alpha])$ is the kernel of the homomorphism
\[
\begin{tikzcd}
 H^0(\Sigma, \cE_B^* \otimes \mathcal{L}_A \otimes K^{1/2}) \arrow{r}{\alpha} & H^0(\Sigma, K)
\end{tikzcd}
\]
given by pairing with $\alpha$. The compactification $\overline{\cM}_{\hol}$ is obtained by replacing each fibre $\ker \alpha$ with the projective space $\mathbb{P}( \ker \alpha \oplus \C)$ containing it.

\begin{prop}\label{prop:compactification}
The space $\overline{\cM}_{\hol}$ is metrisable, compact, and contains $\cM_{\hol}$ as an open dense subset. Moreover, the complex analytic structure on $\cM_{\hol}$ extends to a complex analytic structure on $\overline{\cM}_{\hol}$ with respect to which $\cM_{\hol}$ is Zariski open.
\end{prop}
\begin{proof}
It is clear that $\overline{\cM}_{\hol}$ is metrisable and $\cM_{\hol} \subset \overline{\cM}_{\hol}$ is open and dense. In order to show that $\overline{\cM}_{\hol}$ is compact, consider a sequence $[A_i, \alpha_i, \beta_i, t_i] \in \overline{\cM}_{\hol}$; we need to argue that there are sequences $h_i \in \cG^c(\Sigma)$ and $\lambda_i \in \C^*$ such that after passing to a subsequence $h_i \lambda_i (A_i, \alpha_i, \beta_i, t_i)$ converges smoothly in $S$. This is the content of Step $1$ in \cite[Proof of Theorem 2.2]{doan}.

Let $\cC_{\Sigma}^* \subset \cC_{\Sigma}$ be the subset of configurations $(A, \alpha, \beta)$ with $\alpha \neq 0$. The group $\cG^c(\Sigma)$ acts freely on $\cC_{\Sigma}^*$ with quotient $\mathcal{B}_{\Sigma}^*$, a complex Banach manifold. There is a holomorphic vector bundle $\mathcal{W} \to \mathcal{B}_{\Sigma}^*$ such that $\cM_{\hol}$ is the zero set of a holomorphic Fredholm section $\mathcal{S} \colon \mathcal{B}_{\Sigma}^* \to \mathcal{W}$. The zero set of such a section carries a natural complex analytic structure  \cite[Sections 4.1.3-4.1.4]{friedman-morgan}. The complex analytic structure on $\cM_{\hol}$ is extended to $\overline{\cM}_{\hol}$ by extending $\mathcal{S}$ to a Fredholm section whose zero set is $\overline{\cM}_{\hol}$. Replace $\cC_{\Sigma}^*$ by the subspace of $\cC_{\Sigma} \times \C$ consisting of quadruples $(A, \alpha, \beta, t)$ for which $\alpha \neq 0$ and $(\beta, t) \neq (0, 0)$. Let $\overline{\mathcal{B}}_{\Sigma}^*$ be the quotient of this space by the action of $\cG^c(\Sigma) \times \C^*$: it contains $\mathcal{B}_{\Sigma}^*$ as an open subset. Let $\overline{\mathcal{W}} \to \overline{\mathcal{B}}^*_{\Sigma}$ be the vector bundle obtained as the quotient of $\mathcal{W} \times \C^*$ by the lifted action of $\cG^c(\Sigma) \times \C^*$. There is a holomorphic Fredholm section $\overline{\mathcal{S}}$ extending $\mathcal{S}$ so that $\overline{\cM}_{\hol} = \overline{\mathcal{S}}^{-1}(0)$. When restricted to the open subset $\mathcal{B}_{\Sigma}^*$, this reduces to the construction of $\cM_{\hol}$ described above, so the inclusion $\cM_{\hol} \subset \overline{\cM}_{\hol}$ is compatible with the induced analytic structures. Moreover, $\overline{\cM}_{\hol} \setminus \cM_{\hol}$ is the intersection of $\overline{\mathcal{S}}^{-1}(0)$ with the analytic subset $\overline{\mathcal{B}}_{\Sigma}^*\setminus \mathcal{B}_{\Sigma}^*$ given by the equation $t = 0$. We conclude that  $\overline{\cM}_{\hol} \setminus \cM_{\hol}$ is an analytic subset of $ \overline{\cM}_{\hol}$, and so $\cM_{\hol}$ is Zariski open.
\end{proof}

\begin{cor} \label{cor:subspace}
$\cN \subset \cM_{\hol}$ is compact. Furthermore, $\cM_{\hol}$ is compact if and only if $\cN = \cM_{\hol}$.
\end{cor}
\begin{proof}
$\cN$ consists of equivalence classes $[A, \alpha, \beta]$ for which $\beta = 0$; it is compat by Step $1$ in \cite[Proof of Theorem 1.3]{doan}. If $\cN = \cM_{\hol}$, then $\cM_{\hol}$ is compact. To prove the converse, observe that if $\cM_{\hol}$ is non-compact, then $\overline{\cM}_{\hol} \setminus \cM_{\hol}$ is non-empty by Proposition \ref{prop:compactification}. On the other hand, $\overline{\cM}_{\hol} \setminus \cM_{\hol}$ consists of $\cG^c(\Sigma) \times \C^*$--orbits of the form $[A, \alpha, \beta, 0]$ with $\beta \neq 0$ so every element of $\cM_{\hol} \setminus \cN$ gives rise to an element of $\overline{\cM}_{\hol} \setminus \cM_{\hol}$.
\end{proof}

\subsection{A gauge-theoretic compactification}
For an arbitrary three-manifold a good compactification of $\cM$ is yet to be constructed---see \cite[Introduction]{doan} for a discussion of analytical difficulties involved in such a construction. 
%The main technical issue  is the formation of singularities described in subsection \ref{subsec:compactness}. One would wish to form the moduli space of Fueter sections with singularities. This is closely related to the work of Takahashi \cite{takahashi}, but more needs to be done when the singular set is not a smooth submanifold. Another problem is the convergence of Seiberg--Witten multi-monopoles to a Fueter section. So far it has only been shown in weak Sobolev norms, whereas for a meaningful compactification we need strong convergence (or equivalently $C^{\infty}_{\loc}$). 
For $Y = S^1 \times \Sigma$ we can overcome these obstacles thanks to a refined compactness theorem \cite[Theorems 1.4 and 2.2]{doan}.

\begin{rem}
For the remaining part of the paper we make the assumption that $E$ is an $\SU(2)$--bundle. In this case the description of Fueter section simplifies \cite[Proposition A.3]{haydys-walpuski}, \cite{haydys3}. The discussion below should easily generalise to the higher rank case.
\end{rem}

\begin{thm}[\cite{doan}] \label{thm:inv-compactness} 
If $(A_i,\Psi_i = (\alpha_i, \beta_i))$ is a sequence of solutions of \eqref{eqn:nsw-invariant} with $\| \Psi_i \|_{L^2} \to \infty$, then after passing to a subsequence and applying gauge transformations $(A_i, \Psi_i / \| \Psi_i \|_{L^2})$ converges in $C^{\infty}_{\loc}$ on the complement of a finite set $D = \{ x_1, \ldots, x_N \}$. The limiting configuration $(A, \Psi = (\alpha, \beta))$ is defined on $\Sigma \setminus D$ and satisfies
\begin{itemize}
\item $\| \Psi \|_{L^2} = 1$ and $| \Psi | > 0$ on $\Sigma \setminus D$,
\item $\del_{AB} \alpha = 0$, $\del_{AB} \beta = 0$, $\alpha\beta = 0$, and $| \alpha | = | \beta|$ on $\Sigma \setminus D$.
\item $A$ is flat on $\Sigma \setminus D$ and has holonomy contained in $\Z_2$. 
\item There are non-zero integers $a_1, \ldots, a_N$ such that $\sum_{k=1}^N a_k = 2d$ and
\[   \ast \frac{i}{2\pi} F_{A_i} \longrightarrow \frac{1}{2}\sum_{k=1}^N a_k \delta_{x_k} \]
in the sense of measures.
\item For each $k=1, \ldots , N$ we have  
\[ | \Psi (x) | = O( \mathrm{dist}(x_k,x)^{|a_k|/2} ). \]
\end{itemize}
\end{thm}

%$\cM$ will be compactified by adding to it the moduli space of limiting data appearing in Theorem \ref{thm:inv-compactness}. Two sets of such limiting data will be equivalent if they differ by a gauge transformation of a particular type.

\begin{defn}
Let $D \subset \Sigma$ be a finite set. We say that a gauge transformation in $\cG(\Sigma \setminus D)$ or $\cG^c(\Sigma \setminus D)$ is \emph{simple} if it has degree zero around each point of $D$. Denote by $\cG_0(\Sigma \setminus D)$ and $\cG^c_0(\Sigma \setminus D)$ the subgroups of simple gauge transformations.
\end{defn}

%The limiting data in Theorem \ref{thm:inv-compactness} involves a measure concentrated at $D$. In fact, any flat connection on $\Sigma \setminus D$ gives rise to such a measure.

\begin{defn} \label{defn:measure}
Let $D \subset \Sigma$ be a finite subset. With any flat connection $A \in \cA(\Sigma \setminus D, L)$ we associate a measure $i \ast F_A$ on $\Sigma$ as follows. For $x \in D$ let $B \subset \Sigma$ be a small disc centred at $x$ and not containing other points of $D$. In a unitary trivialisation of $\restr{L}{B}$ we have $A = d + a$ for a one-form $a \in \Omega^1( B \setminus \{ x \}, i \R)$. Denote
\[ q_x = \int_{\partial B} ia. \]
The measure $i \ast F_A$ is defined by
\[ i \ast F_A := \sum_{x \in D} q_x \delta_x. \]
One easily checks that $i\ast F_A$ is well-defined and invariant under simple gauge equivalences.
\end{defn}

\begin{defn} \label{defn:inv-fueter}
A \emph{limiting configuration} is a triple $(A, \Psi, D)$ comprising of a finite subset $D = \{ x_1, \ldots , x_N \} \subset \Sigma$, a connection $A \in \cA(\Sigma \setminus D,L)$, and a pair $\Psi = (\alpha, \beta)$ of nowhere-vanishing sections $\alpha \in \Gamma( \Sigma \setminus D, E^* \otimes L \otimes K^{1/2})$ and $\beta \in \Gamma(\Sigma \setminus D, E \otimes L^* \otimes K^{1/2})$ satisfying
\begin{itemize}
\item $\| \Psi \|_{L^2} = 1$ and $|\Psi| > 0$ on $\Sigma \setminus D$
\item $\del_{AB} \alpha = 0$, $\del_{AB} \beta = 0$, $\alpha\beta = 0$, and $| \alpha | = | \beta|$ on $\Sigma \setminus D$.
\item $A$ is flat on $\Sigma \setminus D$ and has holonomy contained in $\Z_2$.
\item There are non-zero integers $a_1, \ldots, a_N$ such that $\sum_{k=1}^N a_k = 2d$ and
\[   \ast \frac{i}{2\pi} F_{A} = \frac{1}{2}\sum_{k=1}^N a_k \delta_{x_k} \quad \textnormal{as measures}. \]
\item For each $k=1, \ldots , N$ we have  
\[ | \Psi (x) | = O( \mathrm{dist}(x_k,x)^{|a_k|/2} ). \]
\end{itemize}
$(A,\Psi,D)$ and $(A',\Psi',D')$ are \emph{simple gauge equivalent} if $D = D'$ and they differ by an element $u \in \cG_0(\Sigma \setminus D)$. Let $\mathcal{F}$ be the set of simple gauge equivalence classes of limiting configurations.
\end{defn}

The space $\mathcal{F}$ can be equipped with a natural topology in which a sequence $[A_i, \Psi_i, D_i]$ converges to $[A,\Psi,D]$ if and only if $i \ast F_{A_i'} \to i \ast F_A$ weakly as measures and after applying a sequence in $\cG_0(\Sigma \setminus D)$ we have $A_i \to A$ and $\Psi_i \to \Psi$ in $C^{\infty}_{\loc}$ on $\Sigma \setminus D$. 

\begin{defn} \label{defn:limiting}
Let $[A,\Psi,D]$ be an equivalence class in $\mathcal{F}$. For $\epsilon > 0$, $\delta > 0$, an integer $k \geq 0$, and a continuous function $f \colon \Sigma \to \R$ we define $V_{\epsilon,\delta,k,f}(A,\Psi,D) \subset \mathcal{F}$ to be the set of the elements of $\mathcal{F}$ which have a representative $(A',\Psi',D')$ satisfying
\begin{itemize}
\item $D' \subset D_{\epsilon}$ where $D_{\epsilon} :=   \{ x \in \Sigma \ | \ \mathrm{dist}(x, D) < \epsilon \}$,
\item $\| A' - A \|_{C^k(\Sigma \setminus D_{\epsilon})} < \delta$,
\item $\| \Psi' - \Psi \|_{C^k(\Sigma \setminus D_{\epsilon})} < \delta$. 
\item $\left| \int_{\Sigma} (i \ast F_{A'}) f - \int_{\Sigma} ( i \ast F_A) f \right| < \delta$.
\end{itemize}
\end{defn}

\begin{lem} \label{lem:fuetertopology}
The family of subsets
\[ \left\{ V_{\epsilon,\delta,k,f}(A,\Psi,D) \right\} \]
forms a base of a Hausdorff topology on $\mathcal{F}$. 
\end{lem}
The proof is a simple application of \cite[Proposition 2.3.15]{donaldson-kronheimer}. 
The next step is to combine $\cM$ and $\mathcal{F}$ into one topological space. For this purpose it is convenient to identify points of $\cM$ with gauge equivalence classes of triples $(A, \Psi, t)$, where $t \in (0, \infty)$, $\| \Psi \|_{L^2} = 1$, and
\begin{equation} \label{eqn:nsw-scaled} 
\left\{
\begin{array}{l}
\D_{AB} \Psi = 0, \\
t^2 F_A = \mu(\Psi). 
\end{array} \right.
\end{equation}
The map $(A, \Psi, t) \mapsto (A, t^{-1} \Psi)$ gives a homeomorphism between the space of such classes and the moduli space $\cM^*$. Recall that in our setting there are no reducibles so there is no boundary at $t \to \infty$. The boundary at $t \to 0$ is obtained by gluing in the space of limiting configurations.

\begin{defn}
Let $[A,\Psi,D]$ be an equivalence class in $\mathcal{F}$. For $\epsilon > 0$, $\delta > 0$, an integer $k \geq 0$, and a continuous function $f \colon \Sigma \to \R$ we define
\[ W_{\epsilon, \delta, k, f}(A,\Psi,D) \subset \cM \]
to be the set of the elements of $\cM$ that have a representative $(A', \Psi', t)$ satisfying
\begin{itemize} 
\item $t < \delta$,
\item $\| A' - A \|_{C^k(\Sigma \setminus D_{\epsilon})} < \delta$,
\item $\| \Psi' - \Psi \|_{C^k(\Sigma \setminus D_{\epsilon})} < \delta$,
\item $\left| \int_{\Sigma} ( i \ast F_{A'})f - \int_{\Sigma} (i \ast F_A)f \right| < \delta$.
\end{itemize}
\end{defn}

\begin{defn} \label{defn:compactifiedmoduli}
The \emph{compactified moduli space} is
\[ \overline{\cM} \coloneqq \cM \cup \mathcal{F} \]
equipped with the topology whose basis are the subsets
\[ \overline{W}_{\epsilon,\delta,k,f}(A,\Psi,D) \coloneqq W_{\epsilon, \delta, k,f}(A,\Psi,D)  \cup V_{\epsilon, \delta, k,f}(A,\Psi,D). \]
\end{defn}

%We can now endow the compactified moduli space with a Hausdorff topology with respect to which $\cM$ is an open subset. As before, we have
%The proof of the lemma below is almost the same as that of Lemma \ref{lem:fuetertopology}.

\begin{lem} \label{lem:hausdorff}
Let $\cU$ be a base of the topology on $\cM$. The family of subsets 
\[ \left\{ \overline{W}_{\epsilon,\delta,k}(A,\Psi,D) \right\} \cup \cU \]
 forms a base of a Hausdorff topology on $\overline{\cM}$. 
\end{lem}

%\begin{thm} \label{thm:compactification}
%The isomorphism $\cM \cong \cM_{\hol}$ extends to a homeomorphism of the compactifications $\overline{\cM} \cong \overline{\cM}_{\hol}$. 
%\end{thm}
%
%As a consequence of Proposition \ref{prop:compactification} and Theorem \ref{thm:homeo0} we obtain some topological information about the compactified moduli space. 
%
%\begin{cor} \label{cor:compactification}
%$\overline{\cM}$ is metrisable, compact, and contains $\cM$ as an open dense subset. 
%\end{cor}

%\begin{rem} \label{rem:topology}
%It would be interesting to prove Corollary \ref{cor:compactification} directly from the definition of the compactified moduli space. Showing metrisability should be similar to the case of the Uhlenbeck compactification of the instanton moduli space \cite[Section 4.4.1]{donaldson-kronheimer}. To prove that $\overline{\cM}$ is sequentially compact we use Theorem \ref{thm:inv-compactness}. As stated it implies that the closure of $\cM$ is contained in $\overline{\cM}$. In order to show the compactness of $\mathcal{F}$ we need to extend Theorem \ref{thm:inv-compactness} so that it deals with sequences of Fueter sections. Such an extension is implicit in the proofs of Proposition \ref{prop:compactification} and Theorem \ref{thm:homeo0}. Finally, to show that $\cM$ is dense one needs a gluing theorem which would guarantee that any limiting configuration is the limit of a sequence of solutions of \eqref{eqn:nsw-invariant}.  See \cite{mazzeo-et.al} for a related result for the Hitchin equations. 
%\end{rem}

\subsection{A homeomorphism at infinity}
The main ingredient in the proof of Theorem \ref{thm:homeo0} is

\begin{prop} \label{prop:homeofueter}
The spaces $\overline{\cM}_{\hol} \setminus \cM_{\hol}$ and $\mathcal{F}$ are homeomorphic. 
\end{prop}

The proof of Proposition is preceded by auxiliary results about limiting configurations. The first of them is a complex-geometric analogue of the statement that a limiting configuration induces a flat connection with $\Z_2$ holonomy.

\begin{lem} \label{lem:varphi}
Let $(A, \alpha, \beta)$ be a solution of \eqref{eqn:holomorphic} with $\alpha \neq 0$ and $\beta \neq 0$. Denote by $\mathcal{L}$ the holomorphic line bundle $(L, \del_A)$. Let $D_1$ and $D_2$ be the zero divisors of $\alpha$ and $\beta$ respectively, and $\cO(D_1 - D_2)$ the holomorphic line bundle associated to $D_1 - D_2$. There is a canonical holomorphic isomorphism  
\[ \varphi_{\alpha \beta} \colon \cO(D_1 - D_2) \longrightarrow \mathcal{L}^2. \]
\end{lem}
\begin{proof}
Recall that $\alpha$ is a holomorphic section of $\cE \otimes \mathcal{L} \otimes K^{1/2}$ and $\beta$ is a holomorphic section of $\cE^* \otimes \mathcal{L}^* \otimes K^{1/2}$. Since $\alpha \beta = 0$ and the rank of $\cE$ is two, we have the short exact sequence
\begin{equation} \label{eqn:varphi-exact}
\begin{tikzcd} 
0 \ar{r} & \mathcal{L}^{-1} \otimes K^{-1/2} \otimes \cO(D_1) \ar{r}{\alpha} & \cE \ar{r}{\beta} & \mathcal{L}^{-1} \otimes K^{1/2} \otimes \cO(-D_2) \ar{r} & 0. 
\end{tikzcd}
\end{equation}
The associated isomorphism of the determinant line bundles is
\[ \mathcal{L}^{-2} \otimes \cO(D_1 - D_2) \cong \det \cE \cong \cO, \]
where the last isomorphism follows from the fact that $\cE$ is a holomorphic $SL(2,\C)$-bundle. Tensoring both sides with $\mathcal{L}^2$, we obtain the desired isomorphism $\varphi_{\alpha \beta}$. It is canonically determined by $\alpha$ and $\beta$.
\end{proof}

The lemma below provides an upper bound on the number of components, counted with multiplicities, of the singular set of a limiting configuration. 

\begin{lem} \label{lem:bounded}
There exists $M \geq 0$, depending only on the holomorphic bundle $\cE$, with the following significance. If $(\mathcal{L}, \alpha, \beta)$ is a triple as in Lemma \ref{lem:varphi}, then 
\[ \deg D_1 + \deg D_2 \leq M. \]
\end{lem}
\begin{proof}
Tensoring exact sequence \eqref{eqn:varphi-exact} with $K^{1/2}$ we see that $\mathcal{L}^{-1} \otimes \cO(D_1)$ is a holomorphic subbundle of $\cE \otimes K^{1/2}$. It is an elementary fact that the degrees of line subbundles of a given holomorphic bundle are bounded above \cite[Corollary 10.9]{mukai}.  In fact, if $\mathcal{F}$ is a holomorphic vector bundle and $A \subset \mathcal{F}$ a line subbundle, then
\begin{equation} \label{eqn:bound}
\deg A \leq h^0(\Sigma, \mathcal{F}) + 2g(\Sigma)-2
\end{equation}
where $g(\Sigma)$ is the genus of $\Sigma$. (We will use this bound later.) Thus, we have an upper bound on the degree of $\mathcal{L}^{-1} \otimes \cO(D_1)$. On the other hand, $\mathcal{L}^2$ is isomorphic to $\cO(D_1 - D_2)$, so
\[ 
\deg \left( \mathcal{L}^{-1} \otimes \cO(D_1) \right) = - \deg \mathcal{L} + \deg D_1 = \frac{1}{2} \left( \deg D_1 + \deg D_2 \right),
\]
which proves the lemma.
\end{proof}

The next result will be useful in proving the convergence of measures.

\begin{lem} \label{lem:measures}
Let $f \colon \Sigma \to \R$ be a continuous function, $\gamma > 0$, and $D \subset \Sigma$ a finite subset. Then there exist $\epsilon> 0$ and $\delta > 0$ with the following property. Suppose that $D' \subset \Sigma$ is another finite subset, and $A$ and $A'$ are two flat connections over $\Sigma \setminus D$ and $\Sigma \setminus D'$ respectively, satisfying
\begin{itemize}
\item $D' \subset D_{\epsilon}$,
\item $\| A' - A \|_{C^0(\Sigma \setminus D_{\epsilon})} < \delta$,
\item the measures $\ast i F_A$ and $\ast i F_{A'}$ have integer weights.
\end{itemize}
Then we have
\[ \left| \int_{\Sigma} ( i\ast F_{A'} )f - \int_{\Sigma}  (i \ast F_A)f \right| \leq \gamma \| i \ast F_{A'} \|. \]
where $\| i \ast F_{A'} \|$ is the \emph{total variation} of the measure $i \ast F_{A'}$ given by
\[  \| i \ast F_{A'} \|  = \sum_{x \in D'} | q_x | \]
 for   $i \ast F_{A'} = \sum_{x \in D'} q_x \delta_x$ . 
\end{lem}

\begin{proof} 
Let $D = \{ x_1, \ldots, x_N \}$ and $a_1, \ldots, a_N$ be the integer weights of the measure $i \ast F_A$ as in Definition \ref{defn:measure}. Suppose that $\epsilon$ is small enough so that the discs $B_i$ of radius $\epsilon$ centred $x_i$ are pairwise disjoint.  Partition the set $D'$ into disjoint subsets $D_1', \ldots, D_N'$ consisting of points within $\epsilon$--distance from $x_1, \ldots, x_N$ respectively. For each $i$ choose small disjoint discs $E_{i1}, E_{i2}, \ldots $ centred at points of $D_i'$ and contained in $B_i$. Let $b_{i1}, b_{i2} \ldots $ be the weights of the points in $D_i'$ in the measure $i \ast F_{A'}$. 

In a unitary trivialisation of $L$ over each $B_i$ we have
\[ A = d + a \qquad \textnormal{and} \qquad A' = d + a', \]
where one-form $a$ is defined on $B_i \setminus \{ x_i \}$ and $a'$ is defined on $B_i \setminus D_i'$. By the hypothesis of the lemma $\| a - a' \|_{C^0(\partial B_i)} < \delta$. Thus, for sufficiently small $\delta$,
\[ \left| a_i - \sum_j b_{ij}\right| =\left| \int_{\partial B_i} i a - \sum_j \int_{\partial E_{ij}} ia' \right| = 
 \ \left| \int_{\partial B_i} ia - \int_{\partial B_i} ia' \right| < 1. \]
 Since all numbers $a_i$, $b_{ij}$ are integers, so we conclude that
\[ a_i = \sum_j b_{ij}. \]
For each $i$ denote the points of $D_i'$ by $\{ x_{ij} \}$. Then
\[ \begin{split}
 \left| \int_{\Sigma} ( i\ast F_{A'} )f - \int_{\Sigma}  (i \ast F_A)f \right| &=  \left| \sum_{i} a_i f(x_i) - \sum_{i j} b_{ij} f(x_{ij}) \right| \\
 & \leq \sum_{ij} |b_{ij}| \left| f(x_i) - f(x_{ij}) \right|.
 \end{split}
 \]
By the continuity of $f$ we can choose $\epsilon$ sufficiently small so that 
\[ \sup_{x \in B_i} | f(x_i) - f(x) | \leq \gamma \]
for all $i = 1, \ldots, N$. Then
 \[   \left| \int_{\Sigma} ( i\ast F_{A'} )f - \int_{\Sigma}  (i \ast F_A)f \right|  \leq \gamma \| i \ast F_{A'}.\|  \qedhere \]
 \end{proof}

The last lemma allows us to extend a limiting configuration to a holomorphic section. The proof is a minor variation of \cite[Proof of Lemma 2.1]{doan}.

\begin{lem} \label{lem:extend}
Let $L$ be a Hermitian line bundle over the unit disc $B \subset \C$. Suppose that $A$ is a unitary connection on $\restr{L}{B \setminus \{ 0 \}}$ and $\varphi$ a section of $L$ over $B \setminus \{ 0 \}$ satisfying $\del_{A} \varphi = 0$ and $| \varphi | = 1$. Denote by $\deg \varphi$ the degree of $\restr{\varphi}{\partial B}$. Then
\begin{enumerate}
\item $F_A = 0$ on $B \setminus \{ 0 \}$ and $i \ast F_A =( 2\pi \deg \varphi )\delta_0$ as measures.
\item There exists a complex gauge transformation $h \colon B \setminus \{ 0 \} \to \C^*$ such that $h$ has degree zero around zero and in some trivialisation of $L$ around zero $h(A)$ is the trivial connection and $h\varphi = z^k$.
\end{enumerate}
\end{lem}

\begin{proof}[Proof of Proposition \ref{prop:homeofueter}]
Set $\mathcal{X} = \overline{\cM}_{\hol} \setminus \cM_{\hol}$. We will construct a continuous bijection $\mathcal{X} \to \mathcal{F}$. Since the domain is compact by Proposition \ref{prop:compactification} and the target space is Hausdorff by Lemma \ref{lem:fuetertopology}, such a map is necessarily a homeomorphism.

\setcounter{step}{0}
\begin{step} 
The construction of $\mathcal{X} \to \mathcal{F}$.
\end{step}
Let $[A, \alpha, \beta, t] \in \mathcal{X}$. By definition, $t=0$, $\alpha \neq 0$, $\beta \neq 0$, and \eqref{eqn:holomorphic} is satisfied. Let $D_1$, $D_2$ be the zero divisors of $\alpha$, $\beta$, respectively. We will interchangeably consider them as divisors or as subsets of $\Sigma$. Set $D = D_1 \cup D_2$. We claim that there is a simple complex gauge transformation $h \in \cG^c_0(\Sigma \setminus D)$ such that $(hA, h\alpha, h^{-1} \beta)$ is a limiting configuration. A necessary condition is
\begin{equation} \label{eqn:complexified}
 | h \alpha | = | h^{-1} \beta | \quad \textnormal{on } \Sigma \setminus D. 
 \end{equation}
 A transformation satisfying \eqref{eqn:complexified} exists since $\alpha$ and $\beta$ are both non-zero on $\Sigma \setminus D$ and we can set $h = \sqrt{ | \beta| / | \alpha| }$; any other choice of $h$ will differ from that one by an element of $\cG_0(\Sigma \setminus D)$. 

The map  $\mathcal{X} \to \mathcal{F}$ is defined by $[A,\alpha,\beta,0] \mapsto (A', \alpha', \beta') := (h(A), h\alpha, h^{-1} \beta)$. We need to show that $(A', \alpha', \beta')$ represents a class in $\mathcal{F}$. We clearly have 
\[ 
\left\{
\begin{array}{l}
\del_{A'B} \alpha' = 0, \\
\del_{A'B} \beta' = 0, \\
| \alpha' | = | \beta' |.
\end{array} \right. \]
Moreover, for $\Psi' := (\alpha', \beta')$
\[ | \Psi' | = \sqrt{ | \alpha' |^2 + | \beta' |^2 } = \sqrt{  | h|^2 |\alpha|^2 + |h|^{-2} | \beta |^2 } = \sqrt{ 2 | \alpha | | \beta |}. \]
Integrating over $\Sigma$ yields
\[ \| \Psi' \|_{L^2} = \sqrt{\int_{\Sigma} 2 | \alpha |  | \beta | }. \]
After rescaling $\beta$, which does not change the class of $[A,\alpha,\beta,0]\in\mathcal{X}$, we can assume $\| \Psi \|_{L^2} = 1$. We also see that  $| \Psi | > 0$ on $\Sigma \setminus D$ and in a neighbourhood of every $x \in D$ 
\begin{equation} \label{eqn:vanishing} | \Psi(y) | = O\left( \mathrm{dist}(y,x)^{\frac{\mathrm{ord}_x(\alpha) + \mathrm{ord}_x(\beta)}{2}} \right), 
\end{equation}
where $\mathrm{ord}_x(\alpha)$ and $\mathrm{ord}_x(\beta)$ denote the order of vanishing of $\alpha$ and $\beta$ at $x$. 

It remains to prove that $A'$ is flat and the measure $i \ast F_{A'}$ is as in Definition \ref{defn:inv-fueter}. Let $\varphi_{\alpha \beta} \colon \cO(D_1 - D_2) \to L^2$ be the $A$--holomorphic isomorphism from Lemma \ref{lem:varphi}. The construction of Lemma \ref{lem:varphi} is local, so we can also define an analogous map $\varphi_{\alpha' \beta'}$ corresponding to sections $\alpha'$ and $\beta'$. Since they are both nowhere vanishing, $\varphi_{\alpha' \beta'}$ is an $A'$--holomorphic isomorphism of the trivial bundle over $\Sigma \setminus D$ and $\restr{L^2}{\Sigma \setminus D}$; thus, $\varphi_{\alpha' \beta'}$ is a nowhere vanishing $A'$--holomorphic section of $L^2$ over $\Sigma \setminus D$. Moreover, on $\Sigma\setminus D$
\[ | \varphi_{\alpha' \beta'} | = \frac{ | \alpha' | }{ | \beta' |} = 1 \quad \textnormal{and} \quad \varphi_{\alpha' \beta'} = h^2 \varphi_{\alpha \beta} \]
By Lemma \ref{lem:extend}, the tensor product connection $A \otimes A$ on $L^2$ is flat; so $A$ itself is flat and for every $x \in D$ the weight of the measure $i/2\pi \ast F_A$ at $x$ is equal to half of the degree of $\varphi_{\alpha' \beta'}$ around $x$.  Since $h$ has degree  zero around each point $x \in D$, the degrees of $\varphi_{\alpha \beta}$ and $\varphi_{\alpha' \beta'}$ around $x$ agree. Denote this degree by $q_x \in \Z$. Since the zero divisor of $\varphi_{\alpha\beta}$ is $D_1 - D_2$, 
\[ \sum_{x \in D} q_x = \deg D_1 - \deg D_2 = \deg ( L^2) = 2d. \]

Finally, observe that for every $x \in D$ we have
\[ q_x  = \mathrm{ord}_x( \alpha )- \mathrm{ord}_x (\beta) \leq \mathrm{ord}_x (\alpha) + \mathrm{ord}_x (\beta). \]
Together with \eqref{eqn:vanishing} this shows that $(A', \alpha', \beta', D)$ is a limiting configuration. It is easy to check that if we replace $(A,\alpha,\beta,0)$ by a different quadruple in the same orbit of the $\cG^c(\Sigma) \times \C^*$--action, then the resulting limiting configurations are simple gauge-equivalent.

\begin{step}
$\mathcal{X} \to \mathcal{F}$ is injective.
\end{step}
Suppose that $(A_1, \alpha_1, \beta_1, 0)$ and $(A_2, \alpha_2, \beta_2, 0)$ give rise to limiting configurations that are simple gauge equivalent. In particular, they have the same singular set, $D$ say. Suppose that $\beta_1$ and $\beta_2$ are scaled so that
\[ \int_{\Sigma} 2 | \alpha_1 | | \beta_1 | = \int_{\Sigma} 2 | \alpha_2 | | \beta_2 | = 1. \]
Composing the simple gauge equivalence of the limiting configurations with complex gauge transformations satisfying \eqref{eqn:complexified}, we obtain $t \in \cG^c_0(\Sigma \setminus D)$ such that on $\Sigma \setminus D$
\[ A_2 = t(A_1), \quad \alpha_2 = t \alpha_1, \quad \beta_2 = t^{-1} \beta_1, \quad \varphi_{\alpha_2\beta_2} =  t^2 \varphi_{\alpha_1\beta_1}. \]
Even though $t$ is not defined at the points of $D$, the holomorphic data is. In particular,  $\varphi_{\alpha_1 \beta_1}$ and $\varphi_{\alpha_2 \beta_2}$ have zeroes at every point of $D$. Moreover, the zeroes are of the same order---this is equivalent to  the measures of the corresponding limiting configurations being equal. We conclude that $t$ is bounded around $D$. Since it is also $(A_1,A_2)$--holomorphic it extends to a holomorphic isomorphism between $(L, \del_{ A_1})$ and $(L, \del_{A_2})$ and so $(A_1, \alpha_1, \beta_1, 0)$ and $(A_2, \alpha_2, \beta_2, 0)$ give rise to the same point in $\mathcal{X}$ and the map $\mathcal{X} \to \mathcal{F}$ is injective.

\begin{step}
$\mathcal{X} \to \mathcal{F}$ is surjective.
\end{step}
 Let $(A', \Psi' = (\alpha', \beta'), D) \in \mathcal{F}$. We need to find $h \in \cG^c_0(\Sigma \setminus D)$ such that $(A, \alpha, \beta) \coloneqq (h(A'), h\alpha', h^{-1} \beta')$ extends smoothly to the whole of $\Sigma$. Furthermore, we should have $D = D_1 \cup D_2$ where $D_1$ and $D_2$ are the zero divisors of $\alpha$ and $\beta$ respectively. 

Let $\varphi_{\alpha' \beta'} \in \Gamma(\Sigma \setminus D, L^2)$ be as before. Then $\del_{A'} \varphi_{\alpha' \beta'} = 0$ and $| \varphi_{\alpha' \beta'}| = 1$ on $\Sigma \setminus D$. By Lemma \ref{lem:extend}, applied to the connection $A' \otimes A'$ and section $\varphi_{\alpha' \beta'}$, there exists $k \in \cG^c_0(\Sigma \setminus D, L)$ such that $C \coloneqq k(A' \otimes A')$ extends to a connection on a line bundle $T \to \Sigma$ and $k \varphi_{\alpha' \beta'}$ extends to a meromorphic section of $(T, \del_C)$. We claim that $T =L^2$ as unitary bundles, that $k = h^2$ for some $h \in \cG^c_0(\Sigma \setminus D)$, and that $C = A \otimes A$ for $A = h(A')$. This follows from the assumption on the measure $i \ast F_{A'}$ induced by the limiting configuration $(A', \alpha', \beta', D)$; by Lemma \ref{lem:extend} for every point $x \in D$ the meromorphic section $h^2 \varphi_{\alpha' \beta'}$ vanishes to the order $q_x$ defined by
\[ \frac{i}{2\pi} F_{A'} = \frac{1}{2} \sum_{x \in D} q_x \delta_x. \]
($x$ is a pole of order $| q_x|$ if $q_x < 0$.) The claim is then a consequence of the assumption $\sum_{x \in D} q_x = 2d = \deg( L^2)$ and the fact that $k$ has degree zero around the points of $D$. Thus, $A = h(A')$ extends. We need to show that $\alpha = h \alpha'$ and $\beta = h^{-1} \beta'$ extend. Observe that 
\[ | \alpha | | \beta | = | \alpha' | | \beta' | = \frac{1}{2} | \Psi' |^2, \]
where we have used $| \alpha' | = | \beta'|$ and $\Psi' = (\alpha', \beta')$. As a consequence, for every $x \in D$
\[ |\alpha(y) | | \beta(y) | = O\left( \mathrm{dist}(x,y)^{|q_x|} \right).\]
On the other hand, we have
\[ \frac{ | \alpha | }{ | \beta | } = | h^2 | \frac{ | \alpha' | } { | \beta' | } = \left| h^2 \varphi_{\alpha' \beta'} \right|, \]
so around $x \in D$
\[ \frac{ | \alpha (y) | }{ | \beta(y) | } = O \left( \mathrm{dist}(x,y)^{q_x} \right). \]
We conclude that
\[ | \alpha(y) | = O\left( \mathrm{dist}(x,y)^{\frac{ |q_x| + q_x}{2}} \right)\qquad \textnormal{and} \qquad  | \beta(y) | = O\left( \mathrm{dist}(x,y)^{\frac{ |q_x| - q_x}{2} } \right), \]
which shows that both $\alpha$ and $\beta$ are bounded over $\Sigma \setminus D$. Since they are also holomorphic, they extend to globally defined sections. Hence, $(A, \alpha, \beta, 0)$ represents a point in $\mathcal{X}$ corresponding to $(A', \alpha', \beta', D)$ under $\mathcal{X} \to \mathcal{F}$.

\begin{step}
$\mathcal{X} \to \mathcal{F}$ is continuous.
\end{step}
Suppose that $[A_i, \alpha_i, \beta_i] \to [A, \alpha, \beta]$ in $\mathcal{X}$. Let $(A_i', \alpha_i', \beta_i', D_i)$ and $(A', \alpha', \beta', D)$ be the corresponding points in $\mathcal{F}$. We will prove that $(A_i', \alpha_i', \beta_i', D_i)$ converges to $(A', \alpha', \beta', D)$ as limiting configurations. We easily check that the points of $D_i$ concentrate around $D$ and modulo simple gauge transformations
\[ A_i' \to A', \qquad \alpha_i' \to \alpha', \qquad \beta_i' \to \beta' \]
in $C^{\infty}_{\loc}$ on  $\Sigma \setminus D$. By Lemma \ref{lem:measures}, $i \ast F_{A_i'} \to i \ast F_{A'}$ as measures provided that the sequence of total variations $\| i \ast F_{A_i'} \|$ is bounded. $\| i \ast F_{A_i'} \|$ is up to a constant equal to the degree of $D_1^i + D_2^i$ where $D_1^i$ and $D_2^i$ are the zero divisors of $\alpha_i$ and $\beta_i$. By Lemma \ref{lem:bounded} this degree is bounded above. We conclude that $[A_i', \alpha_i', \beta_i', D_i] \to [A', \alpha', \beta', D]$ in $\mathcal{F}$.
\end{proof} 

\subsection{Proof of Theorem \ref{thm:homeo0}}
We construct a bijective map  $\overline{\cM}_{\hol} \to \overline{\cM}$ from the homeomorphisms $\cM_{\hol} \to \cM$ from Theorem \ref{thm:holomorphic} and $\overline{\cM}_{\hol} \setminus \cM_{\hol} \to \mathcal{F}$ from Proposition \ref{prop:homeofueter}. This map is continuous when restricted to $\cM$ and its complement. It remains to show that it is continuous; indeed,  $\overline{\cM}_{\hol}$ is compact by Proposition \ref{prop:compactification} and $\overline{\cM}$ is Hausdorff by Lemma \ref{lem:hausdorff}, so a continous bijection $\overline{\cM}_{\hol} \to \overline{\cM}$ is a homeomorphism.

Let $(A_i, \alpha_i, \beta_i, t_i)$ be a sequence representing points in $\cM_{\hol}$ and $(A_i', \Psi_i', t_i')$ the corresponding sequence of solutions of \eqref{eqn:nsw-scaled}. Suppose that $t_i \to 0$ and $(A_i, \alpha_i, \beta_i)$ converges in $C^{\infty}$ to $(A, \alpha, \beta)$ with $\alpha \neq 0$ and $\beta \neq 0$. This limit represents an element of $\overline{ \cM_{\hol}} \setminus \cM_{\hol}$ and thus corresponds to a limiting configuration $(A',\Psi', D)$. We need to show that after applying gauge transformations the sequence of Seiberg--Witten solutions $(A_i', \Psi_i', t_i')$ converges in the sense of Definition \ref{defn:compactifiedmoduli} to a limiting configuration which is simple gauge-equivalent to $(A', \Psi', D)$. By Theorem \ref{thm:inv-compactness}, the sequence converges and by Proposition \ref{prop:homeofueter}, the limiting configuration is simple gauge-equivalent to $(A',\Psi',D)$. This shows that the map $\overline{\cM}_{\hol} \to \overline{\cM}$ is continuous.
\qed

\section{Fueter sections and complex geometry} 
\label{sec_fuetersections}
The main result of this section will ensure the compactness of $\cM(g,\sigma)$ for any product metric $g$ on $Y=S^1\times\Sigma$ and $\sigma$ generic among the parameters pulled-back from $\Sigma$.
%We are still in the setup of the previous subsection, that is $Y = S^1 \times \Sigma$ with the product metric $g$,  and the background bundle $E$ and connection $B$ are pulled back from $\Sigma$. In this subsection we show that for a generic choice of such circle-invariant $B$ there do not exist Fueter sections in the sense of Definition \ref{defn:fueter}. Recall that the non-existence of Fueter sections is a crucial assumption in Theorem \ref{thm:moduli} which allows us to count Seiberg--Witten multi-monopoles.

\begin{thm} \label{thm:nofueter}
\leavevmode
\begin{enumerate}
  \item For every product Riemannian metric $g$ on $Y = S^1 \times \Sigma$ there is a residual subset $\cA(g) \subset \cA(\Sigma, E)$ with the property that if $B \in \cA(g)$, then there exist no Fueter sections with respect to $(g,B)$.
 \item Let $(g_t)_{t\in[0,1]}$ be a path of product metrics and $B_0 \in \cA(g_0)$, $B_1 \in \cA(g_1)$. For a generic path $(B_t)_{t\in[0,1]}$ in $\cA(\Sigma,E)$ connecting $B_0$ and $B_1$ there exist no Fueter sections with respect to $(g_t,B_t)$ for all $t \in [0,1]$.
\end{enumerate}
\end{thm}

Recall that $\cM_{\hol}$ contains a compact subspace $\cN$ consisting of holomorphic triples $(\mathcal{L}, \alpha, \beta)$ with $\beta = 0$. As a result of Theorems \ref{thm:inv-compactness}, \ref{thm:nofueter}, and Corollary \ref{cor:subspace}, we obtain

\begin{cor} \label{cor:nofueter} For a generic choice of $B \in \cA(\Sigma,E)$ we have 
\[ \overline{\cM}_{\hol} = \cM_{\hol} = \cN. \] 
\end{cor} 

\begin{cor}\label{cor:moduliempty}
If $d - \tau = 0$ then $\cM_{\hol}$ is empty for a generic choice of $B \in \cA(\Sigma, E)$. 
\end{cor}

Here is an outline of the proof of Theorem \ref{thm:nofueter}: first, we describe Fueter sections in terms of complex-geometric data on $\Sigma$. Next, we show that this data is described by a Fredholm problem of non-positive index given by the Riemann--Roch theorem. As a result, we can apply the Sard--Smale theorem to exclude the existence of such data for a generic choice of $B\in\cA(\Sigma,E)$.

\subsection{Circle-invariance of Fueter sections}

For the remaining part of the section we continue to assume that $g$ is a product metric on $Y = S^1\times\Sigma$ and $B \in \cA(\Sigma,E)$.

Let $(A,\Psi,Z)$ is a Fueter section as in Definition \ref{defn:fueter}. Suppose that $Z =S^1 \times D$ for $D \subset \Sigma$, and that $(A,\Psi)$ is pulled back from $\Sigma$. Then, as in subsection \ref{subsec:invariance}, we have $\Psi = (\alpha, \beta)$ where
\[ \alpha \in \Gamma(\Sigma \setminus D, E^* \otimes L \otimes K^{1/2}), \]
\[ \beta \in \Gamma(\Sigma \setminus D, E \otimes L^* \otimes K^{1/2}). \]
The Fueter equations $\D_{AB}\Psi =0$ and $\mu(\Psi)=0$ are equivalent to
\begin{equation} \label{eqn:fueterinvariant}
\left\{
\begin{array}{l}
\del_{AB} \alpha = 0, \quad \del_{AB} \beta = 0,\\
\alpha \beta = 0, \\
| \alpha | = | \beta |.\\
\end{array}
\right.
\end{equation}

%Next, we prove that every Fueter section is in fact gauge-equivalent to a solution of \eqref{eqn:fueterinvariant}. This is analogous to the circle-invariance statement for generalised Seiberg--Witten monopoles, see Theorem \ref{thm:invariance}. 

\begin{prop} \label{prop:fueterinvariance}
Let $(A,\Psi,Z)$ be a Fueter section as in Definition \ref{defn:fueter}. Then $Z = S^1 \times D$ for a finite subset $D \subset \Sigma$. Moreover, there is a gauge transformation $u \in \cG(Y \setminus Z)$ such that $u$ has degree zero around each component of $Z$ and $u(A, \Psi)$ is pulled-back from $\Sigma \setminus D$.
\end{prop}
\begin{proof}
The proof is similar to that of Theorem \ref{thm:invariance}. We use the notation from subsection \ref{subsec:invariance} and ignore the background connection $B$; the general proof is the same. 

\setcounter{step}{0}
\begin{step}
A Weitzenb\"ock formula.
\end{step}

Let $t$ be the coordinate on the $S^1$ factor of $S^1 \times \Sigma$. Unlike in the proof of Theorem \ref{thm:invariance}, we cannot put $A$ in a temporal gauge, even after pulling-back to $\R \times \Sigma$, because \emph{a priori} the singular set $Z$ could intersect the $t$--axis in a complicated way. However, we still have
\begin{equation} \label{eqn:diracfueter} 0 = \D_{A} \Psi = - \sigma \nabla_t \Psi +  \del_{A} \Psi, 
\end{equation}
where $\nabla_t = \nabla_A( \partial / \partial t)$ and $\del_A$ is the Dolbeault operator induced by $A$ on the $\{ t \} \times \Sigma$ slice. Let $\nabla_{\Sigma}$ be the part of the covariant derivative $\nabla_{A}$ in the $\Sigma$-direction. Since $A$ is flat, we have 
\[ 0 = \nabla_A^2 = \nabla_t \nabla_{\Sigma} + \nabla_{\Sigma} \nabla_t. \]
Applying $\sigma$ and $\nabla_t^* = - \nabla_t$ to \eqref{eqn:diracfueter} and using the above commutation relation, we obtain
\begin{equation} \label{eqn:diracfueter2}
 0 =  \nabla_t^* \nabla_t \Psi + \sigma \del_A \sigma \del_A \Psi =  \nabla_t^* \nabla_t \Psi + \del_A^* \del_A \Psi. 
 \end{equation}
 
\begin{step}
Integration by parts; the circle-invariance of $Z$.
\end{step}
We want to integrate \eqref{eqn:diracfueter2} by parts to conclude $\nabla_t \Psi = 0$ and $\del_A \Psi = 0$. The equality holds only on  $Y \setminus Z$, so we need to use a cut-off function. Let $f \colon \R \to [0,1]$ be smooth and such that
\[ \left\{
\begin{array}{ll}
f = 0 & \textnormal{on } (-\infty, 0], \\
f=1 & \textnormal{on } [1, \infty)
\end{array}
\right.\]
For every $\epsilon > 0$ we define the cut-off function $\chi_{\epsilon} \colon Y\to [0,1]$ by
\[ \chi_{\epsilon}(x) = f\left( \frac{| \Psi(x) | - \epsilon}{\epsilon} \right). \] 
Let $Z_{\epsilon}$ be the subset of points in $Y$ satisfying $| \Psi(x) | < \epsilon$. We have 
\[ \left\{
\begin{array}{ll}
\chi_{\epsilon} = 0 & \textnormal{on }  Z_{\epsilon}, \\
\chi_{\epsilon} =1 & \textnormal{on } Y \setminus Z_{2\epsilon}
\end{array}
\right.\]
and $\chi_{\epsilon}$ is smooth on $Y$. Take the inner product of \eqref{eqn:diracfueter2} with $\chi_{\epsilon}^2 \Psi$ and integrate by parts:
\begin{equation} \label{eqn:diracfueter3}
 \begin{split} 0 &= \int_{Y} | \nabla_t( \chi_{\epsilon} \Psi ) |^2 + \int_Y | \del_A ( \chi_{\epsilon} \Psi ) |^2 - \int_Y \left( | \partial_t \chi_{\epsilon} |^2 + | \del \chi_{\epsilon} |^2  \right)| \Psi |^2  \\
 & \geq \int_{Y} | \nabla_t( \chi_{\epsilon} \Psi ) |^2 + \int_Y | \del_A ( \chi_{\epsilon} \Psi ) |^2 - 2 \int_Y | d \chi_{\epsilon} |^2 | \Psi |^2. 
 \end{split}
 \end{equation}
We need to show that the last term becomes arbitrarily small as $\epsilon$ tends to zero. By definition, $| \Psi | \leq 2\epsilon$ on $Z_{2 \epsilon}$. Let $P_{\epsilon} = Z_{2\epsilon} \setminus Z_{\epsilon}$. By Kato's inequality
\begin{align*}
 \int_Y | \Psi |^2 | d \chi_{\epsilon} |^2 
 &\leq \int_{P_{\epsilon}} | \Psi |^2 \frac{\left| f' \left( \frac{ | \Psi(x) | - \epsilon}{\epsilon} \right) \right|^2}{\epsilon^2} | \nabla_A \Psi |^2 \\
 & \leq 4 \| f \|^2_{C^1} \int_{P_{\epsilon}} | \nabla_A \Psi |^2 \\
 & \leq C \vol(P_{\epsilon}) \| \nabla_A \Psi \|_{L^2(M\setminus Z)}^2.
\end{align*}
The right-hand side converges to zero as $\epsilon \to 0$ since $| \nabla_A \Psi |^2$ is integrable, $Z = \cap_{\epsilon >0} Z_{\epsilon}$, and $\vol(Z) = 0$ by Taubes \cite[Theorem 1.3]{taubes2}. Taking $\epsilon \to 0$ in  \eqref{eqn:diracfueter3}, we conclude that on $Y \setminus Z$
\[ \nabla_t \Psi = 0 \qquad \textnormal{and} \qquad \del_A \Psi = 0. \]
In particular, 
\[ \partial_t | \Psi |^2 = 2 \langle \nabla_t \Psi, \Psi \rangle = 0 \]
so $|\Psi|$ is invariant under the circle action on $Y \setminus Z$. It is also continuous on the whole of $Y$ and $| \Psi |^{-1}(0) = Z$, so  $Z$ is necessarily of the form $S^1 \times D$ for a proper subset $D \subset \Sigma$. 

\begin{step}
$(A,\Psi)$ is pulled-back from $\Sigma\setminus D$.
\end{step}
We put $A$ in a temporal gauge over $S^1 \times (\Sigma \setminus D)$  as in the proof of Theorem \ref{thm:invariance}. The gauge transformation \eqref{eqn:temporal} used to do that is the exponential of a smooth function $\Sigma \setminus D \to i \R$ when restricted to each slice $\{ t \} \times (\Sigma \setminus D)$; thus, it has degree zero around the components of $Z$. The same argument as in the proof of Theorem \ref{thm:invariance} shows that $\restr{L}{Y \setminus Z}$ is pulled back from a bundle on $\Sigma \setminus D$ and $(A,\Psi)$ is pulled-back from a configuration on $\Sigma\setminus D$ satisfying \eqref{eqn:fueterinvariant}.

\begin{step}
 $D$ is a finite set of points.
\end{step}
It is enough to show that $D$ is locally finite. Suppose that $\Sigma$ is a unit disc and that $L$ and $E$ are trivial. The complement $\Sigma \setminus D$ is a non-compact Riemann surface and 
 $(L, \del_A)$ defines a holomorphic line bundle over $\Sigma\setminus D$ which is necessarily trivial \cite[Theorem 30.3]{forster}. Thus, there is $h \in \cG^c(B \setminus D)$ such that $h(A)$ agrees with the product connection on the trivial bundle, and $h \alpha$ and $h^{-1} \beta$ correspond to holomorphic maps $\Sigma\setminus D \to \C^2$. Let $\gamma = (h \alpha) \otimes (h^{-1} \beta)$; it is a holomorphic map $\Sigma \setminus D \to\C^2 \otimes \C^2 = \C^4$ satisfying
 \[ | \gamma | = | \alpha | | \beta | = \frac{1}{2} | \Psi |^2, \]
 so $D$ is the zero set of $|\gamma|$. Thus, $\gamma$ is continuous on $\Sigma$ and holomorphic on $\Sigma \setminus D$. By a theorem of Rad\'o \cite[Theorem 12.14]{rudin}, $\gamma$ is holomorphic on $\Sigma$ and so $D=\gamma^{-1}(0)$ is locally finite.
\end{proof}

\subsection{A holomorphic description of Fueter sections}

%Finally, we show that Fueter sections over $Y = S^1 \times \Sigma$ correspond to certain holomorphic data over $\Sigma$.

\begin{prop} \label{prop:meromorphic}
If $(A,\Psi,Z)$ is a Fueter section, with $\Psi = (\alpha,\beta)$ and $Z = S^1\times D$ as in Proposition \ref{prop:fueterinvariance}, then there exist $h \in \cG^c_0(\Sigma \setminus D)$ and divisors $D_1$, $D_2$ such that
\begin{enumerate} 
\item $D = D_1 \cup D_2$ as sets and the divisor $D_1 + D_2$ is effective,
\item $\tilde{A} := h(A)$ extends to a unitary connection on a line bundle over $\Sigma$, not necessarily isomorphic to $L$, defining a holomorphic line bundle $\mathcal{L} \to \Sigma$,
\item sections $\tilde{\alpha} = h\alpha$ and $\tilde{\beta} = h^{-1}\beta$ extend to holomorphic sections that fit into the short exact sequence
\begin{equation}\label{eqn:fueterexact}
\begin{tikzcd} 
0 \ar{r} & \mathcal{L}^{-1} \otimes K^{-1/2} \otimes \cO(D_1) \ar{r}{\tilde{\alpha}} & \cE \ar{r}{\tilde{\beta}} & \mathcal{L}^{-1} \otimes K^{1/2} \otimes \cO(-D_2) \ar{r} & 0. 
\end{tikzcd}
\end{equation}
\end{enumerate}
Conversely, every set of holomorphic data $(\tilde{A}, \tilde{\alpha}, \tilde{\beta}, D_1, D_2)$ satisfying conditions $(1)$, $(2)$, $(3)$ can be obtained from a Fueter section $(A,\Psi,Z)$ in this way.
\end{prop}
\begin{proof}
This is similar to Step 3 in the proof of Proposition \ref{prop:homeofueter}. Using Lemma \ref{lem:extend}, we find $h \in \cG^c_0(\Sigma \setminus D, L)$ such that $\tilde{A} = h(A)$ extends yielding a holomorphic line bundle $\mathcal{L}$, say, and $h^2 \varphi_{\alpha \beta}$ extends to a meromorphic section of $\mathcal{L}^2$. Let $\tilde{\alpha} = h \alpha$ and $\tilde{\beta} = h^{-1} \beta$. Then 
\[ \frac{ | \tilde{\alpha} |}{ | \tilde{\beta} | } = | h^2 \varphi_{\alpha\beta} | \qquad \textnormal{and} \qquad  |\tilde{ \alpha}| | \tilde{ \beta} | = |\alpha | | \beta | = \frac{1}{2} | \Psi |^2. \] 
Since $h^2 \varphi_{\alpha \beta}$ is meromorphic and $|\Psi|$ extends to a continuous function on $\Sigma$, it follows that $\tilde{\alpha}$ and $\tilde{\beta}$ extend to meromorphic sections. Let $D_1$ and $D_2$ be the associated divisors of zeroes and poles. We have $D = D_1 \cup D_2$ as sets and the condition $D = | \Psi |^{-1}(0)$ implies that $D_1 + D_2 \geq 0$.  The existence of the short exact sequence involving $\tilde{\alpha}$ and $\tilde{\beta}$ was established in \eqref{eqn:varphi-exact}.
\end{proof}

The next lemma provides a restriction on the possible holomorphic bundles $\cE$ fitting into the short exact sequence \eqref{eqn:fueterexact}.

\begin{lem} \label{lem:nofueter1}
Under the assumptions of Proposition \ref{prop:meromorphic} there exists a holomorphic line bundle $M$ satisfying  $h^0( M^2 ) > 0$ and $h^0( \cE \otimes K^{1/2} \otimes M^{-1} ) > 0$.  
\end{lem}
\begin{proof}
Recall that by Lemma \ref{lem:varphi} we have $\mathcal{L}^2 = \cO(D_1 - D_2)$. Set $M = \mathcal{L}^{-1} \otimes \cO(D_1)$. Then 
\[ M^2 = \mathcal{L}^{-2} \otimes \cO(2D_1) = \cO(D_2 - D_1 + 2D_1) = \cO(D_1 + D_2). \]
We have $h^0(M^2) > 0$  because the divisor $D_1 + D_2$ is effective. On the other hand, multiplying exact sequence \eqref{eqn:fueterexact} by $M^{-1} \otimes K^{1/2}$, we obtain an injective map $\cO \to \cE \otimes M^{-1} \otimes K^{1/2}$, which is the same as a nowhere vanishing section of  $ \cE \otimes K^{1/2} \otimes M^{-1}$.
\end{proof}

\subsection{Proof of Theorem \ref{thm:nofueter}}

\begin{lem} \label{lem:nofueter2}
Fix $k \geq 0$. Let $\cZ \subset \cA(\Sigma,E)$ be the subset consisting of those connections $B$ for which there exists a degree $k$ holomorphic line bundle $M \to \Sigma$ satisfying
\[ h^0(M^2) > 0 \quad \textnormal{and} \quad h^0( \cE_B \otimes K^{1/2} \otimes M^{-1}) > 0. \]
The complement $\cA(\Sigma,E) \setminus \cZ$ is residual. Furthermore, for all $B_0,B_1 \in \cA(\Sigma,E) \setminus \cZ$, a generic path in $\cA(\Sigma,E)$ connecting $B_0$ and $B_1$ is disjoint from $\cZ$.
\end{lem}

\begin{proof}
We use a transversality argument similar to the one used to show Proposition \ref{prop:transversality}. As in that case we pass to suitable Sobolev completions of the spaces of connections and sections (for simplicity we keep the same notation). The statement for $C^{\infty}$ topology will follow from Taubes' trick discussed in the proof of Proposition \ref{prop:transversality}.

Let $T \to \Sigma$ be a unitary line bundle of degree $k$. Denote $F = E^* \otimes K^{1/2} \otimes T^{-1}$ and consider 
\[ \cA(\Sigma,E) \times \cA(\Sigma,T) \times \Gamma(F) \times \Gamma(T^2) \to \Omega^{0,1} (F) \times \Omega^{0,1}( T^2 ), \]
\[ (B, A, \psi, \alpha) \mapsto (\del_{AB} \psi, \del_{A} \alpha). \] 
This map is $\cG^c(\Sigma)$--equivariant. Let  $\mathcal{X}$ be the open subset of $\cA(\Sigma,T)\times\Gamma(F)\times\Gamma(T^2)/ \cG^c(\Sigma)$ given by $\{ [B,A,\psi, \alpha] \ | \ \psi \neq 0, \ \alpha \neq 0 \}$. Let $\mathcal{V} \to \cA(\Sigma, E) \times \mathcal{X}$ be the Banach vector bundle obtained from taking the $\cG^c(\Sigma)$--quotient of the trivial bundle with fibre $ \Omega^{0,1} (F) \times \Omega^{0,1}( T )$. Then the map introduced above descends to a smooth section $s  \colon \cA(\Sigma, E) \times \mathcal{X} \to \mathcal{V}$. For every $B \in \cA(\Sigma,E)$ the restriction $s_B = s(B, \cdot)$ is a Fredholm section whose index is the Euler characteristic of the elliptic complex
\begin{equation} \label{eqn:fuetercmplx}
\begin{tikzcd}
\Omega^0(\C) \ar{r} & \Omega^{0,1}(\C) \oplus \Gamma(F) \oplus \Gamma(T^2) \ar{r} & \Omega^{0,1}(F) \oplus \Omega^{0,1}(T^2).
\end{tikzcd}
\end{equation}
The first arrow in the complex is the linearised action of $\cG^c(\Sigma)$, whereas the second is the linearisation of the map $(A, \psi, \alpha) \mapsto (\del_{AB} \psi, \del_A \alpha)$. This elliptic complex agrees up to terms of order zero with the direct sum of the complexes
\[ 
\begin{tikzcd}
\Omega^0(\C) \ar{r}{\del} & \Omega^{0,1}(\C) \ar{r} & 0 
\end{tikzcd}
\qquad \textnormal{and}
\]
\[ 
\begin{tikzcd}
0 \arrow{r} & \Gamma(F) \oplus \Gamma(T^2) \arrow{rr}{\del_{AB} \oplus \del_A} && \Omega^{0,1}(F) \oplus \Omega^{0,1}(T^2).
\end{tikzcd}
\]
By the Riemann--Roch theorem, the Euler characteristic of this complex is
\[ \chi( \cO ) - \chi(F) - \chi(T^2) = (1-g) - (\deg(F)+2-2g) - (2 \deg(T) + 1-g) = 0 \]
because $\deg(F) = 2g-2 - 2\deg(T)$.  Thus, $s_B$ is a Fredholm section of index zero.

The proof will be completed if we can show that $s$ is transverse to the zero section at all points $[B,A,\psi,\alpha] \in s^{-1}(0) \subset \mathcal{X}$. Indeed, if this is the case, then by the Sard--Smale theorem, the same is true for $s_B$ for $B$ from a residual subset of $\cA(\Sigma,E)$. For every such $B$ the set 
\[ \{ [A, \psi, \alpha] \ | \ \del_{AB} \psi = 0, \ \del_A \alpha = 0, \ \psi \neq 0 \ \alpha \neq 0 \} \]
is a zero-dimensional submanifold of $\mathcal{X}$. This submanifold must be empty as otherwise it would contain a subset homeomorphic to $\C^*$ given by $[A, t \psi, \alpha]$ for $t \in \C^*$. This proves that for a generic $B$ there is no holomorphic line bundle $M = (T, \del_A)$ together with non-zero $\alpha \in H^0(M^2)$ and $\psi \in H^0( \cE_B \otimes K^{1/2} \otimes M^{-1})$. The statement for paths is proved in the same way.
%for a generic path $( B_t )_{t \in [0,1]}$ the sunion $\bigcup_{t\in[0,1]} s_{B_t}^{-1}(0)$ is a one-dimensional submanifold of $\mathcal{X}$; it must be empty because otherwise it would contain a subset homeomorphic to $\C^*$. 

It remains to show that $s$ is transverse to the zero section. At a point $[B,A,\psi,\alpha] \in s^{-1}(0)$ the first map in  \eqref{eqn:fuetercmplx} is injective. Thus, it is enough to prove the surjectivity of the operator combining the second map of \eqref{eqn:fuetercmplx} and the linearisation of $\del_{AB}$ with respect to $B$: 
\[ \Omega^{0,1}(\mathrm{End}(F)) \oplus \Omega^{0,1}(\C) \oplus \Gamma(F) \oplus \Gamma(T^2) \longrightarrow \Omega^{0,1}(F) \oplus \Omega^{0,1}(T^2) \]
\[ (b,a,u,v) \mapsto ( (b+a)\psi + \del_{AB} u, a \alpha + \del_A v). \]
If the map were not surjective, there would exist a non-zero $(p,q) \in \Omega^{0,1}(F) \oplus \Omega^{0,1}(T^2)$ $L^2$--orthogonal to the image; which in turn would imply $\del_{AB}^* p = 0$, $\del_A^* q = 0$, and
\[ \langle b \psi , p \rangle_{L^2} = 0, \qquad \langle a \alpha, q \rangle_{L^2} = 0 \]
for all $b \in \Omega^{0,1}(\mathrm{End}(F))$ and $a \in \Omega^{0,1}(\C)$. Note that $\psi$ and $\alpha$ are both non-zero and holomorphic;  $p$ and $q$ are anti-holomorphic and at least one of them is non-zero. Using a bump function as in the proof of Proposition \ref{prop:transversality} it is easy to construct $b$ and $a$ such that
\[ \langle b \psi , p \rangle_{L^2} + \langle a \alpha, q \rangle_{L^2} > 0. \qedhere\]
\end{proof}

Theorem \ref{thm:nofueter} follows immediately from the previous results.
Let $B\in\cA(\Sigma,E)$ and denote by $\cE_B$  the corresponding holomorphic bundle. Propositions \ref{prop:fueterinvariance} and \ref{prop:meromorphic} show that a Fueter sections with respect to $(g,B)$ corresponds to a holomorphic triple $(\mathcal{L},\alpha, \beta)$ fitting into the short exact sequence \eqref{eqn:fueterexact}. On the other hand, by Lemmas \ref{lem:nofueter1} and \ref{lem:nofueter2}, for a generic choice of $B$ the holomorphic bundle $\cE_B$ does not fit into any such sequence. The same is true when $B$ varies in a generic one-parameter family by the second part of Lemma \ref{lem:nofueter2}. \qed

\subsection{Fueter sections and limiting configurations}

%We have already seen a class of solutions $(A,\alpha,\beta, D)$ to \eqref{eqn:fueterinvariant}: the limiting configurations introduced in Definition \ref{defn:limiting}. We will see that not every Fueter section is is a limiting configuration.
% However, the correspondence between limiting configurations and holomorphic data described in Proposition \ref{prop:homeofueter} extends to more general configurations satisfying \eqref{eqn:fueterinvariant}.

Every limiting configurations, as in Definition \ref{defn:limiting}, is an example of a Fueter section on $Y=S^1\times\Sigma$. The results established in this section allow us to construct a counterexample to the converse statement.

\begin{exmp}\label{exmp:fueterlimiting}
Suppose that the genus of $\Sigma$ is positive so that the canonical divisor $K$ is effective. Let $C_1$ and $C_2$ be two divisors satisfying $C_1 + C_2 = \frac{1}{2}K \geq 0$. Set 
\[ \mathcal{L} := \cO(C_1 - C_2), \qquad D_1 := 2C_1, \qquad D_2 := 2C_2, \]
Then 
\[ \mathcal{L}^{-1} \otimes K^{-1/2} \otimes \cO(D_1) = \cO(C_1 + C_2 - \frac{1}{2}K) = \cO, \]
\[ \mathcal{L}^{-1} \otimes K^{1/2} \otimes \cO(-D_2) = \cO(-C_1 - C_2 + \frac{1}{2}K) = \cO. \]
Set  $\cE := \cO \oplus \cO$. A Fueter section is given by maps $\tilde{\alpha}$ and $\tilde{\beta}$ making the sequence \eqref{eqn:fueterexact} exact. In the present setting,   \eqref{eqn:fueterexact} is equivalent to
\[ \begin{tikzcd}
0 \ar{r} & \cO \ar{r}{\tilde{\alpha}} & \cO \oplus \cO \ar{r}{\tilde{\beta}} & \cO \ar{r} & 0,
\end{tikzcd} \]
so there is an obvious choice of $\tilde{\alpha}$ and $\tilde{\beta}$ making the sequence exact. According to Proposition \ref{prop:meromorphic}, this gives rise to a Fueter section with singular set $D = D_1 \cup D_2$. However, such a Fueter section is not a limiting configuration unless both divisors $C_1$ and $C_2$ are effective. 
\end{exmp}

\section{Moduli spaces of framed vortices} 
\label{sec:vortices}
By Theorem \ref{thm:nofueter}, for a generic choice of a circle-invariant parameter, $\cM$ is homeomorphic to the compact space $\cN$ introduced in Definition \ref{defn:vortexmoduli}. In this section we prove that $\cN$ is a K\"ahler manifold and that the signed count of Seiberg--Witten multi-monopoles on $Y=S^1\times\Sigma$ is the signed Euler characteristic of $\cN$, which proves Theorem \ref{thm:generic0}. We then establish some general properties of $\cN$ using methods of complex geometry.

%This involves comparing $\cM$ and $\cN$ not only as topological spaces but as moduli spaces equipped with deformation theories. We end the section by describing the moduli spaces when the genus of $\Sigma$ is zero, one, and two.

\label{section:vortices}
\subsection{Framed vortices}
We continue to assume throughout this section that $E$ is an $\SU(2)$--bundle and that $d - \tau < 0$ where $d = \deg L$ and $\tau = \int_{\Sigma} i\eta / 2\pi$ \footnote{Most of the results generalise easily to the other cases. We will later discuss the role of the sign of $d - \tau$.}.

 $\cN$ depends on the conformal structure on $\Sigma$, holomorphic structure $\cE_B = (E^*, \del_B)$, and $d$. Its points can be interpreted in three ways. 

\begin{enumerate}
\item As isomorphism classes of pairs $(\mathcal{L}, \alpha)$, where $\mathcal{L} \to \Sigma$ is a degree $d$ holomorphic line bundle and $\alpha$ is a non-zero holomorphic section of $\cE_B \otimes \mathcal{L} \otimes K^{1/2}$. 
\item As $\cG^c(\Sigma)$--equivalence classes of pairs 
\[ (A, \alpha)  \in \cA(\Sigma,L) \times \Gamma(\Sigma, E^* \otimes L \otimes K^{1/2}) \]
satisfying $\del_{AB} \alpha = 0$ and $\alpha \neq 0$.
\item As $\cG(\Sigma)$--equivalence classes of pairs $(A, \alpha)$ as above satisfying
\begin{equation} \label{eqn:framed}
\left\{
\begin{array}{c}
\del_{AB} \alpha = 0 \quad \textnormal{and } \alpha \neq 0, \\
i \ast F_A + | \alpha |^2 - i \ast \eta = 0,
\end{array}
\right.
\end{equation}
\end{enumerate}
Following \cite{garcia-prada}, we will refer to $\cN$ as the \emph{moduli space of framed vortices}. 
%However, we employ the second, holomorphic description to endow $\cN$ with a complex analytic structure. 

\subsection{Deformation theory}
Here we relate the deformation theories of $\cN$ and $\cM_{\hol}$. 
\begin{thm} \label{thm:framed}
For every conformal class of a metric $g$ on $\Sigma$ there exists a residual subset $\cA^{reg}(g) \subset \cA(\Sigma, E)$ such that for every $B \in \cA^{reg}(g)$
\begin{enumerate}
\item $\cN = \cN(g,\cE_B)$ is a compact K\"ahler manifold of complex dimension $g(\Sigma)-1+2d$,
\item $\cM_{\hol} = \cM_{\hol}(g,\cE_B)$ is Zariski smooth and the inclusion $\cN \hookrightarrow \cM_{\hol}$ is a homeomorphism inducing an isomorphism of Zariski tangent spaces at every point,
\item the relative orientation on the obstruction bundle $\Ob \to \cM_{\hol}$ is compatible with the orientation of the cotangent bundle $T^* \cN \to \cN$ induced from the complex structure; equivalently, for every connected component $C$ of $\cM_{\hol}$ we have $\sign(C) = (-1)^{g(\Sigma)-1}$. 
\end{enumerate}
\end{thm}

\begin{cor} \label{cor:regular}
The following conditions are equivalent:
\begin{enumerate}
\item $\cN$ is regular as the moduli space of framed vortices.
\item $\cM_{\hol}$ is compact and equal to $\cN$.
\item There exist no triple $(\mathcal{L}, \alpha, \beta)$ consisting of a degree $d$ holomorphic line bundle $\mathcal{L} \to \Sigma$ and non-zero holomorphic sections $\alpha \in H^0( \cE_B \otimes \mathcal{L}\otimes K^{1/2})$ and $\beta \in H^0( \cE_B^* \otimes \mathcal{L}^* \otimes K^{1/2})$ satisfying $\alpha \beta = 0 \in H^0(K)$.
\end{enumerate}
\end{cor}

\begin{proof}
The equivalence of $(1)$ and $(2)$  follows from Corollary \ref{cor:subspace} and the identification of the obstruction bundle $\bigcup_{A,\alpha} H^2_{A,\alpha}$ with $\overline{\cM}_{\hol} \setminus \cM_{\hol}$, shown in the proof of Theorem \ref{thm:framed}. The equivalence of $(2)$ and $(3)$ is obvious from the definition of $\overline{\cM}_{\hol}$. 
\end{proof}

%Recall that $\cN(g, \cE_B)$ can be defined purely in terms of the complex geometry $\Sigma$. We will see in Section \ref{sec:examples} that the biholomorphism type of $\cN(g, \cE_B)$ can change when we vary $(g,\cE_B)$. Nevertheless, Theorem \ref{thm:sigmacount} shows that the Euler characteristic  $\cN(g,\cE_B)$ does not depend on $g$ and the connection $B \in \cA(\Sigma,E)$, as long as $B$ is generic. In fact, the proof of Theorem \ref{thm:sigmacount} shows that for two different generic choices $(g_0,B_0)$ and $(g_1,B_1)$ the moduli spaces $\cN(g_0, \cE_{B_0})$ and $\cN(g_1, \cE_{B_1})$ are fibres of a locally trivial fibration over $[0,1]$, and hence they are diffeomorphic.

\begin{cor} \label{cor:invariance}
 The diffeomorphism type of $\cN(g, \cE_B)$ does not depend on the metric on $\Sigma$ and the connection $B \in \cA(\Sigma,E)$, as long as $B$ is generic. 
\end{cor}
\begin{proof}
Let $g_0$, $g_1$ be metrics on $\Sigma$ and $B_0 \in \cA(g_0)$, $B_1 \in \cA(g_1)$ as in Theorem \ref{thm:nofueter}. For a generich path $(B_t)$ in $\cA(\Sigma,E)$ connecting $B_0$ and $B_1$ there exist no Fueter sections with respect to $(g_t,B_t)$ and so $\cM(g_t,B_t)$ is compact for all $t\in[0,1]$. By Corollary \ref{cor:regular}, $\cN(g_t,\cE_{B_t}) = \cM(g_t,B_t)$ is compact and regular as the moduli space of framed vortices. Thus, $\bigcup_{t\in[0,1]}\cN(g_t,\cE_{B_t}) \to [0,1]$ is a smooth fibre bundle and every fibre $\cN(g_t,\cE_{B_t})$ is diffeomorphic to $\cN(g_0,\cE_{B_0})$.
\end{proof}

The construction of an analytic structure on $\cN$ follows the general scheme that by now is familiar to the reader. 
Consider the elliptic complex associated to a solution $(A,\alpha)$ of $\del_{AB} \alpha = 0$:
\[
\begin{tikzcd}
\Omega^0(\C) \ar{r}{G^c_{A,\alpha}} & \Omega^{0,1}(\C) \oplus \Gamma(E^* \otimes L \otimes K^{1/2}) \ar{r}{T_{A, \alpha}} & \Omega^{0,1}(E^* \otimes L \otimes K^{1/2}).
\end{tikzcd}
\]
where $G^c_{A,\alpha}$ is the linearised action of $\cG^c(\Sigma)$ 
\[ G^c_{A,\alpha}(f) = (- \del f, f \alpha) \qquad \textnormal{for } f \in \Omega^0(\C), \]
and $T_{A,\alpha}$ is the linearisation of the Dolbeault operator
\[ T_{A,\alpha}( a^{0,1}, \phi) = ( \del_{AB} \phi + a^{0,1} \alpha ) \qquad \textnormal{for } (a^{0,1}, \phi) \in  \Omega^{0,1}(\C) \oplus \Gamma(E^* \otimes L \otimes K^{1/2}). \]
Denote by $H^0_{A, \alpha}$, $H^1_{A,\alpha}$, and $H^2_{A, \alpha}$ the homology groups of this complex. By definition $\cN$ consists of solutions with $\alpha \neq 0$, so $H^0_{A,\alpha} = 0$. On the other hand, the deformation complex is isomorphic modulo lower order term to the sum of the Dolbeault complexes for $\del$ on $\Omega^0(\C)$ and $\del_{AB}$ on $E^* \otimes L \otimes K^{1/2}$ (with a shift). By the Riemann--Roch theorem the expected complex dimension of $\cN$ is
\[  \dim_{\C} H^1_{A, \alpha} -  \dim_{\C} H^2_{A,\alpha} =  \chi(\Sigma, \cO) - \chi(\Sigma, \cE \otimes \mathcal{L} \otimes K^{1/2} ) = g(\Sigma) -1 + 2d. \]

\begin{proof}[Proof of Theorem \ref{thm:framed}]
The proof proceeds in three steps.
\setcounter{step}{0}
\begin{step}
$\cN$ is a compact K\"ahler manifold.
\end{step}
We already know that $\cN$ is compact. By Corollary \ref{cor:nofueter}, $\cM_{\hol}  = \cN$ for a generic $B$. 
One can show that $\cN$ is generically smooth in the same way as in Proposition \ref{prop:regular}. 
%After showing that the universal moduli space is a smooth Banach manifold we would conclude from the Sard-Smale theorem that $\cN$ is generically smooth and of expected dimension. 
Alternatively, we can interpret the elements of $H^2_{A,\alpha}$ as Fueter sections:
\[ H^2_{A,\alpha} = \ker T_{A,\alpha}^* = \left\{ q \in \Gamma(E^* \otimes L \otimes K^{-1/2}) \ | \ \overline{\alpha} q = 0, \ \del_{AB}^* q = 0 \right\}. \]
Every non-zero element of $H^2_{A,\alpha}$ gives rise to a non-zero $\beta = \overline{q} \in \Gamma(E \otimes L^* \otimes K^{1/2})$ satisfying $\del_{AB} \beta = 0$ and $\alpha \beta = 0$. Thus, the triple $(A,\alpha,\beta)$ is an element of $\overline{\cM}_{\hol} \setminus \cM_{\hol}$ corresponding to a Fueter section as in Proposition \ref{prop:homeofueter}.  By Theorem \ref{thm:nofueter}, for a generic $B$ there are no Fueter sections so $H^2_{A,\alpha} = 0$ for all $[A,\alpha] \in \cN$. This implies that $\cN$ is a complex manifold of dimension $g(\Sigma) - 1 + 2d$ whose holomorphic tangent space at $[A,\alpha]$ is $H^1_{A,\alpha}$. It admits a natural Hermitian metric induced from the $L^2$--inner product on the space of connections and sections. This metric is K\"ahler because $\cN$ is the moduli space of solutions of the framed vortex equations \eqref{eqn:framed}, which is an infinite-dimensional K\"ahler quotient. For details, see \cite{perutz, dey-thakre}.

\begin{step}
$H^1_{A,\alpha}$ is naturally isomorphic to the Zariski tangent space to $\cM_{\hol}$ at $[A,\alpha,0]$.
\end{step}
$H^1_{A,\alpha}$ consists of pairs $(a^{0,1},u) \in \Omega^{0,1}(\C) \oplus \Gamma(E^* \otimes L \otimes K^{1/2})$ satisfying the linearised equation
\[ 
\del_{AB} u + a^{0,1} \alpha = 0
\]
together with the complex Coulomb gauge  $(G^c_{A,\alpha})^*(a^{0,1},u) = 0$. By \eqref{eqn:deformation2}, the tangent space to $\cM_{\hol}$ at consists of triples $(a^{0,1}, u, v)$ where $a^{0,1}$, $u$ are as above, $v \in \Gamma(E \otimes L^* \otimes K^{1/2})$, and  
\[\left\{
\begin{array}{l}
\del_{AB} u + a^{0,1} \alpha = 0,\\
\del_{AB} v = 0, \\
\alpha v = 0
\end{array}
\right. \]
together with the complex Coulomb gauge for $(a^{0,1}, u,v)$. Any non-zero $v$ satisfying the conditions above would give an element $(A, \alpha, v)$ of $\overline{\cM}_{\hol} \setminus \cM_{\hol}$. Since $B$ has been chosen so that  $\overline{\cM}_{\hol} \setminus \cM_{\hol}$ is empty,  $v = 0$ and the equations obeyed by $(a^{0,1}, u, 0)$ are identical to the ones defining $H^1_{A, \alpha}$. We conclude that the Zariski tangent spaces to $\cN$ and $\cM_{\hol}$ are equal.  

\begin{step}
Comparing the orientations.
\end{step}
Let $(A, \Psi)$ be an irreducible solution of the Seiberg--Witten equations. We have $\Psi = (\alpha, 0)$ where $(A,\alpha)$ is a solution of the framed vortex equations. Consider the extended Hessian operator introduced in subsection \ref{subsec:moduli}:
\[ L_{A,\Psi} \colon \Omega^1(i \R) \oplus \Omega^0(i \R) \oplus \Gamma(E^* \otimes S \otimes L ) \longrightarrow \Omega^1(i \R) \oplus \Omega^0(i \R) \oplus \Gamma(E^* \otimes S \otimes L ) \]
Write $L_{A,\Psi} = L_{A,0} + P$, where 
\[ L_{A,0} = 
\left(
\begin{array}{ccc}
\ast d & -d & 0 \\
- d^* & 0 & 0 \\
0 & 0 & \D_{AB} 
\end{array}
\right) \]
and 
\[ P(a,v, \phi) = ( i \Im \langle \Psi, \phi \rangle , - 2 \ast \mu(\phi, \Psi) , - a \cdot \Psi + v \Psi ). \]
The kernel and cokernel of $L_{A,0}$ are naturally identified with 
\[ H^1( Y, i \R) \oplus H^0 ( Y, i \R) \oplus \ker \D_{AB}. \]
The isomorphism between $\det L_{A,0}$ and  $\det L_{A,\Psi}$, defining the relative orientation on the obstruction bundle, factors through the determinant space $\det P$ of the finite dimensional map
\[ P \colon H^1( Y, i \R) \oplus H^0 ( Y, i \R) \oplus \ker \D_{AB} \to H^1( Y, i \R) \oplus H^0 ( Y, i \R) \oplus \ker \D_{AB} \]
induced from the zeroth order operator $P$ defined above (for simplicity we use the same letter to denote the induced finite dimensional map).  As in the proof of Proposition \ref{prop:invariance}, we have  $H^1(Y, i \R) = H^1( S^1 , i \R ) \oplus H^{0,1}(\Sigma)$. Consider the complex structure on $H^1(S^1, i\R) \oplus H^0(Y, i\R)$ coming from the identification
\[ H^1(S^1, i\R) \oplus H^0(Y, i\R) = i\R \oplus i \R = \C. \]
Let $dt$ be the one-form spanning $H^1(S^1, \R)$. Under the Clifford multiplication, $dt$ acts as the multiplication by $i$ on $S$, and so $i dt$ acts as the multiplication by $-1$. Hene, under the isomorphism $H^1(S^1, i\R) \oplus H^0(Y, i\R) = \C$, the map $(a,v) \mapsto (- a \cdot \Psi + v \Psi)$ is given by $(x+iy) \mapsto (x+iy) \Psi$  and so, in particular, it is complex linear. Next, consider the first two components of $P$, that is the map $\phi \mapsto (i \Im(\Psi, \phi), -2\ast \mu(\phi, \Psi))$.  Decompose the moment map into $\mu = \mu_{\R} \oplus \mu_{\C}$ as in the proof of Proposition \ref{prop:invariance}. The map $\phi \mapsto \mu_{\C}(\phi, \Psi)$ is complex linear from $\ker \D_{AB}$ to $H^{0,1}(\Sigma)$. We are left with the map from $\ker \D_{AB}$ to $H^1(S^1, i \R) \oplus H^0(Y, i\R)$ given by
\[ \phi \mapsto (i \Im(\Psi, \phi) - 2 \ast \mu_{\R}(\phi, \Psi)). \]
We have $\phi = (u,v)$ under the splitting $S = K^{1/2} \oplus K^{-1/2}$. Following the identifications from the proof of Proposition \ref{prop:invariance} we find that $\ast \mu_{\R}(\phi, \Psi) = - 2 i \Re ( \alpha, u)$ and so our map is
 \[ \phi = (u,v) \mapsto (i  \Im ( \alpha, u) , - 4 i \Re (\alpha, u ). \]
Up to a constant, it coincides with the complex linear map 
\[ u \mapsto - \Re(\alpha,u) + i\Im(\alpha, u) = - \overline{(\alpha, u)} = - (u, \alpha). \]
We conclude that the isomorphism $\det P \cong \det L_{A,0}$ agrees with the orientations induced from the complex structures on the cohomology groups. The same is true for $\det P \cong\det L_{A,\Psi}$ where the complex structures on $H^1_{A,\Psi} = \ker L_{A,\Psi}$ and $H^2_{A,\Psi} = \mathrm{coker} L_{A,\Psi}$ come from the isomorphism of analytic spaces $\cM \cong \cM_{\hol}$ given by Theorem \ref{thm:holomorphic}. The tangent and obstruction spaces to $\cM_{\hol}$ are canonically identified with the tangent space to $\cN$. Therefore, the relative orientation on the obstruction bundle agrees on the complex orientation on $T^* \cN \to \cN$. 
\end{proof}

\subsection{Dependence on the perturbing two-form} \label{subsec:sign}
We have so far ignored the fact that for fixed $d = \deg L$ there are two definitions of $\cM_{\hol}$ depending on the sign of $d - \tau$ \footnote{The case $d- \tau =0$ is uninteresting as the moduli space is generically empty, see Corollary \ref{cor:moduliempty}.}, see Definition \ref{defn:holomorphic}. Recall that $\tau := \int_{\Sigma} \frac{i \eta}{2\pi}$ depends on the choice of the perturbing two-form $\eta$. In classical Seiberg--Witten theory, the moduli space of solutions on $S^1 \times \Sigma$ is either $\Sym^{d+g-1} \Sigma$ or $\Sym^{-d+g-1} \Sigma$, depending on the sign of $d - \tau$. On the other hand, the Seiberg--Witten invariant does not depend on the choice of the perturbing two-form, which is reflected by the identity
\[ \chi(\Sym^{d+g-1} \Sigma) = \chi(\Sym^{-d+g-1} \Sigma). \]
 
Theorem \ref{thm:generic0} gives us an analogous identity for the moduli spaces of Seiberg--Witten multi-monopoles. Fix $(g,B)$, with $B$ generic, and denote by $\cM_{\hol}^+(d)$ and $\cM_{\hol}^-(d)$ the moduli spaces corresponding to a given degree $d$ and two choices of the sign of $d - \tau$. (In the previous subsections we have always assumed $\cM_{\hol} = \cM_{\hol}^+$.) The map $(\mathcal{L}, \alpha, \beta) \mapsto (\mathcal{L}^*, \beta, \alpha)$ induces an isomorphism $\cM_{\hol}^+(d) \longrightarrow \cM_{\hol}^-(-d)$. This is in agreement with the standard involution in Seiberg--Witten theory \cite[Section 6.8]{morgan}. 

If $g(\Sigma) \geq 1$, the count of Seiberg--Witten multi-monopoles $(-1)^{g(\Sigma)-1}\chi(\cM_{\hol}^+(d))$ does not depend on the choice of a generic choice of $(B,\eta)$ as long as $d - \tau < 0$. The same is true for $(-1)^{g(\Sigma)-1}\chi(\cM_{\hol}^-(d))$  when $d - \tau > 0$. On the other hand, the choice of $\eta$ is immaterial from the viewpoint of the three-dimensional theory---by Theorem \ref{thm:moduli}, $\SW(g,B,\eta)$ does not depend on the choice of  $\eta$ as long as the moduli space is compact and Zariski smooth. We conclude that for a generic choice of $B \in \cA(\Sigma,E)$ we have $\chi( \cM_{\hol}^+(d) ) = \chi( \cM_{\hol}^-(d))$. Combining this with the isomorphism $\cM_{\hol}^+(d) \longrightarrow \cM_{\hol}^-(-d)$, we obtain

\begin{cor}  \label{cor:sign}
If $g(\Sigma) \geq 1$, then for a generic choice of $B \in \cA(\Sigma,E)$ we have 
\[  \chi( \cM_{\hol}^+(d) ) = \chi( \cM_{\hol}^+(-d)). \]
\end{cor}

Although the moduli spaces can be defined in terms of the complex geometry of $\Sigma$, it is far from obvious how to prove the above equality without a reference to the three-dimensional theory. We will see in the next section that $\cM_{\hol}^+(d)$ and $\cM_{\hol}^-(d)$ can be non-homeomorphic.

\section{Examples and computations}
\label{sec:examples}
%In this and the following subsection we study $\cM_{\hol}$ using methods of complex geometry. Fix the complex structure on $\Sigma$ and spin structure $K^{1/2}$.
%We describe the topology of $\cM_{\hol}$ when $\Sigma$ is a Riemann surface of genus zero, one, and two.
In this section we study $\cM_{\hol}$ using methods of complex geometry. We prove some general properties of the moduli spaces and give their complete description when $\Sigma$ is a  Riemann surface of genus zero,  one, or two.  

\begin{thm}
\label{thm:moduliexamples}
Let $g$ be a product metric on $Y = S^1 \times \Sigma$,  $B$ a generic connection pulled-back from $\Sigma$, and $\eta \in \Omega^2(\Sigma, i\R)$ a two-form satisfying $\tau(\eta) > 0$. 

Set $\cM=\cM(g,B,\eta)$ and $\SW = \SW(g,B, \eta)$. 
\begin{enumerate}
\item If $d < (1-g(\Sigma))/2$, then $\cM$ is empty and $\SW = 0$.
\item If $d \geq 0$, then $\cM$ admits a holomorphic map to the Jacobian torus of $\Sigma$. Its fibres are projective spaces. If $d > 0$, this map is surjective. If $d = 0$, its image is a divisor in the linear system $|2 \Theta|$ where $\Theta$ is the theta divisor.
\item If $d \geq g(\Sigma)-1$, then $\cM$ is biholomorphic to the projectivisation of a rank $2d$ holomorphic vector bundle over the Jacobian of $\Sigma$ and $\SW = 0$.
\item If $d = 0$ and $g(\Sigma) = 1$, then $\cM$ consists of two points and $\SW= 2$.
\item If $d = 0$ and $g(\Sigma) = 2$, then $\cM$ is biholomorphic to a closed Riemann surface of genus five and $\SW = 8$.
\end{enumerate}
\end{thm}

\begin{defn}
We will denote by $\cM_{\hol}(d, \cE)$ the moduli space defined using a degree $d$ line bundle $L$,   holomorphic structure $\cE$ on $E^*$, and any perturbing two-form $\eta$ satisfying $d - \eta < 0$; similarly we define $\cN(d, \cE)$. A property will be said to hold for a generic $\cE$ if it holds for all $\cE$ of the form $\cE_B = (E, \del_B)$ for $B$ from a residual subset of $\cA(\Sigma,E)$.
\end{defn}

\subsection{Generalised theta divisors} \label{subsec:theta}

For $d=0$ the moduli spaces of framed vortices are related to \emph{generalised theta divisors}  \cite{beauville}. Let $\cE \to \Sigma$ be a rank two holomorphic stable bundle with trivial determinant. For any line bundle $A \in J^{g(\Sigma)-1}$ the Riemann--Roch theorem gives us
\[ \chi( \cE \otimes A) = 2 \deg(A) + 2(1-g(\Sigma)) = 0, \]
so we expect $H^0(\cE \otimes A)= H^1( \cE \otimes A) = 0$ if $A$ is generic. 
\begin{defn}
The \emph{generalised theta divisor of} $\cE$ is 
\[ \theta(\cE) := \{ A \in J^{g(\Sigma)-1} \ | \ h^0(\cE \otimes A) > 0 \}; \]
\end{defn}
One can show that $\theta(\cE)$ is a divisor\footnote{This is no longer true if $\cE$ is of higher rank as it can happen that $\theta(\cE)=J^{g(\Sigma)-1}$.} in $J^{g(\Sigma)-1}$  in the linear system $| 2 \Theta | = \mathbb{CP}^{2g(\Sigma)-1}$, where $\Theta$ is the classical theta divisor
\[ \Theta := \{ A \in J^{g(\Sigma)-1} \ | \ h^0(A) > 0 \}. \]

%The relevance of $\theta(\cE)$ to the study of framed vortices is as follows.

\begin{prop}
If $\cE$ is a rank two holomorphic stable bundle with trivial determinant, then there is a surjective morphism $\cN(0, \cE) \to \theta(\cE)$ whose fibres are projective spaces.
\end{prop}

\begin{proof}
A point in $\cN(0, \cE)$ is an equivalence class $[\mathcal{L}, \alpha] $ where $\mathcal{L} \in J^0$ and $\alpha \in H^0( \cE \otimes \mathcal{L} \otimes K^{1/2})$, with $\alpha \neq 0$. Since $\det(\mathcal{L} \otimes K^{1/2}) = g(\Sigma)-1$, we have $\mathcal{L} \otimes K^{1/2} \in \theta(\cE)$. The morphism $\cN(0, \cE) \to \theta(\cE)$ is given by $[ \mathcal{L}, \alpha ] \mapsto \mathcal{L} \otimes K^{1/2}$. The preimage of $\mathcal{L} \otimes K^{1/2}$ is $\mathbb{P} H^0( \cE \otimes \mathcal{L} \otimes K^{1/2})$.
\end{proof} 

\begin{rem}
The divisor $\theta(\cE)$ can be described geometrically as follows.  By a theorem of Lefschetz, the linear system $| 2 \Theta |$ is base-point free and gives rise to a holomorphic map $J^{g(\Sigma)-1}\to | 2 \Theta |^*$. It follows that $\theta(\cE)$, as a subset of $J^{g(\Sigma)-1}$, is the preimage of a hyperplane in $| 2 \Theta |^*$ under the map $J^{g(\Sigma)-1} \to |2 \Theta|^*$. This hyperplane is easy to identify---it is exactly $\theta(\cE)$ thought of as a point in $| 2 \Theta |$ or, equivalently, as a hyperplane in $| 2 \Theta |^*$. Varying the background bundle $\cE$, we vary the corresponding hyperplane and therefore the divisor $\theta(\cE) \subset J^{g(\Sigma)-1}$. 
\end{rem}

\subsection{General properties of the moduli spaces}
%In this subsection we describe some properties of $\cM_{\hol}$ for a general Riemann surface $\Sigma$. 

\begin{prop} \label{prop:general}
For a generic choice of $\cE$ the following holds. 
\begin{enumerate}
\item If $d < \frac{1-g(\Sigma)}{2}$, then $\cM_{\hol}(d, \cE)$ is empty.
\item If $d \geq 0$, then $\cM_{\hol}(d, \cE)$ is non-empty. 
\item If $d \geq g(\Sigma)-1$, then $\cM_{\hol}(d, \cE)$ is Zariski smooth with the underying complex manifold biholomorphic to the projectivisation of a rank $2d$ vector bundle over $J^d$.
\end{enumerate}
\end{prop}

\begin{rem}
Proposition \ref{prop:general} shows that the most interesting case is $(1-g(\Sigma))/2 \leq d < 0$. It is an interesting question whether $\cM_{\hol}(d, \cE)$ is generically non-empty for $d$ in this range.
\end{rem}

The proof of Proposition \ref{prop:general} relies on the following general result about holomorphic vector bundles on Riemann surfaces. Recall that $\cE$  \emph{stable} if for any holomorphic line bundle $A$ the existence of a non-zero holomorphic map $A \to \cE$ implies $\deg(A) < 0$. 
%The next lemma shows that when the genus of $\Sigma$ is at least two, a generic connection on $E$ induces a stable holomorphic structure. The main idea is that the locus of unstable bundles can be described as the zero set of a Fredholm map of non-positive index.

\begin{lem} \label{lem:stable}
If $g(\Sigma) \geq 2$, then there is an open dense subset of $\cA(\Sigma, E)$ such that for every connection $B$ from this subset the corresponding holomorphic bundle $\cE_B$ is stable. 
\end{lem}

\begin{proof}
$\cE$ fails to be stable if and only if there is a holomorphic line bundle $\mathcal{L}$ with $\deg ( \mathcal{L}) = d \geq 0$ and a non-zero map $\theta \colon \mathcal{L} \to \cE_B$. In other words, if $L$ is a unitary bundle underlying $\mathcal{L}$ and $A$ is a connection inducing $\mathcal{L}$, then $\theta \in \Gamma(L^* \otimes E)$ satisfies $\del_{AB} \theta = 0$.
Consider 
\[
 \cU_d := \{ B\in \cA(\Sigma,E) \ | \  \ker\del_{AB} = \{ 0 \} \textnormal{ for all } A \in \cA(\Sigma,L) \}
\]
for a fixed degree $d$ unitary bundle $L$.

\setcounter{step}{0}
\begin{step}
$\cU_d$ is open.
\end{step}
Let $B\in\cU_d$. For every $A\in\cA(\Sigma,L)$ there is a neighbourhood of $(A,B)$ in $\cA(\Sigma,L) \times \cA(\Sigma, E)$ such that for all $(A',B')$ from this neighbourhood $\ker \del_{A'B'} = 0$. Since this condition is invariant under the action of $\cG^c(\Sigma)$, and $A(\Sigma, L) / \cG^c(\Sigma)$ is homeomorphic to a torus---in particular, compact---there is a neighbourhood of $B$ in $\cA(\Sigma, E)$ such that for all $B'$ from this neighbourhood $\ker \del_{AB'} = 0$ for all $A$. All such $B'$ belong to $\cU_d$ which proves that $\cU_d$ is open.

\begin{step}
$\cU_d$ is dense.
\end{step}
We only outline the proof as it is similar to that of Proposition \ref{prop:transversality}. In what follows we replace the spaces of connections and sections by their Sobolev completions. Cover $\cA(\Sigma, L) / \cG^c(\Sigma)$ by finitely many charts that can be lifted to $\cG^c(\Sigma)$--slices in $\cA(\Sigma,L)$. Let $V$ be such a chart; it is an open subset of $\R^{2g}$ parametrising a smooth family of connections $\{ A_x \}_{x \in U}$. The claim will follow if can show that 
\[
  \cU_V := \{ B\in\cA(\Sigma,E) \ | \ \ker\del_{A_x B} = \{ 0 \} \textnormal{ for all } x\in V \}
\]
is dense in $\cA(\Sigma,E)$. To prove this, consider 
\[ S := \left\{ \theta \in \Gamma(\Sigma, L^* \otimes E) \ | \ \| \theta \|_{L^2} = 1 \right\} \]
and the map
\[ f \colon \cA(\Sigma, E) \times V \times S \longrightarrow \Omega^{0,1}(\Sigma, L^* \otimes E) \]
\[ f(B,x,\theta) := \del_{A_x B} \theta. \]
For every $B \in \cA(\Sigma,E)$ the restriction $f_B \coloneqq f(B, \cdot, \cdot)$ is a Fredholm map (between suitable Sobolev spaces) because its derivative is the sum of $\del_{A_x B}$ and the derivative with respect to $x$, which is a finite-dimensional operator. By the Riemann--Roch theorem,
\begin{equation}\label{eqn:index}
 \begin{split}
\mathrm{ind}_{\R} \ df_B &= \dim V + 2 \mathrm{ind} \ \del_{A_x B} - 1
\\
&= 2g + 4(-d+1-g(\Sigma)) -1 \\
&= 2(-2d + 2 - g(\Sigma)) - 1 \leq 0,
\end{split}
\end{equation}
where we subtract $1$ because $\del_{A_x B}$ is restricted to the tangent space to $S$. A computation similar to that in the proof of Proposition \ref{prop:transversality} shows that the derivative of the full map $f$ is surjective at every point of $f^{-1}(0)$. By the Sard--Smale theorem, the set $f_B^{-1}(0)$ is empty for $B$ from a dense subset of $\cA(\Sigma,E)$; all such $B$ belong to $\cU_V$.

\begin{step}
$\cU:=\bigcap_{d\geq 0}\cU_d$ is open and dense.
\end{step}

$\cU$ is dense by Baire's theorem; it is open by the following argument. By \eqref{eqn:bound}, the existence of a destabilising map $\theta \colon \mathcal{L} \to \cE_B$ implies 
\[ 0 \leq d \leq h^0(\Sigma, \cE_B) + 2g(\Sigma)-2. \]
The right-hand side can only decrease when $B$ is replaced by a sufficiently close $B'$. (Indeed, if we split $\Gamma(\Sigma,E)$ into $\ker\del_B$ and its $L^2$--orthogonal complement $Q$, then by the elliptic estimate  $\del_{B'}$ is non-degenerate when restricted to $Q$ for all nearby connections $B'$; it follows that the projection $\ker \del_{B'} \to \ker \del_B$ is injective. See also 
 \cite[Proposition 11.21]{mukai} for an algebro-geometric proof.) Therefore, to guarantee that a nearby connection $B'$ belongs to $\cU$ it is enough to check that it belongs to $\cU_d$ for finitely many values of $d$. Thus, for every $B \in \cU$ there are finitely many open neighbourhoods of $B$ whose intersection lies entirely in $\cU$.
\end{proof} 

\begin{proof}[Proof of Proposition \ref{prop:general}]\leavevmode
By Theorem \ref{thm:framed}, for a generic choice of $\cE$ the moduli space $\cM_{\hol} = \cN$ is a compact complex manifold of dimension $g(\Sigma) - 1 + 2d$. If $d < (1-g(\Sigma))/2$, then this dimension is negative and $\cM_{\hol}$ must be empty.

 The case $d = 0$ was discussed in the previous subsection. Let $\deg(\mathcal{L}) = d > 0$; then 
\[ h^0( \cE \otimes \mathcal{L} \otimes K^{1/2} ) - h^1( \cE \otimes \mathcal{L} \otimes K^{1/2}) = 2d > 0; \]
thus,  $\mathcal{L}$ is in the image of the projection $\pi \colon \cM_{\hol} \to J^d$ given by $[\mathcal{L}, \alpha, \beta]  \mapsto \mathcal{L}$. For a generic $\cE$, by Theorem \ref{thm:framed},  $\cM_{\hol} = \cN$ and so $\pi^{-1}(\mathcal{L}) = \P H^0( \cE \otimes \mathcal{L})$.

For the proof of the third item, assume $g(\Sigma) \geq 2$; the cases $g(\Sigma) = 0, 1$ will be considered separately in the next section. By Lemma \ref{lem:stable}, a generic $\cE$ is stable and by Serre duality,
\[ h^1( ( \cE \otimes \mathcal{L} \otimes K^{1/2}  ) = h^0 (( \cE^* \otimes \mathcal{L}^* \otimes K^{1/2} ). \]
Any element of $H^0 (( \cE^* \otimes \mathcal{L}^* \otimes K^{1/2} )$ gives a holomorphic map $\mathcal{L} \otimes K^{-1/2} \to \cE^*$. We have
\[ \deg(\mathcal{L} \otimes K^{-1/2}) = d - g + 1 \geq 0. \]
Since $\cE^*$ is stable, it follows that any holomorphic map $\mathcal{L} \otimes K^{-1/2} \to \cE^*$ is trivial. Thus, $h^1( ( \cE \otimes \mathcal{L} \otimes K^{1/2}  ) = 0$ and by the Riemann--Roch theorem $h^0( \cE \otimes \mathcal{L} \otimes K^{1/2}  ) = 2d$ for every $\mathcal{L} \in J^d$. We conclude that $\pi \colon \cM_{\hol} \to J^d$ is the projectivisation of the rank $2d$ vector bundle whose fibre over $\mathcal{L}$ is the cohomology group $H^0( \cE \otimes \mathcal{L} \otimes K^{1/2}  ) = \C^{2d}$. 
\end{proof}

\subsection{Genus zero} 
Let $\Sigma = \mathbb{CP}^1$ with its unique complex structure. For $k\in \mathbb{Z}$ denote by $\cO(k)$ the unique holomorphic line bundle of degree $k$; $K^{1/2} = \cO(-1)$ is the unique spin structure.  By a theorem of Grothendieck, every holomorphic bundle over $\mathbb{CP}^1$ is the direct sum of line bundles. In particular, every holomorphic $\SL(2,\C)$--bundle is of the form $\cE = \cO(k) \oplus \cO(-k)$ for some $k \geq 0$, with $k = 0$ being the generic case.
\begin{prop}
\label{prop_genuszero}
Let $\Sigma = \mathbb{CP}^1$ and $\cE = \cO(k) \oplus\cO(-k)$ for $k \geq 0$. 
\begin{enumerate}
\item If $d \leq 0$ and $k \leq |d|$, then $\cM_{\hol}(d, \cE)$ is empty.
\item If $d > 0$ and $k \leq d$, then $\cM_{\hol}(d, \cE)$ is Zariski smooth with the underlying complex manifold biholomorphic to $\mathbb{CP}^{2d-1}$.
\item If $k > |d|$, then  $\cM_{\hol}(d, \cE)$ is non-compact and its compactification $\overline{\cM_{\hol}}(d, \cE)$ is homeomorphic to a locally trivial $\mathbb{CP}^{k-d}$--fibration over $\mathbb{CP}^{k+d}$. 
\end{enumerate}
\end{prop}

\begin{rem}
The fact that the Euler characteristic of the moduli space depends on the sign of $d$ is consistent with the three-dimensional theory. Since $b_1( S^1 \times \mathbb{CP}^1) = 1$  we do not expect the count of Seiberg--Witten multi-monopoles to be invariant, even when there are no Fueter sections; by Proposition \ref{prop:reducibles}, there are two chambers in the set of parameters which are separated by a codimension one wall where reducibles appear. As discussed in subsection \ref{subsec:sign}, replacing $d$ by $-d$ can be  seen as varying $\tau$ so that the sign of $d - \tau$ changes; any path of parameters joining such two choices of $\tau$ will pass through the wall of reducibles.
\end{rem}

\begin{proof}
$\cM_{\hol}(d, \cE)$ consists of the equivalences classes of pairs $(\alpha, \beta)$ such that
\[ \alpha \in H^0( \cO(k+d-1) ) \oplus H^0(\cO(-k + d-1)), \]
\[ \beta \in H^0( \cO(-k-d-1) ) \oplus H^0( \cO(k-d-1) ), \] 
$\alpha \neq 0$, and $\alpha \beta = 0 \in H^0( \cO(-2))$---this is automatically satisfied since $h^0(\cO(-2)) = 0$.
If $\cE$ is generic, so that $k = 0$, then $d \leq 0$ which implies that $\alpha = 0$ and $\cM_{\hol}(d,\cE)$ is empty. If $d > 0$, then $\alpha \in \mathbb{C}^{2d}$ and $\beta = 0$; it follows that $\cM_{\hol}(d, \cE) = \cN(d, \cE) = \mathbb{CP}^{2d-1}$.  

The same description of $\cM_{\hol}(d, \cE)$ is valid in the non-generic case $k \neq 0$ as long as $k \leq |d|$. When $k > |d|$ the moduli space is no longer compact and Fueter sections appear. If $k > d > 0$, then $\alpha \in \C^{k+d}$, $\beta \in \C^{k-d}$ and $\cM_{\hol}(d, \cE)$ is the total space of the vector bundle $\cO(-1)^{\oplus (k-d)}$ over $\mathbb{CP}^{k+d-1}$. The compactification  $\overline{\cM}_{\hol}(d, \cE)$ is the $\mathbb{CP}^{k-d}$-bundle over $\mathbb{CP}^{k+d-1}$ obtained from the projectivisation of the vector bundle $\cO(-1)^{\oplus (k-d)} \oplus \cO$. 
\end{proof}

\subsection{Genus one} Let $\Sigma = S^1 \times S^1$ equipped with a complex structure making it into an elliptic curve. Isomorphism classes of line bundles of a given degree $d$ on $\Sigma$ form the Jacobian $J^d$ which is isomorphic to the dual torus $\Sigma^*$. The canonical bundle of $\Sigma$ is trivial and without loss of generality we can take $K^{1/2}$ also to be trivial. Holomorphic vector bundles over elliptic curves were classified by Atiyah \cite{atiyah}. A generic holomorphic $SL(2,\C)$--bundle $\cE$ is of the form $\cE = A \oplus A^{-1}$ for a degree zero line bundle $A \to \Sigma$. We may moreover assume that $A^2 \neq \cO$ since there are only four line bundles satisfying $A^2 = \cO$. 

\begin{prop}
Let $\Sigma$ be an elliptic curve. Suppose that $\cE$ is generic, that is of the form $\cE = A \oplus A^{-1}$ for $A \in J^0$ satisfying $A^2 \neq \cO$. 
\begin{enumerate}
\item If $d < 0$, then $\cM_{\hol}(d, \cE)$ is empty.
 \item If $d > 0$, then $\cM_{\hol}(d, \cE)$ is Zariski smooth with the underlying complex manifold biholomorphic to the projectivisation of a rank $2d$ vector bundle over $J^d$.
  \item If $d = 0$, then $\cM_{\hol}(d, \cE)$ is regular and consists of two points. 
  \end{enumerate}
\end{prop}

\begin{rem}
If $d >0$, then the cohomology ring $H^*( \cM_{\hol}(d, \cE), \Z)$ is isomorphic as $H^*( J^d, \Z)$--modules to $H^*( J^d, \Z) [ H ] / (H^{2d})$ where $\deg(H) =2$. In particular, 
\[ \chi( \cM_{\hol}(d, \cE) ) = \chi( J^d ) \chi( \mathbb{CP}^{2d-1}) = 0,  \]
which is consistent with the fact that $\chi( \cM_{\hol}(d, \cE) )$ is invariant under the change $d \mapsto -d$. 
\end{rem}

\begin{proof}
$\cM_{\hol}(d, \cE)$ consists of equivalence classes of triples $(\mathcal{L}, \alpha, \beta)$ where $\mathcal{L} \in J^d$ and
\[ \alpha \in H^0( A \otimes \mathcal{L}) \oplus H^0(A^{-1} \otimes \mathcal{L}), \]\[ \beta \in H^0(A^{-1} \otimes \mathcal{L}^{-1}) \oplus H^0( A \otimes \mathcal{L}^{-1}), \]
satisfying $\alpha \neq 0$ and $\alpha \beta = 0$ in $H^0( \cO) = \C$. 

For $d < 0$ we must have $\alpha = 0$ and so the moduli space is empty. For $d = 0$ the only choices of $\mathcal{L}$ for which $\alpha$ is possibly non-zero are $\mathcal{L} = A^{-1}$ and $\mathcal{L} = A$. If $\mathcal{L} = A^{-1}$, then
\[ \alpha \in H^0( \cO) \oplus H^0( A^{-2}), \]
\[ \beta \in H^0( \cO) \oplus H^0( A^2). \]
Since $A^2$ and $A^{-2}$ are non-trivial, so they have no non-zero sections. The only possibly choice for $\alpha$, up to scaling, is therefore $\alpha = (1,0)$ and the condition $\alpha \beta = 0$ forces $\beta$ to be zero since the pairing $H^0( \cO ) \times H^0( \cO) \to H^0( \cO)$ is simply the multiplication $\C \times \C \to \C$. We repeat the same argument for $\mathcal{L} = A$ and conlude that $\cM_{\hol}(d, \cE)$ consists of two isolated points. In particular it is compact and so Zariski smooth thanks to Corollary \ref{cor:regular}. Since it is also has the correct dimension zero, we conclude that each of the points is regular.

Consider now the case $d > 0$. For every $\mathcal{L} \in J^d$ the Riemann--Roch theorem gives us
\[ h^0( \mathcal{L} \otimes A ) - h^1( \mathcal{L} \otimes A) = d \]
and by Serre duality, $h^1( \mathcal{L} \otimes A ) = h^0( \mathcal{L}^{-1} \otimes A) = 0$ because $\deg(\mathcal{L}^{-1})<0$, Thus $H^0 (\mathcal{L} \otimes A) = \C^{d}$ and the same is true if $A$ is replaced by $A^{-1}$. Therefore, $\alpha$ is identified with a non-zero vector in $\C^{2d}$. On the other hand, $\beta = 0$ since $A^{\pm 1} \otimes \mathcal{L}^{-1}$ has no non-trivial sections. Since the above discussion is valid for any $\mathcal{L} \in J^d$, it follows that  $\cM_{\hol}(d, \cE)$ is a locally trivial $\mathbb{CP}^{2d-1}$--fibration over $J^d$: the projectivisation of a  rank $2d$ holomorphic vector bundle over $J^d$ given by the push-forward of the Poincar\'e line bundle  $\cP \to J^d \times \Sigma$ to the first factor.
\end{proof}

It is worthwhile discussing some non-generic examples. The cases when $\cE = A \oplus A^{-1}$ and either $\deg(A) \neq 0$ or $A^2 = \cO$ are similar to the ones already considered. Another possibility is that $\cE$ is indecomposable, in which case it is of the form $\cE = \cE_0 \otimes A$ where $A \in J^0$ satisfies $A^2 = \cO$ and $\cE_0$ is the unique non-trivial bundle obtained as an extension
\[ \begin{tikzcd}
0 \ar{r} & \cO \ar{r} & \cE_0 \ar{r} & \cO \ar{r} & 0. 
\end{tikzcd} \] 
The line bundle $A$ is uniquely determined by $\cE$. 

\begin{exmp}
\label{exmp_nongeneric1}
Suppose without loss of generality that  $\cE = \cE_0$. It is shown in \cite{atiyah} that if $h^0( \cE \otimes \mathcal{L} ) \neq 0$, then either $\deg(\mathcal{L})>0$ or $\mathcal{L} = \cO$ in which case we have $h^0( \cE) = 1$. We conclude that when $d < 0$ or $d > 0$ the moduli space $\cM_{\hol}(d, \cE)$ is, respectively, empty or the projectivisation of a vector bundle over $J^d$. On the other hand, for $d = 0$ the only choice of $\mathcal{L}$ for which $h^0(\cE \otimes \mathcal{L}) > 0$ is $ \mathcal{L} = \cO$ and we look for holomorphic sections
\[ \alpha \in H^0( \cE ) = \C, \]
\[ \beta \in H^0( \cE^* ) = \C \]
such that $\alpha \neq 0$ and $\alpha \beta = 0$. Up to scaling,  $\alpha = 1$. We will show now that  the pairing $H^0( \cE ) \times H^0( \cE^* ) \to \C$ is trivial and, as a consequence, $\beta$ can be chosen arbitrarily.  Let $\Omega \in H^0( \Lambda^2 \cE)$ be a nowhere vanishing holomorphic volume form. It induces an isomorphism $\cE \to \cE^*$ given by $v \mapsto \Omega(v, \cdot)$. If $\alpha$ is a generator of $H^0( \cE)$, then $\gamma = \Omega(\alpha, \cdot)$ is a non-zero holomorphic section of $H^0( \cE^*)$ and so it must be a generator. On the other hand, $\gamma(\alpha) = \Omega(\alpha, \alpha) = 0$ since $\Omega$ is skew-symmetric---this shows that $\alpha \beta = 0$ for every $\beta \in H^0( \cE^*)$. Therefore, $\cM_{\hol}(0, \cE)$ is homeomorphic to $\C$. Its compactification $\overline{\cM}_{\hol}(0, \cE)$ is homeomorphic to $\mathbb{CP}^1$.
\end{exmp}

\subsection{Genus two} Let $\Sigma$ be a genus two Riemann surface. By Lemma \ref{lem:stable}, a generic holomorphic bundle on $\Sigma$ is stable. The proof of the next lemma can be found in \cite{narasimhan}. 

\begin{lem} \label{lem:genus2}
Let $W$ be a stable rank two vector bundle with trivial determinant over a genus two closed Riemann surface $\Sigma$. If $\xi \to \Sigma$ is a degree $1$ line bundle, then
\begin{enumerate}
\item $h^0(W \otimes \xi) \leq 1$.
\item Any non-zero homomorphism $\xi^* \to W$ is everywhere injective.
\end{enumerate}
\end{lem}

\begin{prop}
Let $\Sigma$ be a closed Riemann surface of genus two. For a generic holomorphic $\SL(2,\C)$--bundle $\cE \to \Sigma$ we have the following description of $\cM_{\hol}(d, \cE)$.
\begin{enumerate}
\item If $d < 0$, then $\cM_{\hol}(d, \cE)$ is empty.
\item If $d = 0$, then $\cM_{\hol}(d, \cE)$ is Zariski smooth with the underlying complex manifold biholomorphic to a closed Riemann surface of genus five.
\item If $d > 0$, then $\cM_{\hol}(d, \cE)$ is Zariski smooth with the underlying complex manifold biholomorphic to the projectivisation of a rank $2d$ vector bundle over $J^d$.
\end{enumerate}
\end{prop}

\begin{proof}
Items $(2)$ and $(3)$ follow from Proposition \ref{prop:general} For $d = 0$ we use the relation between the moduli space of framed vortices and theta divisors described in subsection \ref{subsec:theta}. Let $\cE \to \Sigma$ be a stable $\SL(2,\C)$--bundle. Let $\mathcal{SU}(2)$ be the compactification of the moduli space of such bundles obtained by adding the $S$--equivalence classes of semi-stable bundles. As explained in \cite{narasimhan}, we have $| 2 \Theta | = \mathbb{CP}^3$ and the map introduced in subsection \ref{subsec:theta}
\[ \theta \colon \mathcal{SU}(2) \to \mathbb{CP}^3 \]
\[ \cE \mapsto \theta(\cE) \]
 is an isomorphism. Recall that $\theta(\cE)$ can be seen either as a subset of the Jacobian $J^1$ 
 \[ \theta(\cE) = \{ A \in J^1 \ | \ h^0( \cE \otimes A) > 0 \} \]
 or as the corresponding point in $| 2 \Theta |$. 
 
 $J^1$ is a two-dimensional complex torus and the map $J^1 \to |2 \Theta|^* = (\mathbb{CP}^3 )^*$ induces a degree four embedding of the Kummer surface $J^1 / \Z_2$. Thus, as a subset $\theta(\cE) \subset J^1$ is the preimage of the intersection of the Kummer surface in $(\mathbb{CP}^3 )^* $ with the hyperplane $\theta(\cE) \in \mathbb{CP}^3$ under the quotient map $J^1 \to J^1 / \Z^2$. Since $\theta \colon \mathcal{SU}(2) \to \mathbb{CP}^3$ is an isomorphism, by changing the background bundle $\cE$ we can obtain all hyperplanes in $(\mathbb{CP}^3)^*$. In particular, for a generic choice of $\cE$, the hyperplane $\theta(\cE)$ will avoid  all the $16$ singular points of $J^1 / \Z_2$ and the intersection $J^1/ \Z_2 \cap \theta(\cE)$ will be a smooth complex curve of degree four and genus three. Its preimage under $J^1 \to J^1 / \Z_2$ is a smooth curve $C\subset J^1$ of genus five, by the Hurwitz formula.
 
Let $\pi \colon \cM_{\hol}( 0 , \cE) \to C$ be the composition of $\cM_{\hol}(0, \cE) \to \cN(0, \cE)$ with the projection $\cN(0, \cE) \to \theta(\cE) = C$ from subsection \ref{subsec:theta}.  We claim that this map is an isomorphism for a generic choice of $\cE$. In order to prove that, it is enough to check that the fibre over any line bundle $\mathcal{L} \otimes K^{1/2}$ in $C$ consists of one point. This is equivalent to showing that  $H^0( \cE \otimes \mathcal{L} \otimes K^{1/2})$ is spanned by a single non-zero section $\alpha$ and if $\beta \in H^0( \cE^* \otimes \mathcal{L}^* \otimes K^{1/2})$ is any section satisfying $\alpha \beta = 0$, then $\beta = 0$.  The first claim follows immediately from Lemma \ref{lem:genus2}. As regards the second claim, assume that $\alpha$ and $\beta$ are as above and  $\beta \neq 0$. By Lemma \ref{lem:genus2}, the homomorphisms
\[ \alpha \colon \mathcal{L}^* \otimes K^{-1/2} \longrightarrow \cE \quad \textnormal{and} \quad \beta \colon \mathcal{L} \otimes K^{-1/2} \longrightarrow \cE^* \]
are everywhere injective. Since $\rank\cE=2$, $\alpha \beta = 0$ implies the exactness of the sequence
\begin{equation} \label{eqn:extension}
\begin{tikzcd}
0 \arrow{r} & \mathcal{L}^* \otimes K^{-1/2} \ar{r}{\alpha} & \cE \ar{r}{\beta^t} & \mathcal{L}^* \otimes K^{1/2} \ar{r} & 0.  
\end{tikzcd}
\end{equation}
Since $\det\cE = \cO$, we conclude $\mathcal{L}^2 = \cO$. There are  $16$ line bundles satisfying this condition: order two elements of $J^0$. For each of them, all possible non-trivial extensions $\cE$ as above are classified by the corresponding extension class in $\mathbb{P}H^1(K^{-1}) = \mathbb{CP}^2$. (Note that the extension is non-trivial because $\cE$ is stable.) Thus, all stable bundles $\cE$ that can be represented as such an extension form a proper subvariety of $\mathcal{SU}(2) = \mathbb{CP}^3$ consisting of the images of $16$ maps $\mathbb{CP}^2 \to \mathcal{SU}(2)$. A generic stable bundle $\cE$  will not belong to this subvariety. In this case, we conclude that each fibre of the map $\pi$ consists of a single point and $\pi$ is an isomorphism. In particular, $\cM_{\hol}(0, \cE)$ is compact and therefore Zariski smooth by Corollary \ref{cor:regular}.
\end{proof}

\begin{rem}
\label{rem_nongeneric}
Note that the last part of the proof was unnecessary; we already know that generically $\cM_{\hol} = \cN$ is compact and Zariski smooth which is enough to conclude $\cN = C$. On the other hand, the argument presented above identifies the locus of those semi-stable bundles $\cE \in \mathcal{SU}(2)$ for which Fueter sections appear. It consists of strictly semi-stable bundles $A \oplus A^{-1}$, for some $A \in J^0$, which form the Kummer surface $J^1 / \Z_2$ in $\mathcal{SU}(2) = \mathbb{CP}^3$, and stable bundles $\cE$ that arise from an extension of the form \eqref{eqn:extension} for some element $\mathcal{L} \in J^0$ of order two. The latter form a subvariety covered by the images of $16$ maps $\mathbb{CP}^2 \to \mathbb{CP}^3$. 
\end{rem}

\newpage
\bibliographystyle{alphanum_n.bst}
\bibliography{nsw.bib} 
\end{document}